\documentclass[12pt,reqno]{amsart}

\usepackage[normalem]{ulem}
\usepackage{cancel}

\usepackage{amsfonts,amsmath,amsthm,mathrsfs,mathtools,esint,bbm}
\usepackage{amssymb,epsfig}
\usepackage{enumerate,comment} 

\usepackage{color} 
\definecolor{vert}{rgb}{0,0.6,0}

\usepackage[text={425pt,650pt},centering]{geometry}
\usepackage{graphicx}
\usepackage{epsfig}
\usepackage{tikz,esvect}
\usetikzlibrary{3d,calc,shapes}
\usepackage{pgfplots}
\pgfplotsset{compat=1.18}

\usepackage{caption}
\usepackage{color} 
\definecolor{vert}{rgb}{0,0.6,0}

\usepackage[pdfstartview=FitH, bookmarksnumbered=true,bookmarksopen=true, colorlinks=true, pdfborder={0 0 1}, citecolor=blue, linkcolor=blue,urlcolor=blue]{hyperref}

\numberwithin{figure}{section}

\theoremstyle{plain}
\newtheorem{thm}{Theorem}[section]

\newtheorem{defn}{Definition}

\newtheorem{conj}{Conjecture}

\newtheorem{lem}[thm]{Lemma}
\newtheorem{cor}[thm]{Corollary}
\newtheorem{prop}[thm]{Proposition}
\theoremstyle{remark}
\newtheorem{rem}{\bf{Remark}}
\numberwithin{equation}{section}

\newcommand{\cone}{%
  \begin{tikzpicture}[baseline=0pt, scale=0.23]
\draw (0.3,0) -- (0.7,0) -- (1,1) -- (0,1) -- cycle;
  \end{tikzpicture}%
}

\newcommand{\E}{\mathbb{E}}

\newcommand{\N}{\mathbb{N}}
\newcommand{\bP}{\mathbb{P}}
\newcommand{\R}{\mathbb{R}}

\newcommand{\Z}{\mathbb{Z}}
\newcommand{\dsc}{\text{\tiny{dis}}}

\newcommand{\cB}{\mathcal{B}}

\newcommand{\frakC}{\mathfrak{C}}

\newcommand{\AC}{{\rm AC\,}}

\newcommand{\BUC}{{\rm BUC\,}}

\newcommand{\Lip}{{\rm Lip\,}}

\newcommand{\gam}{\gamma}

\newcommand{\ep}{\varepsilon}

\newcommand{\ol}{\overline}

\newcommand{\beq}{\begin{equation}}
\newcommand{\eeq}{\end{equation}}

\newcommand{\bbN}{{\mathbb{N}}}
\newcommand{\bbR}{{\mathbb{R}}}

\newcommand{\bbP}{{\mathbb{P}}}
\newcommand{\bbE}{{\mathbb{E}}}

\newcommand{\eps}{\varepsilon}

\newcommand{\lb}{\label}

\def\XXint#1#2#3{{\setbox0=\hbox{$#1{#2#3}{\int}$ }
\vcenter{\hbox{$#2#3$ }}\kern-.6\wd0}}
\def\avint{\fint} 

\newcommand{\dd}{\mathrm{d}}

\newcommand*{\rom}[1]{\text{\expandafter\@slowromancap\romannumeral #1@}}

\DeclarePairedDelimiter{\abs}{\lvert}{\rvert}
\DeclarePairedDelimiter{\norm}{\lVert}{\rVert}
\providecommand{\Abs}[1]{\Bigr\lvert#1\Bigl\rvert}
\providecommand{\floor}[1]{\lfloor#1\rfloor}
\providecommand{\ceil}[1]{\lceil#1\rceil}
\providecommand{\mb}[1]{\mathbb{#1}}
\providecommand{\mc}[1]{\mathcal{#1}}
\providecommand{\ms}[1]{\mathscr{#1}}

\begin{document}

\title[Quantification of ergodicity for HJ equations]
{Quantification of ergodicity for Hamilton--Jacobi equations in a dynamic random environment}

\author[X. Guo, W. Jing, H. V. Tran, Y. P. Zhang]{Xiaoqin Guo, Wenjia Jing, Hung Vinh Tran, Yuming Paul Zhang}

\thanks{
X. Guo is supported by Simons Foundation through Collaboration Grant for Mathematicians \#852943. 
W. Jing is partially supported by the NSF of China under grant No.\,12571220 and by the New Cornerstone Investigator Program 100001127.
H. V. Tran is partially supported by NSF grant DMS-2348305. 
Y. P. Zhang is partially supported by NSF CAREER grant DMS-2440215 and Simons Foundation Travel Support MPS-TSM-00007305.
}

\address[X. Guo]
{
Department of Mathematical Sciences, 
University of Cincinnati,
2815 Commons Way, Cincinnati OH 45221-0025}
\email{guoxq@ucmail.uc.edu}

\address[W. Jing]
{
Yau Mathematical Sciences Center, Tsinghua University, No.1 Tsinghua Yuan, Beijing 100084, and Beijing Institute of Mathematical Sciences and Applications, Beijing 101408, China}
\email{wjjing@tsinghua.edu.cn}

\address[H. V. Tran]
{
Department of Mathematics, 
University of Wisconsin-Madison, Van Vleck Hall, 480 Lincoln Drive, Madison, WI 53706}
\email{hung@math.wisc.edu}

\address[Y. P. Zhang]
{
Department of Mathematics and Statistics, Auburn University, Parker Hall, 
221 Roosevelt Concourse, Auburn, AL 36849}
\email{yzhangpaul@auburn.edu}

\keywords{first-order Hamilton--Jacobi equations; tensionless KPZ equation; large-time averages; quantitative homogenization; viscosity solutions}
\subjclass[2010]{
35B27
35B40 
35F21 
49L12
49L25 
}

\begin{abstract}
We study quantitative large-time averages for Hamilton--Jacobi equations in a dynamic random environment that is stationary ergodic and has unit-range dependence in time. Our motivation comes from stochastic growth models related to the tensionless (inviscid) KPZ equation, which can be formulated as Hamilton--Jacobi equations with random forcing. Understanding the large-time behavior of solutions is closely connected to fundamental questions concerning fluctuations and scaling in such growth processes. In this article, we establish, up to slowly varying factors, convergence rates with exponent $1/2$ for the large-time averages of both the solutions and the associated metric problem toward their ergodic limits. Our proof relies crucially on a new almost-Lipschitz regularity theory for the metric problem, which is of independent interest.
\end{abstract}

\maketitle

\tableofcontents

\section{Introduction}

In this paper, we are interested in the following first-order stochastic Hamilton--Jacobi equation, which is motivated by a growth model corresponding to a tensionless (inviscid) KPZ equation
\begin{equation}\label{eq:HJ}
    \begin{cases}
        u_t + H(x,t,Du,\omega)=0 \qquad &\text{ in } \R^d \times (0,\infty),\\
        u(x,0)=0 \qquad &\text{ on } \R^d.
    \end{cases}
\end{equation}
Our main goal is to study the quantitative large-time averages for $u$, the solution of \eqref{eq:HJ}.

\subsection{Motivations}
We are interested in a growth model with a random forcing that is stationary ergodic in space and has unit-range dependence in time. 
A prototypical equation of interest is a tensionless KPZ equation
\begin{equation}\label{eq:KPZ}
u_t + K(Du) + f(x,t,\omega)=0.
\end{equation}
Here, $K:\R^d \to \R$ is convex and superlinear, and 
\[
f(x,t,\omega)=\sum_{i=1}^N f_i(x,\omega)\dot B_i(t),
\]
where $f_i$, $1\leq i \leq N$, are stationary ergodic in $x$, and $(B_i)_{i=1}^N$ is a standard Brownian motions on $\R^N$ (so that $\dot B_i$ denotes white noise in time) which is independent of $(f_i)_{i=1}^N$.
The unknown $u=u(x,t)$ denotes the height of the interface at location $x\in \R^d$ at given time $t\in [0,\infty)$.

\medskip

It is important to study the scaling limits of the solution 
$u$ and of the associated minimizing curves arising from the optimal control representation. 
A series of conjectures and open problems in this direction was formulated in \cite{BK}.

First, \cite{BK} conjectured the existence of a unique time-stationary (up to time-dependent additive constants) global solution of \eqref{eq:KPZ} with sublinear growth. Moreover, it was predicted that global solutions are preserved under the zero-viscosity limit. The uniqueness of such global solutions to \eqref{eq:KPZ} is closely related to the existence and uniqueness of one-sided minimizers (or geodesics).
See \cite[Conjectures 1-2]{BK}.

The next important question is to understand the fluctuation of $u(x,t)$ around its mean and its correct fluctuation exponent, which remains widely open in this tensionless case.
See \cite[Conjecture 3]{BK}, \cite{MR2096046, PhysRevE.106.024802} and the references therein.
In one dimension, \cite{PhysRevE.106.024802} asserted that its scaling behavior is intrinsically anomalous, with different roughness exponents controlling height fluctuations at local and global length scales.
At the global length scales, the scaling seems ballistic, which gives exponents very far from the standard one-dimensional viscous KPZ values.
These have not been verified rigorously yet.
In multi-dimensions, there is no agreed exponent in the literature except for the fact that it is positive.

\smallskip

The aforementioned questions are closely related to questions on quantitative homogenizations of Hamilton--Jacobi equations.
There is now an extensive literature on qualitative stochastic homogenization for Hamilton--Jacobi equations with smooth Hamiltonians, and we will only list a few representative works.
Qualitative homogenization for first-order equations with convex Hamiltonians was first established in 
\cite{rezakhanlou2000homogenization, souganidis1999stochastic}, and that for vanishing second-order (viscous) equation with convex Hamiltonian in stationary ergodic media was later established in \cite{lions2005homogenization, KRV}, while nonuniformly coercive Hamiltonians were treated in \cite{SH2015}.

The investigation of non-convex  Hamiltonians, more specifically non-quasi-convex ones, has revealed more complex behavior. 
Stochastic homogenization for Hamiltonians of separation form $H(p,x,\omega) = G(p)-V(x,\omega)$ was proved for certain classes in multiple space dimensions in \cite{armstrong2015stochastic,MR3856807}, and for general coercive Hamiltonians in one space dimension in \cite{armstrong2016stochastic}.
Variational solutions in the non-convex setting were settled in \cite{viterbo2025stochastic}. For space dimensions larger than one, it has been shown through a definitive counterexample in \cite{ziliotto2017stochastic,feldman2017homogenization} that homogenization can fail in stationary ergodic environments in the absence of convexity. In particular, this happens for Hamiltonians of separation form where $G$ has a strict saddle point. However, for one space dimension, stochastic homogenization for more general forms of Hamiltonians was proved in \cite{MR3466903} for first-order equations and in \cite{MR4920350,davini2024stochastichomogenizationnondegenerateviscous} for non-degenerate second-order equations.

Parallel to these qualitative developments, a quantitative theory has emerged. Algebraic rates of convergence for a general class of equations and geometric motions were obtained in \cite{armstrong2015,armstrong2014error}, under the assumption that the environment satisfies a finite range of dependence.
The obtained algebraic rates are not known to be optimal.

While those qualitative and quantitative stochastic homogenization results were all proved for random environments that oscillate only in space, this line of inquiry has been expanded to the setting of dynamically varying random environments. 
This setting is rather natural in the context of the growth models, including the tensionless KPZ equation.
Qualitative homogenization for space-time stationary ergodic super-linear Hamiltonians in the presence of diffusion was established in \cite{MR2400607}, and in \cite{MR3602941} for super-quadratic Hamiltonians with potentially degenerate diffusions. Stochastic homogenization in the first-order case without diffusion was addressed in \cite{MR2514380}.
For the differentiability of the effective Lagrangian, see \cite{BD25}.
Superlinear growth of the Hamiltonian in the momentum variable yields uniform H\"older regularity of solutions of first-order Hamilton-Jacobi equations; see \cite{MR2969493}. This plays an important role in stochastic homogenization in spatio-temporal random environments. 
For related works on Burgers' equation, we refer the reader to \cite{BCK14,Bakhtin16,Bakhtin-Li19}.
The homogenization of moving interface in a spatio-temporal random environment, where the Hamiltonian only grows linearly, was treated in \cite{MR3817561} under the assumption that the medium is periodic in one variable (time or space) and stationary-ergodic in the other. Another important model with a linearly growing Hamiltonian is the G-equation, which possesses a further essential difficulty: the Hamiltonian may not be coercive. The homogenization was proved in \cite{MR2744920, MR2763033} in the periodic setting, and in \cite{cardaliaguet2013homogenization} for the spatially stationary ergodic case for divergence-free drift fields. For space-time varying stationary divergence-free drift fields that are not too large on average over large scales, homogenization was established in \cite{burago2020feeble, zhang2023homogenization}.

Recent work has also pivoted toward pathwise equations involving very rough coefficients. We refer the reader to \cite{souganidis2019pathwise} for a review of pathwise solutions for fully nonlinear first-and second-order partial differential equations with rough time dependence. For first-order equations, the stochastic homogenization for pathwise Hamilton–Jacobi equations was established in \cite{seeger2018homogenization, seeger2021scaling}. Additionally, the homogenization of \eqref{eq:KPZ}, where the equation is forced by rapidly oscillating noise colored in space and white in time, was addressed in \cite{Seeger}. Finally, the long-time convergence of a class of second-order stochastic Hamilton–Jacobi equations was analyzed in \cite{gassiat2024long}.

\medskip

\subsection{Settings}
Let $(\Omega_0,\mathscr F_0)$ be a measurable space, and write $\Omega:=\Omega_0^{\R^{d+1}},\mathscr F:=\mathscr F_0^{\otimes\R^{d+1}}$. An element $\omega=(\omega_{x,t})_{x\in\R^d,\,t\in \R}\in\Omega$ is called an environment. 
Consider a random field $(\Omega,\mathscr F, \mb P)$ on $\R^d\times\R$ that satisfies the following conditions.

\begin{enumerate}[(i)]
  \item $\mb P$ is stationary ergodic under space-time shifts. That is, the shifts $\{T_{y,s}:(y,s)\in\R^d\times\R\}$ defined by  
  \[
  T_{y,s}\omega:=(\omega_{x+y,t+s})_{x\in\R^d,\,t\in \R}
  \] 
  are measure-preserving and commutative; and for $\mathscr U \in \mathscr F$,
  \[
  \text{ if $T_{y,s}(\mathscr U)=\mathscr U$ for all $(y,s)\in\R^d\times\R$, then either $\mb P(\mathscr U)=0$ or $\mb P(\mathscr U)=1$.}
  \]
   \item $\mb P$ has a unit-range of dependence in time, that is,  $\{\omega_{x,t}:(x,t)\in\R^d\times(-\infty,a]\}$ and $\{\omega_{x,t}:(x,t)\in\R^d\times[a+1,\infty)\}$ are independent (under law $\mb P$) for all $a\in \R$.
\end{enumerate}

\smallskip

The Hamiltonian $H:\R^d\times\Omega\to\R$ is a function whose randomness depends only on the environment at $(0,0)$, that is,  
\[
H(p,\omega)=H(p,\omega_{0,0}).
\]
We assume throughout the paper that
\begin{enumerate}
    \item[(A1)] $p \mapsto H(p,\omega)$ is convex for each $\omega\in \Omega$.
\end{enumerate}
The Lagrangian $L:\R^d\times\Omega\to\R$, the Legendre transform of $H$, is denoted by, for $(v,\omega)\in \R^d\times \Omega$,
\[
L(v,\omega)=\sup_{p\in \R^d} \left(p\cdot v - H(p,\omega) \right).
\]
We write, with abuse of notations, 
\[
H(x,t,p,\omega)=H_\omega(x,t,p):=H(p,\omega_{x,t}),\qquad
  L(x,t,v,\omega)=L_\omega(x,t,v):=L(v,\omega_{x,t}).
\]
We assume that, for $\mathbb P$-almost surely $\omega$, $L_\omega\in \BUC(\R^d\times \R\times B(0,R))$ for every $R>0$ and the following condition holds.

\begin{enumerate}
  \item[(A2)] There exist $q>1$, $N_1>1$, and $\nu:\Omega_0^{\R^{d+1}}\to[{1},\infty)$ such that the random variables $\nu_t(\omega):=\nu(\{\omega_{x,t}:x\in\R^d\})$
    satisfy
 \begin{equation}
  \label{eq:L-bd}
    \frac{1}{N_1}|v|^q-\nu_0\le L(v,\omega)\le N_1|v|^q+\nu_0 \quad \text{ for all }v\in\R^d;
    \end{equation}
  and for all $(x,t,v),(y,s,w)\in \R^d\times \R\times\R^d$,
\begin{equation}
  \label{eq:L-regularity}
  |L(x,t,v,\omega)-L(y,s,w,\omega)|
  \le C|v-w|\left(|v|^{q-1}+|w|^{q-1}+1\right)+\nu_{s}+\nu_{t}.
\end{equation}
\end{enumerate}
Notice that $L_\omega(x,t,v)=L_{T_{x,t}\omega}(0,0,v)$ and $\nu_t(\omega)=\nu_0(T_{x,t}\omega)$ for all $x,t$ and a.s.\ $\omega$.  
\medskip

We sometimes omit the explicit dependence of random variables on the environment $\omega$ when this does not cause any ambiguity.

\subsection{Main results}
For any $(x_i,t_i)\in \R^{d+1}$ with $i=1,2$ and $t_2>t_1$, let 
\begin{equation*}
\begin{aligned}
m(x_1,t_1;x_2,t_2,\omega)&=
\inf \Big\{\int_{t_1}^{t_2} L(\gam(s),s,\dot\gam(s),\omega)\,\dd s: \gam\in\AC([t_1,t_2],\R^d), \\
&\qquad\quad\qquad\quad\qquad\quad\qquad\quad\qquad  \gam(t_1)=x_1,\gam(t_2)=x_2 \Big\}.
\end{aligned}
\end{equation*}
See Figure \ref{fig:space-time curves}.
\begin{figure}[htp!]
\begin{center}
\begin{tikzpicture}[scale=2.5, line cap=round, line join=round]
  \draw[->] (0.6,0.6) -- (2.4,0.6) node[right] {$x$};
  \draw[->] (0.6,0.6) -- (0.6,3.4) node[above] {$t$};

  \fill (1,1) circle (0.03);
  \fill (2,3) circle (0.03);

  \draw[densely dotted] (1,1) rectangle (2,3);


  \draw[very thick, blue]
    (1,1) -- (2,3);

  \draw[very thick, red]
    plot[domain=0:1, samples=160]
    ({1+\x+0.45*\x*(1-\x)}, {1+2*\x});

  \draw[very thick, teal]
    plot[domain=0:1, samples=160]
    ({1+\x-0.45*\x*(1-\x)}, {1+2*\x});

  \draw[very thick, orange]
    plot[domain=0:1, samples=220]
    ({1+\x+0.22*sin(360*\x)}, {1+2*\x});

  \draw[very thick, brown]
    (1,1) .. controls (1.8,0.9) and (1.1,3.1) .. (2,3);

\end{tikzpicture}
\end{center}
\caption{Some admissible curves connecting $(x_1,t_1)$ to $(x_2,t_2)$}\label{fig:space-time curves}
\end{figure}
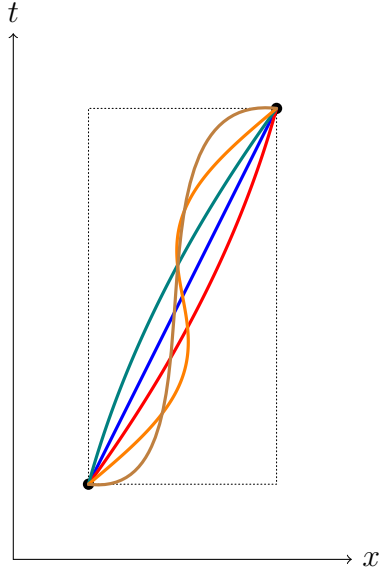

We define, for $(x,t)\in \R^{d+1}$,
\[
m(x,t;x,t,\omega)\equiv 0.
\]
For $(x,t)\in \R^d\times [0,\infty)$, we write
\[
m(x,t,\omega)=m(0,0;x,t,\omega)
\]
for simplicity.

Thanks to \cite{MR2400607,MR2514380,MR3602941,MR3817561, Seeger},
the effective Lagrangian is defined as, for $v\in \R^d$,
\begin{equation}
\label{eq:olLdef}
  \ol{L}(v) := \ol{m}(v,1) = \lim_{t\to \infty} \frac{\mathbb{E} m(tv,t)}{t} = \lim_{t\to \infty} \frac{m(tv,t,\omega)}{t},
\end{equation}
where the last inequality holds for almost surely $\omega \in \Omega$. The effective Hamiltonian is defined via the Legendre transform
\begin{equation}
  \label{eq:olHdef}
  \ol{H}(p) := \sup_{v\in \R^d} \left(p\cdot v- \ol{L}(v)\right), \qquad p\in \R^d.
\end{equation}
It is well known that the large time average $t^{-1}u(x,t)$ is given by the solution to the effective equation
\begin{equation}
\label{eq:effHJ}
  \ol{u}_t + \ol{H}(D\ol{u}) = 0,
\end{equation}
with the same initial condition.
Here, $\ol{u}(\cdot,0) \equiv 0$.
\medskip

The following is our first main result.
\begin{thm}\label{thm:main-1} Assume that, for some $c,C_0>0$,
\begin{equation*}
  \E \exp \Bigl[c\bigl( \int_{0}^{1} \nu_{r} \, \dd r\bigr) \Bigr] <C_0< \infty.
\end{equation*}
Let $u$ be the solution \eqref{eq:HJ}. 
For $R\ge 2$ and $t\ge C$, there exist two positive random variables $\mathcal{X}_1, \mathcal{X}_2$ with $\E[\exp(c\mathcal{X}_1^{2/3})+\exp(c\mathcal{X}_2^{2/5})]<C$ such that, almost surely,
\begin{itemize}
  \item[(i)] $\sup_{x\in B_R}\frac{u(x,t)}{t} + \ol{H}(0)\le  (\log R)^{3/2} t^{-1/2} \psi(t)\mathcal X_1$;
  \item[(ii)] $\inf_{x\in B_R}\frac{u(x,t)}{t} + \ol{H}(0)\ge -(\log R)^{5/2}t^{-1/2} \psi(t)\mathcal X_2$.
\end{itemize}
Here $c, C$ are generic constants depending only on $(c,d,q,C_0, N_1)$, and 
\[
\psi(t):=\exp\left(4c_1(\log\log t)(\log t)^{1/(2\wedge q)}\,\right)
\]
is a slowly varying function defined for $t>e$.
The constant $c_1>0$, depending only on $(c,d,q,N_1)$, is given in Theorem \ref{thm:path-reg}.
\end{thm}

Let us mention that, regarding the convergence of $m$, a similar quantitative version of \eqref{eq:olLdef} is obtained in Corollary \ref{C.4.10}.
As an application, we also obtain the following quantitative homogenization result.

\begin{thm}\label{thm:main-2} Assume that, for some $c, C_0>0$,
\begin{equation*}
  \E \exp \Bigl[c\bigl( \int_{0}^{1} \nu_{r} \, \dd r\bigr) \Bigr] <C_0< \infty.
\end{equation*}
Let $g\in \Lip(\R^d)$.
For $\ep\in (0,1)$, let $u^\ep$ be the solution to
\begin{equation}\label{eq:HJ-ep}
    \begin{cases}
        u_t^\ep + H\left(\tfrac{x}{\ep},\tfrac{t}{\ep},Du^\ep,\omega\right)=0 \qquad &\text{ in } \R^d \times (0,\infty),\\
        u^\ep(x,0)=g(x) \qquad &\text{ on } \R^d.
    \end{cases}
\end{equation}
Let $\ol u$ be the solution to \eqref{eq:effHJ} with initial data $\ol u(x,0)=g(x)$ for $x\in \R^d$.

For any $R\ge 2, \ep\in(0,1), t\ge C\ep$,  there exist two positive random variables $\mathcal{X}_1, \mathcal{X}_2$ with $\E[\exp(c\mathcal{X}_1^{2/3})+\exp(c\mathcal{X}_2^{2/5})]<C$ such that, almost surely, 
\begin{itemize}
  \item[(i)] $\sup_{x\in B_R}\left(u^\ep(x,t)-\ol u(x,t)\right)\le  (\log R)^{3/2}(\ep t)^{1/2} \psi(\tfrac t\ep)\mathcal X_1$;
  \item[(ii)] $\inf_{x\in B_R}\left(u^\ep(x,t)-\ol u(x,t)\right)\ge  -(\log R)^{5/2}(\ep t)^{1/2} \psi(\tfrac t\ep)\mathcal X_2$.
\end{itemize}
And for $\ep\in(0,1), t\in[0, C\ep]$,
\[
\sup_{x\in \R^d}|u^\ep(x,t)-\ol u(x,t)| \leq \bar C t +\ep\int_0^{t/\ep} \nu_0(s)\,\dd s \leq \bar C C \ep+\ep\int_0^{C} \nu_0(s)\,\dd s.
\]
Here $c, C, \bar C$ are constants depending only on $(c,d,q,C_0, N_1, \norm{g}_{\rm Lip}).$
\end{thm}

\medskip

To the best of our knowledge, the results of Theorems \ref{thm:main-1}--\ref{thm:main-2} are new.
A notable feature of our framework is that it requires substantially weaker assumptions on the random environment than previous quantitative results.
We assume only that the environment is stationary ergodic in space, with no mixing condition.
The quantitative convergence rates are obtained solely from the unit-range dependence in time.
By contrast, earlier quantitative results typically rely on finite-range dependence in all variables.
Moreover, we only require $L$ and hence $H$ to satisfy the random bound in (A2) with $\E \exp \Bigl[c\bigl( \int_{0}^{1} \nu_{r} \, \dd r\bigr) \Bigr] <C_0$, and we do not impose any uniform boundedness assumption on them in $t$.

\medskip

To prove the main results, we introduce two key new ingredients.
The first is a regularity theory for optimal curves and the metric distance: we establish H\"older regularity for every exponent in $(0,1)$, as well as an almost Lipschitz estimate (Theorems \ref{thm:path-reg}--\ref{thm:time-reg} and Corollary \ref{cor:Holder}).
These estimates play a central role in the analysis.
For Hamilton--Jacobi equations with superlinear time-dependent Hamiltonians, H\"older regularity was previously obtained in \cite{MR2583307,MR2969493,Seeger,cardaliaguet2021holder}, although with unspecified H\"older exponents.
Such estimates were crucial for the qualitative homogenization results in \cite{MR2514380,MR3602941,MR3817561,Seeger}, but they are not sufficient for quantitative results.
Here, we first prove a rough H\"older estimate in Lemma \ref{lem:rough-sp-reg}, with an explicit exponent, inspired by an idea of \cite{Mitake-Ni-Tran}.
We then refine this estimate through an iterative procedure, with Proposition \ref{prop:m-reg-induc} serving as the key inductive step.
The almost Lipschitz regularity of the solution to \eqref{eq:HJ} is given in Theorem \ref{thm:reg-u-nearly-Lip}.
As far as we know, Theorems \ref{thm:path-reg}--\ref{thm:time-reg} and Theorem \ref{thm:reg-u-nearly-Lip} are new and may be of independent interest.
The regularity-improving iteration itself also appears to be new in the literature of Hamilton--Jacobi equations.

\medskip

The second new ingredient is a quantitative estimate on the deterministic error of the large-time average of the metric distance (Theorem \ref{thm:error-deterministic}), inspired by the approach of \cite{Alexander-97} for first-passage percolation (FPP).
In the classical FPP setting, one can always extract a skeleton from any optimal path since optimal paths are Lipschitz in space.
In our dynamic random setting, however, the absence of a deterministic Lipschitz bound prevents such a direct construction: optimal paths are not a priori restricted to a fixed space-time cone.
While Lipschitz control fails at the microscopic scale, Theorem~\ref{thm:path-reg} provides a (random) Lipschitz regularity at the $O(1)$ scale for the optimal curves.
This allows us, with high probability, to construct a suitable skeleton along an optimal path.
The existence of such a good skeleton (Proposition \ref{prop:skeleton-10n-points}), together with another iterative argument (Lemma \ref{lem:iteration-GAP}), leads to the proof of Theorem \ref{thm:error-deterministic}.
Lemma \ref{lem:iteration-GAP} and Proposition \ref{prop:skeleton-10n-points} extend the corresponding results of \cite{Alexander-97} to the dynamic random setting, where no deterministic Lipschitz bound is available.

Finally, the convergence rates in Theorems \ref{thm:main-1}--\ref{thm:main-2} are of order $1/2$, up to slowly varying factors.
We expect that this rate is essentially optimal, and we formulate this expectation as a conjecture.

\begin{conj}
In the setting of Theorems \ref{thm:main-1}--\ref{thm:main-2}, the convergence rates obtained therein are optimal up to slowly varying factors. In particular, the optimal convergence rate is of exponent $1/2$. 
\end{conj}

We note that assumption (A2) excludes the tensionless KPZ equation from the present analysis. We plan to address this case in future work.

\subsection*{Notations}
We write $A\lesssim B$ if $A\le CB$ where $C$ depends only on $d,q, N_1$, and $A\asymp B$ if $A\lesssim B$ and $A\gtrsim B$. 
We also use the notations $A\lesssim_j B$ and $A\asymp_j B$ to indicate that the multiplicative constant depends on the variable $j$ other than $d,q, N_1$.
We sometimes omit the explicit dependence of random variables on the environment $\omega$ when this does not cause any ambiguity.

\subsection*{Organization of the paper}
The paper is organized as follows.
In Section \ref{sec:metric}, we study the H\"older and almost Lipschitz regularities of the metric distance and the optimal curves.
Sections \ref{sec:large-time-metric} and \ref{sec:deterministic-error-metric} are devoted to the quantitative large-time averages of the metric distance.
The main results are proved in Section \ref{sec:proof of main results}.
Some basic estimates are proved in Appendix \ref{appendix:basic estimates} and some concentration inequalities are given in Appendix \ref{appendix:concentration inequalities}.

\subsection*{Acknowledgment}
The authors would like to thank Konstantin Khanin for various useful discussions concerning the tensionless KPZ equation and related questions at the conference ``Recent progress in Hamilton--Jacobi equations and related topics" at Nanjing University in June 2025.

\section{Regularity of the metric distance and the optimal curves}\label{sec:metric}

In this section, we fix $\omega\in \Omega$ and study the regularity of $m$.
Recall that, for any $(x_i,t_i)\in \R^{d+1}$ with $i=1,2$ and $t_2>t_1$, 
\begin{equation*}
\begin{aligned}
m(x_1,t_1;x_2,t_2,\omega)&=
\inf \Big\{\int_{t_1}^{t_2} L(\gam(s),s,\dot\gam(s),\omega)\,\dd s: \gam\in\AC([t_1,t_2],\R^d), \\
&\qquad\quad\qquad\quad\qquad\quad\qquad\quad\qquad  \gam(t_1)=x_1,\gam(t_2)=x_2 \Big\}.
\end{aligned}
\end{equation*}
For $(x,t)\in \R^d\times [0,\infty)$, we write
\[
m(x,t,\omega)=m(0,0;x,t,\omega).
\]

Here are the main results in this section.
\begin{thm}\label{thm:path-reg}
	For any $x\in\R^d$ and $0<s<t$, let $\gamma=(\gamma(r))_{r\in[0,t]}\in\AC([0,t],\R^d)$ be an optimal path for $m(x,t)=m(0,0;x,t)$. 
    Then, there exists a constant $c_1>0$ such that
\begin{align}
  |\gamma(s)-\gamma(t)|&\lesssim (t-s)\exp\left(c_1(\log(\tfrac{t}{t-s}))^{1/(2\wedge q)}/q\right)A_{x,t,s}^{1/q},\label{eq:best-path-reg}\\
  -\int_s^t\nu_r\,\dd r\le m(\gamma(s),s;x,t)&\lesssim (t-s)\exp\left(c_1(\log(\tfrac{t}{t-s}))^{1/(2\wedge q)}\,\right)A_{x,t,s},\label{eq:best-metric-reg}
\end{align}
where 
\beq\lb{Axts}
 A_{x,t,s}:=\frac{|x|^q}{t^q}+\sup_{w\le s}\avint_{w\vee 0}^t\nu_r\,\dd r.
\eeq
Moreover, if $s>c_0t$ for some $c_0\in(0,1)$, then
\begin{equation}
\label{eq:change-end-path-same-s}
    m(\gamma(t),s)-m(\gamma(s),s)\lesssim_{c_0} (t-s)\exp\left(c_1(\log(\tfrac{t}{t-s}))^{1/(2\wedge q)}\,\right)A_{x,t,s}.
\end{equation}
\end{thm}

The ``almost Lipschitz" regularity of the path implies the following almost Lipschitz continuity of $m$.
Recall that 
\beq\lb{varphi}
\varphi(s)= \begin{cases} \exp\left(c_1(\log s)^{1/(2\wedge q)}\,\right)\quad &\text{ for }s\geq 1,\\
1 \quad &\text{ for }s \in (0,1),
\end{cases}
\eeq
and $A_{x,t,s}$ is given in Theorem~\ref{thm:path-reg}. Note that $\nu_r\geq 1$,  $A_{x,t,s}\geq 1$.

\medskip 

For the regularity of $m$ in space, we have: 
\begin{thm}\label{thm:sp-reg-t}
    For any $x,y\in\R^d$ and $t>0$, we have
\begin{equation}
\label{eq:best-sp-reg-metric}
|m(y,t)-m(x,t)|\lesssim
|x-y|\varphi\bigl(\frac{t}{|x-y|}\bigr)\left(\frac{|x-y|^q}{t^q} + A_{x,t,t-|x-y|}\right).
\end{equation}
\end{thm}

For the regularity of $m$ in time, we have: 
\begin{thm}\label{thm:time-reg}
	For $x\in\R^d$ and $0<s<t$, we have
\begin{equation}
\label{eq:best-metric-reg-t}
-\int_s^t\nu_r\,\dd r\le  m(x,s)-m(x,t)\lesssim (t-s)\varphi\bigl(\frac{t}{t-s}\bigr) \bigl(\frac{t}{s}\bigr)^{q-1} A_{x,t,s},
\end{equation}
\end{thm}

The proof of those regularity results is given later in this section. From the proof, we see that the $A_{x,t,\cdot}$ term in \eqref{eq:best-sp-reg-metric} can be replaced by $A_{x,t,t-|x-y|\varphi(t/|x-y|)}$, and the one in \eqref{eq:best-metric-reg-t} can be replaced by $A_{x,t,t-(t-s)\varphi(t/(t-s))}$, respectively. Since $A_{x,t,\tau}$ is non-decreasing in $\tau$, the estimates are slightly better after those replacements.

Notice that, for any $\varepsilon>0$, there exists $C_{\varepsilon,2\wedge q}>0$ such that
\begin{align*}
   \qquad (t-s)\exp\left(c_1(\log(\tfrac{t}{t-s}))^{1/(2\wedge q)}\,\right)&\le C_{\varepsilon,2\wedge q}(t-s)^{1-\varepsilon}t^\varepsilon \quad\text{ for all }0<s<t,\\
   \text{and }\quad (t-s)\exp\left(c_1(\log(\tfrac{t}{t-s}))^{1/(2\wedge q)}\,\right)&\le (t-s)^{1-\varepsilon}t^\varepsilon\quad\text{ whenever }C_{\varepsilon,2\wedge q}\le \tfrac{t}{t-s} .
\end{align*}
\begin{cor}\label{cor:Holder}
	Let $x,y\in\R^d$ and $t>0$. Then, for any $\varepsilon\in(0,1)$,
  \begin{equation*}
  |m(y,t)-m(x,t)|\lesssim_{\varepsilon}t^\varepsilon|x-y|^{1-\varepsilon} B_{x,y,t},
\end{equation*}
where 
\[
  B_{x,y,t}:=\frac{(|x|^q+|y|^q)}{t^q}+\sup_{w\in\bigl[\big(\tfrac{|x-y|}{|x|+|y|+t}\big)^{1-\varepsilon}t,t\bigr]}\avint_{t-w}^t\nu_r\,\dd r.
\]

Next, for $0<s<t$ and any $\varepsilon\in(0,1)$,
	\[
 -\int_s^t\nu_r\,\dd r\le  m(x,s)-m(x,t)\lesssim_{\varepsilon}t^{\varepsilon q}(t-s)^{1-\varepsilon q}\bigl(\frac{t}{s}\bigr)^{q-1}A_{x,t,s}.
\]
\end{cor}

\subsection{Basic bounds}
First, by the upper bound in \eqref{eq:L-bd}, computing the cost along a straight line segment from $(y,s)$ to $(x,t)$,
\begin{equation}\label{eq:m-upper-bd}
  m(y,s;x,t)\le \frac{N_1|x-y|^q}{|t-s|^{q-1}}+\int_s^t\nu_r\,\dd r.
\end{equation}

Next, for any $\gamma\in\AC([s,t],\R^d)$,
 \begin{equation}\label{eq:jensen}
  \int_s^tL(\gamma(r),r,\dot\gamma(r))\,\dd r\ge \int_s^tL\left(\gamma(r),r,\tfrac{\gamma(t)-\gamma(s)}{t-s}\right)\,\dd r-2\int_s^t\nu_r\,\dd r. 
\end{equation}
Indeed, by \eqref{eq:L-regularity}, $|L(\gamma(r),r,\dot\gamma(r))-L(\gamma(\tau),\tau,\dot\gamma(r))|\le \nu_r+\nu_\tau$ for any $r,\tau\in [s,t]$. Thus, for any $\tau\in(s,t), $
\begin{align*}
\avint_s^t L(\gamma(r),r,\dot\gamma(r))\,\dd r&\ge \avint_s^t \left(L(\gamma(\tau),\tau,\dot\gamma(r))-\nu_r-\nu_\tau\right)\,\dd r\nonumber\\
&\ge L\left(\gamma(\tau),\tau,\tfrac{\gamma(t)-\gamma(s)}{t-s}\right)-\avint_s^t\nu_r\,\dd r-\nu_\tau,
\end{align*}
where we used Jensen's inequality in the last inequality. Display \eqref{eq:jensen} follows by integrating the above inequality over all $\tau\in(s,t)$.

By \eqref{eq:jensen} and \eqref{eq:L-bd}, we have
\begin{equation}\label{eq:m-lower-bd}
  m(y,s;x,t)\ge \frac{|x-y|^q}{N_1|t-s|^{q-1}}-3\int_s^t\nu_r\,\dd r.
\end{equation}

\subsection{Rough regularity of the metric and the curve}
Fix $(x_0,t_0)\in \R^d\times (0,\infty)$.
Let $\gamma=(\gamma(r))_{0\le r\le t_0}$ be an optimal path for $m(x_0,t_0)$. 

\medskip

Let $[s,t]\subset[0,t_0]$.
We claim that 
\begin{equation}\label{eq:m-seg-rough-ub}
    m(\gamma(s),s;\gamma(t),t)\le m(x_0,t_0)+\int_{[0,t_0]\setminus[s,t]}\nu_r\,\dd r.
\end{equation}
Consequently, letting $q'=\tfrac{q}{q-1}$ denote the H\"older conjugate of $q>1$,
\begin{equation}\label{eq:path-rough-reg}
  |\gamma(s)-\gamma(t)|\le |s-t|^{1/q'}N_1^{1/q}\left(m(x_0,t_0)+4\int_0^{t_0}\nu_r\,\dd r\right)^{1/q}.
\end{equation}

For the ease of notation, we write 
\begin{equation}\label{notation:m-gamma}
  m^\gamma(s,t):=\int_s^tL(\gamma(r),r,\dot\gamma(r))\,\dd r.
\end{equation}
 
By the lower bound in \eqref{eq:L-bd}, we have
\begin{align*}
m^\gamma(0,t_0)&=m^\gamma(0,s)+m^\gamma(s,t)+m^\gamma(t,t_0)\\
&\ge -\int_{[0,s]\cup[t,t_0]}\nu_r\,\dd r+m^\gamma(s,t).
\end{align*}
Display \eqref{eq:m-seg-rough-ub} is proved. 
 Display \eqref{eq:path-rough-reg} follows from \eqref{eq:m-seg-rough-ub} and inequality \eqref{eq:m-lower-bd}.

\begin{rem} 
  The right-hand side of \eqref{eq:m-seg-rough-ub} does not go to $0$ as $s\to t$, so inequality \eqref{eq:m-seg-rough-ub} is not desirable since continuity of the cost is expected along the path. Thus, inequality 	\eqref{eq:path-rough-reg} is not optimal as well, since it is deduced from \eqref{eq:m-seg-rough-ub}.
\end{rem}

\begin{lem}\label{lem:rough-sp-reg}
Let $x,y\in\R^d$ and let $\gamma=(\gamma(r))_{r\in[0,t]}$ be an optimal path for $m(x,t)=m(0,0;x,t)$. 
Then, for any $0<\tau<t$,
\begin{align}\label{eq:spa-reg-pre-optim}
 \MoveEqLeft m(y,t)-m(x,t)\\
 &\le C_q|x-y|\Bigl(|\frac{x-\gamma({t-\tau})}{\tau}|^{q-1}+|\frac{x-y}{\tau}|^{q-1}+1\Bigr)+4\int_{t-\tau}^t\nu_r\,\dd r.\nonumber
\end{align}
\end{lem}
\begin{proof}
  Let $\beta\in \AC([0,t],\R^d)$ be such that $\beta$ coincides with $\gamma$ up to time $t-\tau$ and is a straight line segment connecting $(\gamma(t-\tau),t-\tau)$ to $(y,t)$ on the time interval $[t-\tau,t]$.
  See Figure \ref{fig:gamma-beta}.
  Then,
\begin{align*}
\MoveEqLeft  m(y,t)-m(x,t)\\&\le m^\beta(0,t)-m^\gamma(0,t)=m^\beta(t-\tau,t)-m^\gamma(t-\tau,t)\\
  &=\int_{t-\tau}^tL\left(\beta(r),r,\tfrac{y-\gamma({t-\tau})}{\tau}\right)\,\dd r-m^\gamma(t-\tau,t)\\
  &\stackrel{\eqref{eq:jensen}}\le \int_{t-\tau}^tL\left(\beta(r),r,\tfrac{y-\gamma({t-\tau})}{\tau}\right)\,\dd r
  -\int_{t-\tau}^tL\left(\gamma(r),r,\tfrac{x-\gamma({t-\tau})}{\tau}\right)\,\dd r+2\int_{t-\tau}^t\nu_r\,\dd r\\
  &\stackrel{\eqref{eq:L-regularity}}\le 
  C\tau\bigl|\frac{x-y}{\tau}\bigr|\Bigl(\bigl|\frac{x-\gamma({t-\tau})}{\tau}\bigr|^{q-1}+\bigl|\frac{y-\gamma({t-\tau})}{\tau}\bigr|^{q-1}+1\Bigr)+4\int_{t-\tau}^t\nu_r\,\dd r\\
  &\le C_q|x-y|\Bigl(\bigl|\frac{x-\gamma({t-\tau})}{\tau}\bigr|^{q-1}+\bigl|\frac{x-y}{\tau}\bigr|^{q-1}+1\Bigr)+4\int_{t-\tau}^t\nu_r\,\dd r.
\end{align*}
The proof is complete.
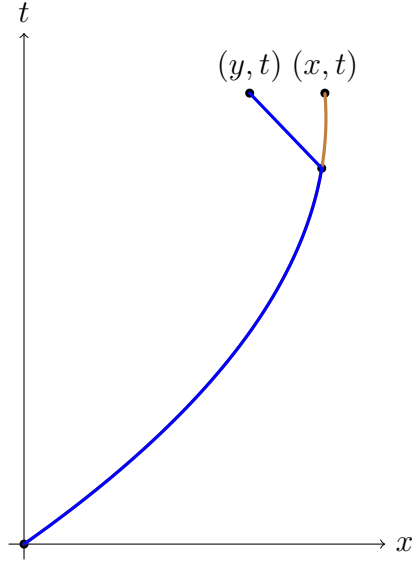
\begin{figure}[htp!]
\begin{center}
\begin{tikzpicture}[scale=2, line cap=round, line join=round]
  \draw[->] (-0.1,0) -- (2.4,0) node[right] {$x$};
  \draw[->] (0,-0.1) -- (0,3.4) node[above] {$t$};

  \coordinate (A) at (0,0);        
  \coordinate (B) at (2,3);        
  \coordinate (C) at (1.5,3);      

  %
  \coordinate (M) at (1.9792,2.5);

  \fill (A) circle (0.03);
  \fill (B) circle (0.03);
  \fill (C) circle (0.03);
  \fill (M) circle (0.03);

  \draw[very thick, brown]
    plot[domain=0:3, samples=200]
      ({(2/3)*\x + 0.25*\x*(3-\x)}, {\x});

  \draw[very thick, blue]
    plot[domain=0:2.5, samples=140]
      ({(2/3)*\x + 0.25*\x*(3-\x)}, {\x});

  \draw[very thick, blue] (M) -- (C);

  \node[above] at (B) {$(x,t)$};
  \node[above] at (C) {$(y,t)$};

\end{tikzpicture}
\end{center}
\caption{Curves $\gamma$ and $\beta$}\label{fig:gamma-beta}
\end{figure}
\end{proof}

\subsection{Almost-Lipschitz regularity via iterations}
For $a\in (0,1]$, set
 \[
  f(a):=\frac{a}{(1-a)(q-1)q+q}+\frac{q-1}{q}, \qquad g(a):=\frac{a}{(1-a)(q-1)+1}.
\] 
Notice that 
\[
f(a) = \frac{1}{q}g(a) + \frac{1}{q'} \quad \text{ and } \quad 1-f(a)=\frac{1}{q}(1-g(a)).
\]

The following proposition is crucial in the proof of Theorem~\ref{thm:path-reg}. 
It states that a rough regularity estimate for an optimal path can be bootstrapped to obtain better regularity. 
\begin{prop}\label{prop:m-reg-induc}
	Let $x\in\R^d$, and let $\gamma$ be an optimal path for $m(x,t)$. 
	Define a quantity $A_{t,s}(x)$ (and, simply denoted by $A_{t,s}$ in the sequel) by
\begin{equation}\label{def:ats}
 A_{t,s}=\frac{m(x,t)}{t}+4\sup_{w\in(0,s)}\avint_w^t\nu_r\,\dd r.
\end{equation}
    Fix $c_0\in (0,1)$.
	Suppose for some constants $K\ge 1$, $a\in[1/q',1)$, and $l\geq a/(1-a)$, it holds  that, for all $s\in((1-c_0^l)t, t]$ (that is, $(t-s)/t\le c_0^l$),
	\begin{equation}\label{eq:induction}
  |\gamma(t)-\gamma(s)|\le K^{1/q}\bigl(\frac{t-s}{t}\bigr)^atA_{t,s}^{1/q} .
\end{equation}
Then, there exists a constant $C_1=C_1(N_1,d,q,c_0)$ such that, for all $s\in((1-c_0^{l/g(a)})t,t]$ (that is, $(t-s)/s < c_0^{l/g(a)}$),
\begin{align}
  m(\gamma(t),s)-m(\gamma(s),s)&\le C_1K \bigl(\frac{t-s}{t}\bigr)^{g(a)}tA_{t,s},\label{eq:251207-2}\\
m^\gamma(s,t)&\le C_1K \bigl(\frac{t-s}{t}\bigr)^{g(a)}tA_{t,s},
\label{eq:m-seq-better}\\
  |\gamma(t)-\gamma(s)|&\le (C_1K)^{1/q}\bigl(\frac{t-s}{t}\bigr)^{f(a)}tA_{t,s}^{1/q}.\label{eq:induction-path-reg-next}
\end{align}
\end{prop}

\begin{proof} 
We divide the proof into several steps.

\medskip

\noindent {\bf Step 1.} Recall notation $m^\gamma$ in \eqref{notation:m-gamma}.
First, we claim that, for $x\in \R^d$,
\begin{equation}\label{eq:251207-3}
  m(x,t)-m(x,s)\le \int_s^t\nu_r\,\dd r.
\end{equation}
To see this, let $\alpha\in\AC([0,s],\R^d)$ be an optimal path for $m(x,s)$, and
let $\beta\in\AC([0,t],\R^d)$ be a curve with $\beta|_{[0,s]}=\alpha|_{[0,s]}$ and $\beta|_{[s,t]}\equiv x$. 
Then,
\begin{align*}
m(x,t)-m(x,s)&\le m^\beta(0,t)-m^\alpha(0,s)\nonumber\\
&=m^\beta(s,t)=\int_s^tL(x,r,0)\,\dd r\nonumber\\
&\stackrel{\eqref{eq:L-bd}}\le \int_s^t\nu_r\,\dd r.
\end{align*}
Inequality \eqref{eq:251207-3} is proved.

\medskip

\noindent {\bf Step 2.}	
Next, we will prove \eqref{eq:251207-2} for $s\in (0,t)$ such that $(\tfrac{t-s}{t})< c_0^{l/g(a)}$.  
To this end,  let $\theta := g(a)$ and set
	\[
  \tau=t\bigl(\frac{t-s}{t}\bigr)^\theta \in (0,t), \qquad \tau_0=\tau-(t-s)\in(0,s).
\]
Note also that $s-\tau_0=t-\tau$. We check that, 
\begin{equation}
\label{eq:timeratios}
\frac{t-s}{t} < c_0^{l/g(a)} < c_0^l, \qquad \frac{t-(t-\tau)}{t} = \bigl(\frac{t-s}{t}\bigr)^{\theta} < (c_0^{l/g(a)})^\theta = c_0^l.
\end{equation}
Moreover,
\begin{equation}
  \label{eq:tauratios}
1>\frac{\tau_0}{\tau} = 1-\bigl(\frac{t-s}{t}\bigr)^{1-g(a)} \ge 1-c_0^{l(1-g(a))/g(a)} \ge 1-c_0^q.
\end{equation}
Here we used the fact that $\tfrac{1-g(a)}{g(a)} = \tfrac{q(1-a)}{a}$, and $l(1-a)/a \ge 1$. 

\medskip

By Lemma~\ref{lem:rough-sp-reg}, we have
\begin{align*}
&\qquad m(\gamma(t),s)-m(\gamma(s),s)\\
&\lesssim  \ |\gamma(t)-\gamma(s)||\frac{\gamma(s)-\gamma({s-\tau_0})}{\tau_0}|^{q-1}+\frac{|\gamma(t)-\gamma(s)|^{q}}{\tau_0^{q-1}}+|\gamma(s)-\gamma(t)|+\int_{s-\tau_0}^s\nu_r\,\dd r\\
&\lesssim  \ \frac{|\gamma(t)-\gamma(s)|}{\tau_0^{q-1}}\Bigl[|\gamma(t)-\gamma({s})|^{q-1}+|\gamma(t)-\gamma({s-\tau_0})|^{q-1}\Bigr]+\frac{|\gamma(t)-\gamma(s)|^q}{\tau_0^{q-1}}\\
&\qquad\qquad\qquad\qquad\qquad\qquad\qquad\qquad \ \ +|\gamma(s)-\gamma(t)|+\int_{t-\tau}^t\nu_r\,\dd r\\
&\asymp \ \frac{|\gamma(t)-\gamma(s)|^q}{\tau_0^{q-1}}+\frac{|\gamma(t)-\gamma(s)||\gamma(t)-\gamma({t-\tau})|^{q-1}}{\tau_0^{q-1}}+|\gamma(t)-\gamma(s)|+\int_{t-\tau}^t\nu_r\,\dd r\\
&\lesssim_{c_0} \frac{|\gamma(t)-\gamma(s)|^q}{\tau^{q-1}}+\frac{|\gamma(t)-\gamma(s)||\gamma(t)-\gamma({t-\tau})|^{q-1}}{\tau^{q-1}}+|\gamma(t)-\gamma(s)|+\tau \fint_{t-\tau}^t \nu_r \,\dd r.
\end{align*}
In the last line, we used the bounds in \eqref{eq:tauratios} which says $\tau$ and $\tau_0$ are comparable. 

\smallskip 

In view of the inequalities \eqref{eq:timeratios}, for both $|\gamma(t)-\gamma(s)|$ and $|\gamma(t)-\gamma(t-\tau)|$, we can apply \eqref{eq:induction} to estimate those terms. Since $t-\tau < s$, $\fint_{t-\tau}^t \nu_r\,\dd r\le A_{t,s}$. Then from the above estimate we deduce (recall also that $K,A_{t,s}\ge 1$)
 \begin{align*}
\MoveEqLeft m(\gamma(t),s)-m(\gamma(s),s)\\
 &\lesssim_{c_0} 
 KtA_{t,s}\Bigl[\bigl(\frac{t-s}{t}\bigr)^{aq-\theta(q-1)}
 +\bigl(\frac{t-s}{t}\bigr)^{a-\theta(1-a)(q-1)}+\bigl(\frac{t-s}{t}\bigr)^{a}+\bigl(\frac{t-s}{t}\bigl)^\theta
 \Bigr].
\end{align*}
Recalling that $\theta=\tfrac{a}{(1-a)(q-1)+1}$ and noticing that $ a\in[1/q',1)$, we check that the smallest exponent in $(\tfrac{t-s}{t})$ in the above estimate is precisely $\theta$, and, hence, get
\[
m(\gamma(t),s)-m(\gamma(s),s)\lesssim_{c_0} Kt\bigl(\frac{t-s}{t}\bigr)^\theta A_{t,s}
\]
Display \eqref{eq:251207-2} is proved for $s \in (0,t)$ such that $(s-t)/t < c_0^{l/g(a)}$.

\medskip

\noindent {\bf Step 3.} 
Combining \eqref{def:ats}, \eqref{eq:251207-3},
and \eqref{eq:251207-2}, we obtain, for any $s\in(0,t)$ such that $(t-s)/t < c_0^{l/g(a)}$,
\begin{align*}
 m^\gamma(s,t)&=m(\gamma(t),t)-m(\gamma(t),s)+m(\gamma(t),s)-m(\gamma(s),s)\\
 &\lesssim_{c_0}Kt \bigl(\frac{t-s}{t}\bigr)^{g(a)}A_{t,s}+\int_s^t\nu_r\,\dd r\lesssim_{c_0} Kt\bigl(\frac{t-s}{t}\bigr)^{g(a)}A_{t,s}
\end{align*}
which proves \eqref{eq:m-seq-better}.
Furthermore, by \eqref{eq:m-lower-bd}, 
\begin{equation*}
m^\gamma(s,t)\ge \frac{|\gamma(t)-\gamma(s)|^q}{N_1(t-s)^{q-1}}-3\bigl(\frac{t-s}{t}\bigr)t\fint_s^t \nu_r\,\dd r.
\end{equation*}
This inequality and \eqref{eq:m-seq-better} imply, for any $s\in((1-e^{-c_0l/g(a)})t,t)$,
\[
  \frac{|\gamma(t)-\gamma(s)|^q}{N_1(t-s)^{q-1}}\lesssim_{c_0} Kt\bigl(\frac{t-s}{t}\bigr)^{g(a)} A_{t,s}
\]
and so
\[
  |\gamma(t)-\gamma(s)|\lesssim_{c_0} (N_1K)^{1/q}\bigl(\frac{t-s}{t}\bigr)^{(g(a)+q-1)/q}tA_{t,s}^{1/q}.
\]
In other words, \eqref{eq:induction} still holds for $s\in((1-c_0^{l/g(a)})t,t)$, with $a$ replaced by 
\[
  f(a) = \frac{g(a) + q-1}{q} = \frac{1}{q}g(a) + \frac{1}{q'}
\]
and $K^{1/q}$ replaced by  $(C_1K)^{1/q}$ for some $C_1=C_1(q,c_0)N_1$. Display \eqref{eq:induction-path-reg-next} is proved.
The proof is complete.
\end{proof}

We are now ready to prove Theorem~\ref{thm:path-reg}. 

\begin{proof}[Proof of Theorem~\ref{thm:path-reg}]Notice that the lower bound in \eqref{eq:best-metric-reg} follows immediately from \eqref{eq:m-lower-bd}. We will focus on the upper bounds.

	Recall that \eqref{eq:path-rough-reg} implies 
\begin{equation}\label{eq:initial-induc}
    |\gamma(s)-\gamma(t)|\lesssim \bigl(\frac{t-s}{t}\bigr)^{1/q'}tA_{t,s}^{1/q} \qquad\forall \ s\in(0,t).
\end{equation}
Fix any $c_0 \in (0,1)$. When $s\le (1-c_0^{2q})t$, we have $(t-s)/t\asymp 1$, and \eqref{eq:initial-induc} implies \eqref{eq:best-path-reg}. Then, the upper bound of \eqref{eq:best-metric-reg} follows from \eqref{eq:m-upper-bd} and \eqref{eq:best-path-reg}. 

It remains to consider the case $s>(1-c_0^{2q})t.$
By \eqref{eq:initial-induc}, inequality \eqref{eq:induction} holds for $a$ being $a_0:=1/q'$. 
For $n\in \N$, set
\[
  a_n:=f^{\circ n}(a_0),
\]
where $f^{\circ n}$ denotes the composition of $f$ for $n$ times. 
Notice that the function $f$ is strictly increasing on interval $[a_0,1)$ with $f(1)=1$. Hence $a_n\uparrow 1$ as $n\to\infty$. 
In fact, in light of Lemma \ref{lem:b_n},
\begin{align}
  b_n=:1-a_n&=\frac{1}{n(q-1)+q}.\label{eq:bn-formula}
\end{align}
Set $l_0=a_0/(1-a_0)=q-1$, and, for $n\in \N$,
\[
l_n=\frac{l_0}{g(a_0)\cdots g(a_{n-1})}.
\]
Noticing that $1-g(a)=q(1-f(a))$, we have
\[
g(a_k)=1-q(1-f(a_k))=1-q(1-a_{k+1})=\frac{(k+1)(q-1)}{(k+1)(q-1)+q}=\frac{k+1}{k+1+q'}<\frac{k+1}{k+2}.
\]
In particular,
\[
l_n>(n+1)(q-1)=\frac{a_n}{1-a_n}.
\]
We also note that
\[
l_n=\frac{(q-1)(1+q')\cdots(n+q')}{n!}\asymp n^{q'} \quad \text{ and } \quad l_n \leq (q-1)e^{q'}n^{q'}.
\]
See Lemma \ref{lem:Gamma function}.
 By Proposition~\ref{prop:m-reg-induc} and induction, as long as
 \[
  s\in((1-c_0^{l_n})t, t],
\]
 we have (by an appropriate updated value of $C_1$)
 \begin{align}\label{eq:indu-results}
\begin{cases}
	&m(\gamma(t),s)-m(\gamma(s),s)\le C_1^{n} (\tfrac{t}{t-s})^{qb_n}(t-s)A_{t,s},\\
&m^\gamma(s,t)\le C_1^n (\tfrac{t}{t-s})^{qb_n}(t-s)A_{t,s},\\
  &|\gamma(t)-\gamma(s)|\le C_1^{n/q}(\tfrac{t}{t-s})^{b_{n}}(t-s)A_{t,s}^{1/q} .
\end{cases}
\end{align}

We let $N=N(c_0,q,t,s)$ be the smallest integer such that $s\le (1-c_0^{l_{N+1}})t$, which is equivalent to
\[
  \log\bigl(\frac{t}{t-s}\bigr)\le |\log c_0| l_{N+1}\leq |\log c_0| (q-1)e^{q'}(N+1)^{q'}.
\]
Then, inequalities in \eqref{eq:indu-results} hold for all $n\le N$. 
Let $J$ be the smallest integer such that 
\[
\bigl(\log\bigl(\frac{t}{t-s}\bigr)\bigr)^{1/(2\vee q')}\le
\bigl((q-1)e^{q'}|\log c_0|\bigr)^{1/(2\vee q')}(J+1),
\]
then $J\le N$. Moreover,
\begin{align*}
    C_1^J\bigl(\frac{t}{t-s}\bigr)^{qb_J}\le \exp\bigl(J\log C_1 + q'(J+1)^{-1}\log\bigl(\frac{t}{t-s}\bigr)\bigr)\le \exp\bigl(c_1(\log(\tfrac{t}{t-s}))^{1/(2\wedge q)}\bigr).
\end{align*}
We deduce
\begin{align*}
  |\gamma(s)-\gamma(t)|&\lesssim(t-s)\exp\left(c_1(\log(\tfrac{t}{t-s}))^{1/(2\wedge q)}/q\right)A_{x,t,s}^{1/q},\\  m(\gamma(s),s;x,t)&\lesssim(t-s)\exp\left(c_1(\log(\tfrac{t}{t-s}))^{1/(2\wedge q)}\,\right)A_{x,t,s}.
\end{align*}
Theorem~\ref{thm:path-reg} follows.
\end{proof}

Next, we prove that $m$ is almost Lipschitz in the spatial variable.

\begin{proof}[Proof of Theorem~\ref{thm:sp-reg-t}] 
Take $\tau\in(0,t)$, which is to be determined later. 
	Let $\gamma\in{\rm AC}([0,t],\R^d)$ be an optimal path for $m(x,t)$. 
	By \eqref{eq:best-path-reg} of Theorem~\ref{thm:path-reg}, 
	\[
  \left|\frac{x-\gamma({t-\tau})}{\tau}\right|\lesssim \varphi\bigl(\frac{t}{\tau}\bigr)^{1/q}A_{x,t,t-\tau}^{1/q},
\]
where $\varphi(s):=\exp(c_1(\log s)^{1/(2\wedge q)})$ for $s\geq 1$. We extend the function so that $\varphi(s)=1$ for $s\leq 1$.
This inequality, together with Lemma~\ref{lem:rough-sp-reg}, yields
\begin{equation}
\label{eq:pfthm2.2-1}
m(y,t)-m(x,t)\lesssim |x-y| \varphi\bigl(\frac{t}{\tau}\bigr)^{1/q'}A_{x,t,t-\tau}^{1/q'}+\frac{|x-y|^q}{\tau^{q-1}}+|x-y|+\tau A_{x,t,t-\tau}.
\end{equation}
Now, let $\tau:=|x-y|\varphi(\tfrac{t}{|x-y|})$. If $\tau < t$ is satisfied, we have the inequality above. Since $\varphi(s)\geq 1$ and is non-decreasing, $q'\geq 1$, and $A_{x,t,\cdot} \ge 1$, we have
\begin{equation*}
\begin{aligned}
  &\varphi(\tfrac{t}{\tau}) \le \varphi(\tfrac{t}{|x-y|}), \quad \varphi(\tfrac{t}{\tau})^{\frac1{q'}}A_{x,t,t-\tau}^{\frac1{q'}} \le \varphi(\tfrac{t}{\tau})A_{x,t,t-\tau}, \\
  &\frac{|x-y|^q}{\tau^{q-1}} \le \tau \le |x-y|\varphi(\tfrac{t}{|x-y|})A_{x,t,t-\tau}.
  \end{aligned}
\end{equation*}
Each of the terms on the the right hand side of \eqref{eq:pfthm2.2-1} is hence bounded by
\begin{equation*}
|x-y|\varphi(t/|x-y|)A_{x,t,t-\tau}.
\end{equation*}
Since in the current case $|x-y|/t \le \tau/t <1$ and $A_{x,t,t-\tau}\ge 1$, we see that the desired estimate \eqref{eq:best-sp-reg-metric} holds.

\medskip 

Finally, suppose that $\tau$ defined above satisfies $\tau \ge t$, then using the basic bounds, we get
\begin{equation}
\label{eq:best-sp-reg-metric-0}
\begin{aligned}
 |m(y,t)-m(x,t)| 
&\lesssim t\left(\frac{|x|^q+|y|^q}{t^q} + \fint_0^t \nu_r\,\dd r\right) \\
 &\lesssim
|x-y|\varphi\bigl(\frac{t}{|x-y|}\bigr)\left(\frac{|x|^q+|x-y|^q}{t^q} + \fint_0^t \nu_r\,\dd r\right).
\end{aligned}
\end{equation}
This still yields the desired estimate \eqref{eq:best-sp-reg-metric} and completes the proof.
\end{proof}

We now prove that $m$ is almost Lipschitz in time.
\begin{proof}[Proof of Theorem~\ref{thm:time-reg}]
The lower bound was in fact proved in \eqref{eq:251207-3}. 
For $s\in(0,\tfrac12 t]$, the upper bound follows from the fact that $t\asymp(t-s)$ and 
\[
  |m(x,s)|+|m(x,t)|\stackrel{\eqref{eq:m-upper-bd},\eqref{eq:m-lower-bd}}\lesssim \left(\frac{t}{s}\right)
^{q-1}\left(\frac{|x|^q}{|t|^{q-1}}+\int_0^t\nu_r\,\dd r\right).\]

It remains to prove the upper bound for $s\in(\tfrac{1}{2}t,t)$.
To this end,  let $\theta\in (0, 1)$ to be determined and set
	\[
  \tau:=(t-s)\varphi\bigl(\frac{t}{t-s}\bigr)^\theta \quad \text{and}\quad \tau_0:=\tau-(t-s).
\]
The definition of $\varphi$ yields
\[
\log \bigl(\varphi\bigl(\frac{t}{t-s}\bigr)^\theta \bigr)=c_1\theta \bigl(\log \frac{t}{t-s}\bigr)^{1/(2\wedge q)}.
\]
By fixing $\theta$ to be sufficiently small, we have $c_1\theta (\log z)^{1/(2\wedge q)}\leq \log z$ for all $z\geq 2$. This implies that for $s\in (\tfrac{1}{2}t,t)$, we have 
\[
\tau\in (t-s,t),\quad \tau_0\in (0,s)\quad\text{and}\quad \tau (1-\varphi(2)^{-\theta})\leq \tau_0\leq\tau.
\]

Let $\gamma\in{\rm AC}([0,t],\R^d)$ be an optimal path for $m(x,t)$. 
Let $\beta\in{\rm AC}([0,s],\R^d)$ be such that $\beta$ agrees with $\gamma$ on $[0,s-\tau_0]$, and $\beta$ is a straight line segment on $[s-\tau_0,s]$ connecting $(\gamma(s-\tau_0),s-\tau_0)$ to $(\beta(s),s)=(x,s)$. 
Then
\begin{align*}
\MoveEqLeft m(x,s)-m(x,t)\\&\le m^\beta(0,s)-m^\gamma(0,t)\\
&= m^\beta(s-\tau_0,s)-m^\gamma(s-\tau_0,t)\\
&\stackrel{\eqref{eq:jensen}}\le \int_{s-\tau_0}^s L(\beta(r),r,\tfrac{x-\gamma({s-\tau_0})}{\tau_0})\,\dd r-\int_{s-\tau_0}^tL(\gamma(r),r,\tfrac{x-\gamma({s-\tau_0})}{t-s+\tau_0})\,\dd r+2\int_{s-\tau_0}^t
\nu_r\,\dd r\\
&\stackrel{(A2)}\lesssim 
\int_{s-\tau_0}^s\Abs{\tfrac{x-\gamma({s-\tau_0})}{\tau_0}-\tfrac{x-\gamma({s-\tau_0})}{t-s+\tau_0}}\Bigl[\Abs{\tfrac{x-\gamma({s-\tau_0})}{\tau_0}}^{q-1}+\Abs{\tfrac{x-\gamma({s-\tau_0})}{t-s+\tau_0}}^{q-1}+1\Bigr]\,\dd r\\
&\qquad+\int_{s-\tau_0}^t\nu_r\,\dd r
\end{align*}
where in the last inequality we also used that
\[
  -\int_s^tL(\gamma(r),r,\tfrac{x-\gamma({s-\tau_0})}{t-s+\tau_0})\,\dd r\stackrel{(A2)}{\le}\int_s^t\nu_r\,\dd r.
\]

Since  $s-\tau_0=t-\tau$ and $\tau_0\asymp\tau$, this inequality becomes
\begin{align*}
\MoveEqLeft m(x,s)-m(x,t)\\
&\lesssim \tau_0\Abs{\frac{x-\gamma({s-\tau_0})}{\tau_0}-\frac{x-\gamma({s-\tau_0})}{t-s+\tau_0}}\Bigl[\Abs{\frac{x-\gamma({s-\tau_0})}{\tau_0}}^{q-1}+1\Bigr]+\int_{s-\tau_0}^t\nu_r\,\dd r\\
&\asymp \frac{t-s}{\tau}\Bigl[\frac{|x-\gamma({t-\tau})|^q}{\tau^{q-1}}+\abs{x-\gamma({t-\tau})}\Bigr]+\int_{t-\tau}^t\nu_r\,\dd r\\
&\stackrel{\eqref{eq:best-path-reg}}\lesssim
(t-s)\varphi\bigl(\frac{t}{\tau}\bigr) A_{x,t,t-\tau}+(t-s)\varphi\bigl(\frac{t}{\tau}\bigr)^{1/q} A_{x,t,t-\tau}^{1/q}+\tau\avint_{t-\tau}^t\nu_r\,\dd r\\
&\lesssim
(t-s)\varphi\bigl(\frac{t}{\tau}\bigr) A_{x,t,t-\tau}+ \tau A_{x,t,t-\tau}.
\end{align*}	
Since $\tau>t-s$, we have $
\varphi(\tfrac{t}{\tau})\leq \varphi(\tfrac{t}{t-s})$. 
Also using that $\tau\leq (t-s)\varphi(\tfrac{t}{t-s})$ and that $A_{x,t,s}$ is non-decreasing in $s$, we get
\[
  m(x,s)-m(x,t)\lesssim (t-s)\varphi\bigl(\frac{t}{t-s}\bigr) A_{x,t,t-(t-s)\varphi(\tfrac{t}{t-s})} \quad\forall \ s\in(\tfrac12t,t).
\]
Our proof is complete.
\end{proof}


\section{Concentration estimates of the metric distance}\label{sec:large-time-metric}

Here is the main result of this section.

\begin{thm}\label{thm:fluctuation-mxt}
Assume {\rm(A1)--(A2)}, and 
    let $x_0\in\R^d, t_0\ge 2$, $c>0$, and $C_0\geq 1$. 
	\begin{enumerate}[(a)]
  \item\label{item:algeb-integ-nu} If $\norm{\int_0^1\nu_r\,\dd r}_{L^p(\mb P)}<C_0$ for some $p\ge 2$, then there exist $C_1,c_1>0$ such that for any $\lambda>0$,
  \[
  \bP\left(|m(x_0,t_0)-\E m(x_0,t_0)|\ge \lambda t_0^{1/2}\varphi(t_0)(\tfrac{|x_0|^q}{t_0^q}+{C_0}\log t_0)\right)\leq \tfrac{C_1}{\lambda^p}.
\]

\item\label{item:exp-integ-nu} If $\E[\exp(c\int_0^1\nu_r\,\dd r)]<C_0$, then there exist $C_1,c_1>0$ such that for any $\lambda>0$,
\[
\bP\left(|m(x_0,t_0)-\E m(x_0,t_0)|\geq  \lambda t_0^{1/2}\varphi(t_0)(\tfrac{|x_0|^q}{t_0^q}+c^{-1}\log t_0) \right)
{\leq C_1C_0\exp\!\left(-c_1 \lambda^{2/3}\right)}.
\]

\item\label{item:bdd-nu} If $\int_0^1\nu_r\,\dd r<C_0$, $\bP$-a.s., then there exists $c_1>0$ such that for any $\lambda>0$,
\[
\bP\left(|m(x_0,t_0)-\E m(x_0,t_0)|\ge \lambda t_0^{1/2}\varphi(t_0)(\tfrac{|x_0|^q}{t_0^q}+{C_0}\log t_0)\right)\le {2}\exp(-c_1\lambda^2).
\]

\end{enumerate}
In the above statements, $\varphi$ is given in \eqref{varphi}, the constants $C_1,c_1$ are independent of $t_0,x_0$, $c$ and $C_0$. 
\end{thm}

We will use the following lemma.
\begin{lem}\label{lem:conditioning}
Consider the product probability space $(\Omega_1\times\Omega_2, \mathcal F_1\otimes\mathcal F_2,\bP)$.
Let $\mc F_3$ be a sub-$\sigma$-field of $\mc F_2$.  
For $j=2,3$, let $\widetilde{\mathcal{F}}_j := \{\Omega_1\times B\,:\,B\in \mathcal{F}_j\}$; they are  sub-$\sigma$-fields of $\mathcal{F}_1\otimes \mathcal{F}_2$.
Then, for any $\mc F_1\otimes\mc F_3$-measurable integrable random variable $X$, we have 
\begin{equation}\label{eq:filtration}
  \E[X|\widetilde{\mathcal F}_2]=\E[X|\widetilde{\mathcal F}_3].
\end{equation}
\end{lem}

\begin{proof}
	Let $\mc J=\{S\in \mathcal F_1\otimes\mc F_3: \eqref{eq:filtration} \text{ is true for }X=\mathbf 1_{S}\}$. It suffices to show that 
	\begin{equation}\label{eq:251223}
  \mc J=\mathcal F_1\otimes\mc F_3.
\end{equation}
 Let $\mc H=\{A\times B: A\in\mc F_1, B\in\mc F_3\}$ be the collection of cylinders in $\mc F_1\otimes\mc F_3$. Clearly, $\mc H\subset\mc J$. Since $\mc H$ is a $\pi$-system, and $\mc J$ is a $\lambda$-system, by Dynkin's $\pi-\lambda$ theorem, we have
 \[
  \sigma(\mc H)\subset\mc J.
\]
Thus, $\mc F_1\otimes\mc F_3\subset\mc J$. 
Since $\mc J\subset\mc F_1\otimes\mc F_3$, equality \eqref{eq:251223} follows.
\end{proof}

\medskip

We fix $(x_0,t_0)\in\R^d\times[10,\infty)$ and let $\rho:=\tfrac{t_0}{\floor{t_0}}\in[1,\tfrac{11}{10}]$. 

For $A\subset[0,\infty)$, we write 
\[
  \ms F_A=\sigma(\omega_{x,s}:x\in\Z^d,s\in A).
\]
For any $i=0,1,\ldots,\floor{t_0}-2$ , define a $\ms F_{(0,i\rho]}\otimes\ms F_{[(i+2)\rho,t_0)}$-measurable random variable
\begin{equation}\label{def:mui}
  \mu_i:=\inf_{y\in\R^d}\bigl[m(0,0;y,i\rho)+m(y,(i+2)\rho;x_0,t_0)\bigr].
\end{equation}
Notice that, since the field has a unit (time) range of dependence, by Lemma~\ref{lem:conditioning},
\begin{equation}\label{eq:forget-mui}
  \E[\mu_i|\ms F_{(0,i\rho]}]=\E[\mu_i|\ms F_{(0,(i+1)\rho]}], \qquad \forall \ 0\le i\le\floor{t_0}-2.
\end{equation}

Set, for $i\ge 0,$ 
\begin{equation}\label{def:m_i}
  M_i=\E[m(x_0,t_0)|\ms F_{(0,i\rho]}].
\end{equation}
Notice that $(M_i)_{i\ge 0}$ is a martingale with respect to the filtration $(\ms F_{(0,i\rho]})_{i\ge 0}$. Moreover, $M_0=\E[m(x_0,t_0)]$, and $M_n=m(x_0,t_0)$ for all $n\ge \floor{t_0}$.

\begin{prop}
	\label{prop:path-metric-reg}
	Let $x_0\in\R^d, t_0\geq 10$. Let $\gamma=(\gamma(r))_{r\in[0,t_0]}\in{\rm AC}([0,t_0],\R^d)$ be an optimal path for $m(x_0,t_0)=m(0,0;x_0,t_0)$. Then, for any $0\le s<t\le t_0$, 
\begin{align}
  -\int_s^t\nu_r\,\dd r\le m^\gamma(s,t)&\lesssim (t-s) \varphi\bigl(\frac{t_0}{t-s}\bigr)D_{x_0,t_0,s,t}\label{eq:path-reg-appl}
\end{align}
where 
\[
 D_{x_0,t_0,s,t}:=\frac{|x_0|^q}{t_0^q}+\sup_{w\ge t-s} {\Bigl[} \avint_{(t-w)\vee 0}^t\nu_r\,\dd r+\avint_s^{(s+w)\wedge t_0}\nu_r\,\dd r\Bigr].
\]
Moreover, recalling $\mu_i$ in \eqref{def:mui}, if $t_0\ge 2$, then for any $i=0,1,\ldots,\floor{t_0}-2$ ,
\begin{equation}
  \label{eq:m-mui-bd}
  \abs{m(x_0,t_0)-\mu_i}\lesssim\varphi(t_0) D_{x_0,t_0,i\rho,(i+2)\rho}.
\end{equation}
\end{prop}

\begin{proof} 
Note that, for $0<t<t_0$,
\begin{equation}\label{eq:251228}
  \frac{|\gamma(t)|^q}{t^{q-1}}\stackrel{\eqref{eq:m-lower-bd}}{\leq} m^\gamma(0,t)+3\int_s^t\nu_r\,\dd r\stackrel{\eqref{eq:m-seg-rough-ub}}{\leq}m(x_0,t_0)+4\int_0^{t_0}\nu_r\,\dd r\stackrel{\eqref{eq:m-upper-bd}}{\lesssim}\frac{|x_0|^q}{t_0^{q-1}}+\int_0^{t_0}\nu_r\,\dd r.
\end{equation}

The lower bound of \eqref{eq:path-reg-appl} was proved in \eqref{eq:best-metric-reg}.
We will prove the upper bound of \eqref{eq:path-reg-appl} for two cases: $t\ge\tfrac{t_0}{2}$ and $t<\tfrac{t_0}{2}$. 

When $t\ge\tfrac{t_0}{2}$, it follows from \eqref{eq:best-metric-reg}, \eqref{eq:251228}, and the fact that $t\asymp t_0$ that
\begin{equation}\label{eq:path-reg-appl-pf1}
\begin{aligned}
m^\gamma (s,t)&\lesssim (t-s)\exp\left(c_1(\log(\tfrac{t}{t-s}))^{1/(2\wedge q)}\,\right)A_{\gamma(t),t,s}\\
&\lesssim (t-s)\varphi\bigl(\frac{t_0}{t-s}\bigl)\left(\frac{|x_0|^q}{t_0^q}+\sup_{w\ge t-s}\avint_{(t-w)\vee 0}^t\nu_r\,\dd r\right),    
\end{aligned}
\end{equation}
where $A_{x,t,s}$ is given in \eqref{Axts} and $\varphi(s)=\exp(c_1(\log s)^{1/(2\wedge q)})$.

When $s<t<\tfrac{t_0}{2}$, we can `reverse the time' by considering the Lagrangian
\[
  \tilde L(x,t,v):=L(x_0+x,t_0-t,-v)
\]
which satisfies \eqref{eq:L-bd} and \eqref{eq:L-regularity} with $\nu$ replaced by
\[
  \tilde\nu(t):=\nu(t_0-t).
\]
We consider 
\beq\lb{revrese}
\tilde m(x,r)=\inf_{\substack{\alpha\in\AC([0,r],\R^d)\\\alpha(0)=0,\,\alpha(r)=x}}\int_0^r\tilde L(\alpha(\tau),\tau,\dot\alpha(\tau))\,\dd \tau.
\eeq
Then, $\gamma\in{\rm AC}([0,t_0],\R^d)$ being an optimal path for $m(x_0,t_0)$ is equivalent to $\tilde\gamma\in{\rm AC}([0,t_0],\R^d)$ defined by $\tilde\gamma(r):=\gamma(t_0-r)-x_0$ being an optimal path for $\tilde m(-x_0,t_0)$, and 
\[
  m^\gamma(s,t)=\tilde m^{\tilde\gamma}(t_0-t,t_0-s).
\]
Then, applying inequality \eqref{eq:path-reg-appl-pf1} to $(\tilde m,\tilde\gamma,\tilde\nu)$, we get, for $0<s<t<\tfrac{t_0}{2}$,
\begin{align*}
\MoveEqLeft \tilde m^{\tilde\gamma}(t_0-t,t_0-s)\\
&\lesssim (t-s)\exp\left(c_1(\log(\tfrac{t_0}{t-s}))^{1/(2\wedge q)}\,\right)\left(\frac{|x_0|^q}{t_0^q}+\sup_{w\ge t-s}\avint_{(t_0-s-w)\vee 0}^{t_0-s}\tilde\nu_r\,\dd r\right)\\
&=(t-s)\varphi\bigl(\frac{t_0}{t-s}\bigl)\left(\frac{|x_0|^q}{t_0^q}+\sup_{w\ge t-s}\avint_{s}^{(s+w)\wedge t_0}\nu_r\,\dd r\right).
\end{align*}	
This inequality, together with \eqref{eq:path-reg-appl-pf1}, yields the upper bound of \eqref{eq:path-reg-appl}.

It remains to show \eqref{eq:m-mui-bd}. 
To this end,  consider first the case $(i+2)\rho\ge\tfrac{t_0}{2}\geq 5.$ 
Then, it is clear that $i\rho \geq \tfrac{1}{2}(i+2)\rho$.
We have
\begin{align*}
m(x_0,t_0)
&=\bigl[m(0,0;\gamma({(i+2)\rho}),i\rho)+m(\gamma({(i+2)\rho}),(i+2)\rho;x_0,t_0)\bigr]\\
&\qquad+m^\gamma(i\rho,(i+2)\rho)
+[m(0,0;\gamma({i\rho}),i\rho)-m(0,0;\gamma({(i+2)\rho}),i\rho)] \\
&\stackrel{\eqref{eq:change-end-path-same-s}}\ge\mu_i-\int_{i\rho}^{(i+2)\rho}\nu_r\,\dd r-2C_q \rho \varphi\bigl(\frac{(i+2)}{2}\bigr)A_{\gamma({(i+2)\rho}),(i+2)\rho,i\rho}\\
&\stackrel{\eqref{eq:251228}}\ge \mu_i-\int_{i\rho}^{(i+2)\rho}\nu_r\,\dd r-C_{q}\varphi(t_0) D_{x_0,t_0,i\rho,(i+2)\rho}.
\end{align*}
When $(i+2)\rho<\tfrac{t_0}{2}$, using the time reversibility of $m$ as above, 
\begin{align*}
m(x_0,t_0)&=\bigl[m(0,0;\gamma({i\rho}),i\rho)+m(\gamma({i\rho}),(i+2)\rho;x_0,t_0)\bigr]+m^\gamma(i\rho,(i+2)\rho)\\
&\qquad+\bigl[m(\gamma({(i+2)\rho}),(i+2)\rho;x_0,t_0)-m(\gamma({i\rho}),(i+2)\rho;x_0,t_0)\bigr]\\
&\ge \mu_i-\int_{i\rho}^{(i+2)\rho}\nu_r\,\dd r\\
&\qquad+\bigl[\tilde m(\tilde\gamma({t_0-(i+2)\rho}),t_0-(i+2)\rho)-\tilde m(\tilde\gamma({t_0-i\rho}),t_0-(i+2)\rho)\bigr]\\
&\stackrel{\eqref{eq:change-end-path-same-s}}\ge \mu_i-\int_{i\rho}^{(i+2)\rho}\nu_r\,\dd r-C_{q}\varphi(t_0)\left(\frac{|x_0|^q}{t_0^q}+\sup_{w\ge 2\rho}\avint_{(t_0-i\rho-w)\vee 0}^{t_0-i\rho}\tilde\nu_r\,\dd r\right).
\end{align*}
These two inequalities imply, for any $i=0,\ldots, \floor{t_0}-2$,
\begin{equation}\label{eq:m-mui-lb}
    m(x_0,t_0)-\mu_i\gtrsim -\varphi(t_0) D_{x_0,t_0,i\rho,(i+2)\rho}.
\end{equation}
Furthermore, for any $y\in\R^d$, 
\begin{align*}
m(x_0,t_0)&\le m(y,i\rho)+m(y,i\rho;y,(i+2)\rho)+m(y,(i+2)\rho;x_0,t_0)\\
&\stackrel{\eqref{eq:m-upper-bd}}\le m(y,i\rho)+m(y,(i+2)\rho;x_0,t_0)+2\int_{i\rho}^{(i+2)\rho}\nu_r\,\dd r.
\end{align*}
Taking infimum over all $y\in\R^d$, we get
\begin{equation}\label{eq:m-mui-ub}
  m(x_0,t_0)\le \mu_i+2\int_{i\rho}^{(i+2)\rho}\nu_r\,\dd r.
\end{equation}
Combining \eqref{eq:m-mui-lb} and \eqref{eq:m-mui-ub}, inequality \eqref{eq:m-mui-bd} is proved.
\end{proof}

We are ready to prove Theorem~\ref{thm:fluctuation-mxt}.
Since it has three separate parts, we divide the proof into three parts for readability.
\begin{proof}[Proof of Theorem~\ref{thm:fluctuation-mxt}\eqref{item:algeb-integ-nu}]
Recall the definition of $M_i$ in \eqref{def:m_i}.
By Burkholder's inequality, for $n= \floor{t_0}$, we have
\[
\E|M_n-M_0|^p \;\le\; B_p \,\E \Bigl[\sum_{j=1}^n (M_j - M_{j-1})^2\Bigr]^{p/2}.
\]
Since $M_0=\E[m(x_0,t_0)]$, and $M_n=m(x_0,t_0)$ as $n= \floor{t_0}$, H\"{o}lder's inequality yields
\begin{align*}
\E|m(x_0,t_0)-\E m(x_0,t_0)|^p =\E|M_n-M_0|^p\lesssim t_0^{p/2-1}E\Bigl[\sum_{j=1}^n (M_j - M_{j-1})^p\Bigr].
\end{align*}
If assuming that
\begin{equation}\label{eq:delta-m-all-moment}
  \norm{M_{i+1}-M_i}_{L^p(\mb P)}\lesssim  \varphi(t_0)\bigl(\frac{|x_0|^q}{t_0^q}+C_0\log t_0\bigr), \quad i=0,1,\ldots, \floor{t_0}-1,
\end{equation}
then
\begin{align*}
&\bP\Bigl(|m(x_0,t_0)-\E m(x_0,t_0)|\ge \lambda \sqrt{t_0}\varphi(t_0)\bigl(\frac{|x_0|^q}{t_0^q}+C_0\log t_0 \bigr)\Bigr)\\
&\qquad\quad \leq \frac{\E|m(x_0,t_0)-\E m(x_0,t_0)|^p}{\lambda^p t_0^{p/2}\varphi(t_0)^p(\tfrac{|x_0|^q}{t_0^q}+C_0\log t_0 )^p}\lesssim \frac{1}{\lambda^p}.
\end{align*}
Thus, to show \eqref{item:algeb-integ-nu}, it suffices to prove \eqref{eq:delta-m-all-moment}.

\medskip

By \eqref{eq:forget-mui} and by the definition of $M_i$ in \eqref{def:m_i}, for $0\le i\le\floor{t_0}-2$,
\begin{align*}
M_{i+1}-M_i
&=\E[m(x_0,t_0)-\mu_i|\ms F_{(0,(i+1)\rho]}]-\E[m(x_0,t_0)-\mu_i|\ms F_{(0,i\rho]}].
\end{align*}
Hence, by Jensen's inequality, for $0\le i\le\floor{t_0}-2$,
\begin{align}\label{eq:delta-m-lp}
  \norm{M_{i+1}-M_i}_{L^p}\lesssim \norm{m(x_0,t_0)-\mu_i}_{L^p}
  \stackrel{\eqref{eq:m-mui-bd}}\lesssim \varphi(t_0)\norm{D_{x_0,t_0,i\rho,(i+2)\rho}}_{L^p}
\end{align}
\begin{align*}
\sup_{w\ge a}\avint_{(t-w)\vee 0}^t\nu_r\,\dd r\lesssim\sum_{i=0}^{\ceil{\log_2(t/a)}}\avint_{(t-2^ia)\vee 0}^t\nu_r\,\dd r.
\end{align*}
For any positive increasing convex function $h$ on $(0,\infty)$, by Jensen's inequality,
\begin{align}\label{eq:expec-ave}
\E h\Bigl(\frac{1}{\log_2(t/a)}\sup_{w\ge a}\avint_{(t-w)\vee 0}^t\nu_r\,\dd r\Bigr)
&\le 
\E h\Bigl(\frac{C}{\ceil{\log_2(t/a)}}\sum_{i=0}^{\ceil{\log_2(t/a)}}\avint_{(t-2^ia)\vee 0}^t\nu_r\,\dd r\Bigr)\nonumber\\
&\le \frac{1}{\ceil{\log_2(t/a)}}\sum_{i=0}^{\ceil{\log_2(t/a)}}\frac{1}{2^ia}\sum_{k=1}^{2^ia}\E h\Bigl(C\int_{(t-k)\vee 0}^{t-k+1}\nu_r\,\dd r\Bigr)\nonumber\\
&\le \E h\Bigl(C\int_{0}^{1}\nu_r\,\dd r\Bigr).
\end{align}
Applying this inequality to $h(x)=x^p$, we get, for any $a\ge 1$ and $a\le t\le t_0$,
\[
  \norm{\sup_{w\ge a}\avint_{(t-w)\vee 0}^t\nu_r\,\dd r}_{L^p}\lesssim 
  \log (t/a)\norm{\int_0^1\nu_r\,\dd r}_{L^p}\lesssim C_0\log t_0.
\]
Similarly, $\norm{\sup_{w\ge a}\avint_s^{(s+w)\wedge t_0}\nu_r\,\dd r}\lesssim {C_0}\log t_0$ for $0<s\le t_0-a$. We conclude that
\begin{equation}\label{eq:D-mmt-bd}
    \norm{D_{x_0,t_0,s,t}}_{L^p}\lesssim\frac{|x_0|^q}{t_0^q}+{C_0}\log t_0 
    \quad\text{ if }0\le s\le t-1<t\le t_0.
\end{equation}
This inequality, together with \eqref{eq:delta-m-lp}, yields, for $0\le i\le \floor{t_0}-2,$
\begin{equation*}
  \norm{M_{i+1}-M_i}_{L^p}\lesssim \varphi(t_0)\bigl(\frac{|x_0|^q}{t_0^q}+{C_0}\log t_0\bigr).
\end{equation*}

For $i=\floor{t_0}-1$, we have $M_{i+1}-M_i=m^\gamma(t_0-\rho, t_0)-\E[m^\gamma(t_0-\rho, t_0)|\ms F_{(0,t_0-\rho]}]$
and so
\begin{equation}\label{39}
  \norm{M_{i+1}-M_i}_{L^p}\lesssim \norm{m^\gamma(t_0-\rho, t_0)}_{L^p}
  \stackrel{\eqref{eq:path-reg-appl}}\lesssim \varphi(t_0) \norm{D_{x_0,t_0,t_0-\rho,t_0}}_{L^p}
\end{equation}
This inequality, together with \eqref{eq:delta-m-lp} and \eqref{eq:D-mmt-bd}, implies
\eqref{eq:delta-m-all-moment}. 
\end{proof}

We next prove Theorem~\ref{thm:fluctuation-mxt}\eqref{item:bdd-nu} as it follows more or less directly the above proof.

\begin{proof}[Proof of Theorem~\ref{thm:fluctuation-mxt}\eqref{item:bdd-nu}]
It follows from 
\eqref{eq:delta-m-lp} and \eqref{eq:D-mmt-bd} (with $p=1$) and the assumption $\int_0^1\nu_r\,\dd r<C_0$ that for $0\le i\le\floor{t_0}-2$,
\begin{align*}
\left|M_{i+1}-M_i\right|
&\lesssim \varphi(t_0) D_{x_0,t_0,i\rho,(i+2)\rho} \lesssim \varphi(t_0) \bigl(\frac{|x_0|^q}{t_0^q}+C_0 \log t_0\bigr).
\end{align*}
The inequality holds the same for $i=\floor{t_0}-1$ by \eqref{eq:path-reg-appl} and \eqref{39}. Then
Azuma's inequality yields
\begin{align*}
&\bP\Bigl(|m(x_0,t_0)-\E m(x_0,t_0)|\ge \lambda \sqrt{t_0}\varphi(t_0)\bigl(\frac{|x_0|^q}{t_0^q}+C_0 \log t_0\bigr)\Bigr)\\
&\qquad \le 2\exp \left(-{\frac {\lambda^{2}{t_0}\varphi(t_0)^2 (|x_0|^q/t_0^q+C_0\log t_0)^2}{2C\floor{t_0}  \varphi(t_0)^2 (|x_0|^q/t_0^q+C_0\log t_0)^2}}\right) \le 2\exp(-c\lambda^2).
\end{align*}
This yields the conclusion.
\end{proof}

Now, we prove Theorem~\ref{thm:fluctuation-mxt}\eqref{item:exp-integ-nu}, which is more involved.
\begin{proof}[Proof of Theorem~\ref{thm:fluctuation-mxt}\eqref{item:exp-integ-nu}]
 It follows from \eqref{eq:forget-mui}, \eqref{def:m_i}, and \eqref{eq:m-mui-bd} that for some $C_1$ and for all $0\le i\le\floor{t_0}-2$,
\begin{align*}
\left|M_{i+1}-M_i\right|
&=\left|\E[m(x_0,t_0)-\mu_i|\ms F_{(0,(i+1)\rho]}]-\E[m(x_0,t_0)-\mu_i|\ms F_{(0,i\rho]}]\right|\\
&\lesssim \E [ \varphi(t_0) D_{x_0,t_0,i\rho,(i+2)\rho}|\ms F_{(0,(i+1)\rho]}]+\E[ \varphi(t_0) D_{x_0,t_0,i\rho,(i+2)\rho}|\ms F_{(0,i\rho]}] .
\end{align*}
The same inequality holds for $i=\floor{t_0}-1$ by \eqref{eq:path-reg-appl} and \eqref{39}. 
By Jensen's inequality, it follows that
\beq\lb{3.16}
\E\exp (\delta |M_{i+1}-M_i|)\leq \E\left[\exp (C_1\delta \varphi(t_0) D_{x_0,t_0,i\rho,(i+2)\rho})\right].
\eeq

In the following, we estimate the right-hand side of the above. 
Recall that $\rho\in [1,2]$.
We compute
\begin{align*}
\sup_{w\ge 2\rho}  \avint_{((i+2)\rho-w)\vee 0}^{(i+2)\rho}\nu_r\,\dd r&\lesssim \sum_{k=0}^{\ceil{\log_2( (i+2)/2)}}\frac{1}{2^{k+1}}\int_{(i+2-2^{k+1})\rho\vee 0}^{(i+2)\rho}\nu_r\,\dd r\\
  &\lesssim  \sum_{j=0}^{i+1}\frac1{j+1}\int_{(i+1-j)\rho}^{(i+2-j)\rho}\nu_r\,\dd r.
\end{align*}
Similarly, 
\begin{align*}
\sup_{w\ge 2\rho}  \avint_{i\rho}^{(i\rho+w)\wedge t_0}\nu_r\,\dd r&\lesssim  \sum_{j\geq 0,\, (i+j)\rho\leq t_0}\frac1{j+1}\int_{(i+j)\rho}^{(i+1+j)\rho}\nu_r\,\dd r.
\end{align*}
Thus, for any $\delta>0$, we obtain
\begin{align*}
A_{i,\delta}    &:=\exp\left[C_1 \varphi(t_0)\delta \sup_{w\ge 2\rho} {\Bigl[} \avint_{((i+2)\rho-w)\vee 0}^{(i+2)\rho}\nu_r\,\dd r+\avint_{i\rho}^{(i\rho+w)\wedge t_0}\nu_r\,\dd r\Bigr]\right]\\
    &\leq \exp\left[CC_1 \varphi(t_0)\delta \left(\sum_{j=0}^{i+1}\tfrac1{j+1}\int_{(i+1-j)\rho}^{(i+2-j)\rho}\nu_r\,\dd r+\sum_{j\geq 0,\, (i+j)\rho\leq t_0}\tfrac1{j+1}\int_{(i+j)\rho}^{(i+1+j)\rho}\nu_r\,\dd r\right) \right]\\
    &\leq \frac12\exp\left[2CC_1 \varphi(t_0)\delta \left(\sum_{j=0}^{\ceil{t_0/2}}\tfrac1{|2j-i|+1}\int_{2j\rho}^{(2j+1)\rho}\nu_r\,\dd r\right) \right]\\
    &\qquad +\frac12\exp\left[2CC_1 \varphi(t_0)\delta \left(\sum_{j=0}^{\ceil{t_0/2}}\tfrac1{|2j-i|+1}\int_{(2j+1)\rho}^{(2j+2)\rho}\nu_r\,\dd r\right) \right].
\end{align*}

Next, note that the set of
$
\int_{i_k\rho}^{(i_k+1)\rho}\nu_r\,\dd r$ with $i_k\in\mathbb N$
are independent as long as $|i_k-i_{k'}|\geq \rho$ for all $k\neq k'$. Therefore, 
\begin{align*}
\E A_{i,\delta}    &\leq\frac12 \prod_{j=0}^{\ceil{t_0/2}} \E\exp\left[\tfrac{2CC_1 \varphi(t_0)\delta}{|2j-i|+1}\int_{2j\rho}^{(2j+1)\rho}\nu_r\,\dd r\right]\\
    &\qquad\qquad +\frac12\prod_{j=0}^{\ceil{t_0/2}} \E \exp\left[ \tfrac{2CC_1 \varphi(t_0)\delta}{|2j-i|+1}\int_{(2j+1)\rho}^{(2j+2)\rho}\nu_r\,\dd r \right]\\
    &\leq \E\exp \left[4CC_1 \varphi(t_0) (\log t_0)\delta\int_{0}^{\rho}\nu_r\,\dd r\right]\\
    &\leq \frac12\E\exp \left[8CC_1 \varphi(t_0) (\log t_0)\delta\int_{0}^{1}\nu_r\,\dd r\right]\\
    &\qquad\qquad+\frac12\E\exp \left[8CC_1 \varphi(t_0) (\log t_0)\delta\int_{1}^{2}\nu_r\,\dd r\right]\\
    &\leq \E\exp \left[8C C_1 \varphi(t_0) (\log t_0)\delta\int_{0}^{1}\nu_r\,\dd r\right].
\end{align*}
After taking $\delta$ to be sufficiently small depending on $t_0$ and $\tfrac{|x_0|^q}{t_0^q}$ such that 
\[
8C C_1 \varphi(t_0) (\log t_0)\delta\leq c\quad\text{ and }\quad C_1 \delta \varphi(t_0) \frac{|x_0|^q}{t_0^q}\leq 1,
\]
we obtain from \eqref{3.16} that
\[
\E\exp (\delta |M_{i+1}-M_i|)\leq \exp\left[  C_1 \delta \varphi(t_0) \tfrac{|x_0|^q}{t_0^q}\right]\E\exp\left[c\int_{0}^{1}\nu_r\,\dd r\right]\leq C_0e.
\]
It follows from Lemma \ref{thm:LV32} that for all $\lambda>0$
\begin{equation*}
 \bP\left(\delta|M_n-M_0|\geq \lambda n \right)
< C_*(\lambda) \exp\!\left(-4^{-2/3} \lambda^{2/3}n^{1/3}\right),
\end{equation*}
where $C_*(\lambda):=C_d(1+\frac{C_0}{\lambda^{2/3} \wedge \lambda^2})$ and $C_d$ is a dimensional constant.
Replacing $\lambda$ by $\lambda n^{-1/2}$ and taking 
\[
\delta=(C(c^{-1}\log t_0+\tfrac{|x_0|^q}{t_0^q})\varphi(t_0))^{-1}
\]
for some $C>0$ sufficiently large give  for $\lambda>0$,
\begin{equation*}
 \bP\left(|M_n-M_0|\geq C(c^{-1}\log t_0+\tfrac{|x_0|^q}{t_0^q})\varphi(t_0)\lambda n^{1/2}\right)
< C_*(\lambda n^{-1/2})\exp\!\left(-4^{-2/3} \lambda^{2/3}\right).
\end{equation*}
Suppose $n\geq 1$ and $C$ is sufficiently large.  If $\lambda\geq 1$, we have
\[
C_*(\lambda n^{-1/2})\exp\!\left(-4^{-2/3} \lambda^{2/3}\right)\leq C(1+C_0n)\exp\!\left(-4^{-2/3} \lambda^{2/3}\right)
\]
and if $\lambda\leq 1$, $C(1+C_0n)\exp\!\left(-4^{-2/3} \lambda^{2/3}\right)\geq 1$.

Now, using that $M_0=\E[m(x_0,t_0)]$ and $M_{\floor{t_0}}=m(x_0,t_0)$, and $n=\floor{t_0}$, we find that there exists $c_1>0$ such that for all $t_0\geq 2$ and $\lambda>0$,
\begin{align*}
&\bP\left(|m(x_0,t_0)-\E m(x_0,t_0)|\geq  \lambda t_0^{1/2}\varphi(t_0)(\tfrac{|x_0|^q}{t_0^q}+c^{-1}\log t_0) \right)\\
&\qquad\qquad
< C(1+C_0t_0)\exp\!\left(-c_1 \lambda^{2/3}\right)  .  
\end{align*}
Recall that $\varphi(s)= \exp\left(c(\log s)^{1/(2\wedge q)}\,\right)$ for $s\geq 1$. By replacing $c$ by $2c$ (then we can replace $\varphi(t_0)$ by $\varphi(t_0)^{1/2})$, and by using $\lambda$ in place of $\lambda\varphi(t_0)^{-1/2}$, we get
\begin{align*}
&\bP\left(|m(x_0,t_0)-\E m(x_0,t_0)|\geq  \lambda t_0^{1/2}\varphi(t_0)(\tfrac{|x_0|^q}{t_0^q}+c^{-1}\log t_0) \right)\\
&\qquad\qquad
< C(1+C_0t_0)\exp\!\left(-c_1 \lambda^{2/3}\varphi(t_0)^{1/3}\right) .
\end{align*}
Note that there exists $c>0$ such that $\varphi(t_0)\geq c(\log t_0)^4$ for all $t_0\geq 2$. So, if $\lambda\geq1$, we get for some $C=C(c_1)$,
\[
t_0\exp\!\left(-c_1 \lambda^{2/3}\varphi(t_0)^{1/3}\right)\leq C\exp\!\left(-c_1 \lambda^{2/3}\right).
\]
Therefore, for $\lambda\geq 1$,
\[
\bP\left(|m(x_0,t_0)-\E m(x_0,t_0)|\geq  \lambda t_0^{1/2}\varphi(t_0)(\tfrac{|x_0|^q}{t_0^q}+c^{-1}\log t_0) \right)
< C(1+C_0)\exp\!\left(-c_1 \lambda^{2/3}\right) .
\]
The same inequality certainly holds if $\lambda\leq 1$ if $C$ is large. Overall, we conclude the proof.
\end{proof}

\section{Error between the mean  and the ergodic average of the metric}\label{sec:deterministic-error-metric}

By the subadditive ergodic theorem, for a.s. $\omega\in \Omega$, we have
\begin{equation}\label{eq:subadd-ergodic}
  \ol m(x,t)=\lim_{r\to \infty}\frac{m(rx,rt,\omega)}{r}=\lim_{r\to \infty}\frac{m(0,0;rx,rt,\omega)}{r}.
\end{equation}
See \cite{MR2400607,MR2514380,MR3602941,MR3817561, Seeger}.
We have that $\ol m$ is subadditive and positively $1$-homogeneous, that is,
\[
\ol m(rx,rt)=r \ol m(x,t) \quad \text{ for } (x,t)\in \R^d \times (0,\infty), \ r>0.
\]
This implies that $(x,t)\mapsto \ol m(x,t)$ is convex as well.
Geometrically, the graph of $\ol m$ over $\R^d \times (0,\infty)$ forms part of a convex cone.

Our goal is to prove the following theorem.
\begin{thm}\label{thm:error-deterministic} Assume that
\begin{equation}\label{eq:gaussian-integ}
  \E \exp \Bigl[c\bigl( \int_{0}^{1} \nu_{r} \, \dd r\bigr) \Bigr] < \infty.
\end{equation}

For any $(x,t)\in\R^d\times [e,\infty)$,
    \begin{equation}\label{eq:Em-mbar}
    \ol m(x,t) \leq \E m(x,t) \leq \ol m(x,t) + Ct^{1/2}\psi(t)\bigl(\frac{|x|^q}{t^q}+1\bigr)\log\bigl(\frac{|x|}{t}+2\bigr)
    \end{equation}
    where $\psi(t):=\exp\left(4c_1(\log\log t)(\log t)^{1/(2\wedge q)}\,\right)$ is a slowly varying function for $t\ge e$.
\end{thm}
We note that $\psi(t) = (\varphi(t))^{4\log\log t}$ where $\varphi(t)$ is defined in \eqref{varphi}. 

Recall that, by \eqref{eq:m-upper-bd} and \eqref{eq:m-lower-bd}, for any $(x,t)\in\R^d\times{[1,\infty)}$, 
\[
\frac{|x|^q}{N_1 t^{q-1}}-Ct \le \ol m(x,t)\le  \E m(x,t)\le \frac{N_1 |x|^q}{t^{q-1}}+Ct. 
\]
By a constant translation of the Hamiltonian $H$, that is, by redefining $\tilde H:=H-C$ for some big constant $C>0$, then $\tilde L=L+C$, we have, for $(x,t)\in \R^d\times [1,\infty)$,
\begin{equation}\label{eq:wlog}
  \frac{|x|^q}{N_1 t^{q-1}} +t\leq \ol m(x,t)\le  \E m(x,t) \leq \frac{N_1|x|^q}{t^{q-1}}+ Ct.
\end{equation}
Since a constant translation of the Hamiltonian does not change $\ol m(x,t)-\E m(x,t)$, without loss of generality, we can assume \eqref{eq:wlog} throughout this section.

\medskip

For geometric interpretation, we extend $\ol m(x,t)$ to the whole $\R^{d+1}$ with $\ol m(x,-t)=\ol m(x,t)$ for $(x,t)\in \R^d\times (0,\infty)$.
For each $(x,t)\in \R^d\times (0,\infty)$, we let $\ell_{x,t}$ be the supporting hyperplane of $\ol m$ at $(x,t)$, that is, $\ell_{x,t}$ is linear, and
\[
\begin{cases}
    \ell_{x,t}(rx,rt)=\ol m(rx,rt)=r\ol m(x,t) \quad &\text{ for } r\geq 0,\\
    \ell_{x,t}(y,s) \leq \ol m(y,s) \quad &\text{ for } (y,s)\in \R^{d}\times [0,\infty).
\end{cases}
\]
It is clear that for $|y| \leq C_0 s$, $|\ell_{x,t}(y,s)| \leq (N_1 C_0^q+C)s$.
A fact is 
\[
\ol m(y,s)= \sup_{(x,t)\in \R^d\times (0,\infty)} \ell_{x,t}(y,s).
\]

Throughout this section, we let $M\ge 2$ be a large constant to be determined later. 

Note that for any $t\in [e,M]$, inequality \eqref{eq:Em-mbar} is trivial due to \eqref{eq:wlog}. Hence, it suffices to only consider $t\ge M$.

\subsection{An iterative argument}
We will adapt the iterative argument in \cite{Alexander-97} into our dynamic random Hamilton--Jacobi equation setting.

Recall that $\varphi(t) = \exp \left( c_{1} (\log t)^{\frac{1}{2\wedge q}} \right)$.
Set
\[
  \Lambda_{x,t}:=\bigl(\frac{|x|^q}{t^q}+1\bigr)\log\bigl(\frac{|x|}t+2\bigr).
\]

For $(x,t)\in \R^d\times[1,\infty)$, and $n\ge 1$, we denote the conic \emph{frustum}  with ``height" $[1,t]$ and ``opening" $n(\log t)^2 \, \varphi(t)^{1/q} \bigl( \tfrac{|x|}{|t|} + 1 \bigr)$ as
\begin{equation}\label{eq:def-cone}
  \cone_{x,t}(n)=\Bigl\{(y,s)\in\R^d\times[1,t]: |y| \le  n(\log t)^2 \, \varphi(t)^{1/q} \bigl( \frac{|x|}{t} + 1 \bigr)s\,\Bigr\}.
\end{equation}
In fact, we will only use $\cone_{x,t}(1)$ and $\cone_{x,t}(2)$ in our proof. 
For ease of notation, we simply write $\cone_{x,t}(1)$ as $\cone_{x,t}.$
\begin{defn}
	\label{def:chap}
For $(x,t)\in \R^d\times[M,\infty)$,  define the set $Q_{x,t}$ of increments with nice approximation properties as
\begin{multline}\label{eq:def-Qxt}
Q_{x,t} = \Big\{ (y, s) \in \cone_{x,t} :    
 \ell_{x,t}(y, s) \le \ol{m}(x, t), \\
 \E m(y, s) \le \ell_{x,t}(y, s) + t^{1/2} \varphi(t)^3 \Lambda_{x,t} \Big\}.
\end{multline}
\end{defn}
\begin{rem}
 Notice that for large enough $M\ge 1$ and $t\ge M$,  the sub-frustum
  \begin{equation}\label{eq:sub-frustum}
    F_{x,t}:=\{(y,s)\in\cone_{x,t}:1\le s\le 2\}\text{ is always a subset of }Q_{x,t}.
\end{equation}
  Indeed, for any $(y,s)\in F_{x,t}$, we have 
\begin{align}\label{eq:ok-subfrustum-fits-Q}
  \E m(y,s)\stackrel{\eqref{eq:wlog}}\le C {(|y|^q+1)}\le C|s|^q(\log t)^{2q}\varphi(t)\bigl(\frac{|x|^q}{t^q}+1\bigr)\le t^{1/4}\varphi(t)\bigl(\frac{|x|^q}{t^q}+1\bigr)
\end{align}
and 
\[
\ell_{x,t}(y, s)\le \E m(y,s)\stackrel{\eqref{eq:ok-subfrustum-fits-Q},\eqref{eq:wlog}}\le Ct^{-3/4}\varphi(t)\ol m(x,t)\le \ol m(x,t).
\]
Display \eqref{eq:sub-frustum} follows.
See Figure \ref{fig:cone}.
\end{rem}

\begin{figure}[htp!]
\begin{center}
\begin{tikzpicture}[scale=1.05, thick]
\def\xA{4}      
\def\yA{-1.3}   

\def\yTop{4}    
\def\yBot{0}    

\def\xL{0} \def\yL{4}
\def\xR{8} \def\yR{4}

\def\yRedBot{0.00}
\def\yRedTop{3.10}

\def\yInner{0.85}

\def\yBlue{2.25}

\coordinate (A) at (\xA,\yA);

\pgfmathsetmacro{\tTopL}{(\yTop-\yA)/(\yL-\yA)}
\pgfmathsetmacro{\xTopL}{\xA + \tTopL*(\xL-\xA)}
\pgfmathsetmacro{\tBotL}{(\yBot-\yA)/(\yL-\yA)}
\pgfmathsetmacro{\xBotL}{\xA + \tBotL*(\xL-\xA)}

\pgfmathsetmacro{\tRedBL}{(\yRedBot-\yA)/(\yL-\yA)}
\pgfmathsetmacro{\xRedBL}{\xA + \tRedBL*(\xL-\xA)}
\pgfmathsetmacro{\tRedTL}{(\yRedTop-\yA)/(\yL-\yA)}
\pgfmathsetmacro{\xRedTL}{\xA + \tRedTL*(\xL-\xA)}

\pgfmathsetmacro{\tInnerL}{(\yInner-\yA)/(\yL-\yA)}
\pgfmathsetmacro{\xInnerL}{\xA + \tInnerL*(\xL-\xA)}

\pgfmathsetmacro{\tTopR}{(\yTop-\yA)/(\yR-\yA)}
\pgfmathsetmacro{\xTopR}{\xA + \tTopR*(\xR-\xA)}
\pgfmathsetmacro{\tBotR}{(\yBot-\yA)/(\yR-\yA)}
\pgfmathsetmacro{\xBotR}{\xA + \tBotR*(\xR-\xA)}

\pgfmathsetmacro{\tRedBR}{(\yRedBot-\yA)/(\yR-\yA)}
\pgfmathsetmacro{\xRedBR}{\xA + \tRedBR*(\xR-\xA)}
\pgfmathsetmacro{\tRedTR}{(\yRedTop-\yA)/(\yR-\yA)}
\pgfmathsetmacro{\xRedTR}{\xA + \tRedTR*(\xR-\xA)}

\pgfmathsetmacro{\tInnerR}{(\yInner-\yA)/(\yR-\yA)}
\pgfmathsetmacro{\xInnerR}{\xA + \tInnerR*(\xR-\xA)}

\pgfmathsetmacro{\tBlue}{(\yBlue-\yA)/(\yR-\yA)}
\pgfmathsetmacro{\xBlue}{\xA + \tBlue*(\xR-\xA)}

\coordinate (TopL) at (\xTopL,\yTop);
\coordinate (TopR) at (\xTopR,\yTop);
\coordinate (BotL) at (\xBotL,\yBot);
\coordinate (BotR) at (\xBotR,\yBot);

\coordinate (RedBL) at (\xRedBL,\yRedBot);
\coordinate (RedBR) at (\xRedBR,\yRedBot);
\coordinate (RedTL) at (\xRedTL,\yRedTop);
\coordinate (RedTR) at (\xRedTR,\yRedTop);

\coordinate (InnerL) at (\xInnerL,\yInner);
\coordinate (InnerR) at (\xInnerR,\yInner);

\coordinate (BlueDot) at (\xBlue,\yBlue);

\draw[black, thick] (TopL) -- (BotL);
\draw[black, thick] (TopR) -- (BotR);
\draw[black, thick] (TopL) -- (TopR);
\draw[black, thick] (BotL) -- (BotR);

\draw[dashed] (BotL) -- (A) -- (BotR);

\filldraw[red!18, draw=red, thick]
(RedBL) -- (RedTL)
.. controls ($(RedTL)!0.35!(RedTR) + (0,0.35)$) and ($(RedTL)!0.65!(RedTR) + (0,0.20)$)
.. (RedTR)
-- (RedBR) -- cycle;

\begin{scope}
  \clip (RedBL) -- (RedTL)
  .. controls ($(RedTL)!0.35!(RedTR) + (0,0.35)$) and ($(RedTL)!0.65!(RedTR) + (0,0.20)$)
  .. (RedTR) -- (RedBR) -- cycle;
  \foreach \k in {-2,-1.7,...,12}{
    \draw[red, very thin] (\k,\yRedBot-0.6) -- ++(2.4,6.0);
  }
\end{scope}

\draw[black, thick] (InnerL) -- (InnerR);

\draw[->, red, thick]
($(RedTL)!0.20!(RedBL) + (-0.9,0.2)$)
to[bend left=10]
($(RedTL)!0.45!(RedTR) + (-0.2,-0.7)$);

\node[left, red] at ($(RedTL)!0.35!(RedBL) + (-1.0,0.4)$) {$Q_{x,t}$};
\node[left] at ($(BotL)!0.55!(A) + (0.7,1)$) {$F_{x,t}$};

\fill[blue] (BlueDot) circle (0.12);

\end{tikzpicture}
\end{center}
\caption{An example of $Q_{x,t}$ and $F_{x,t}$}\label{fig:cone}
\end{figure}
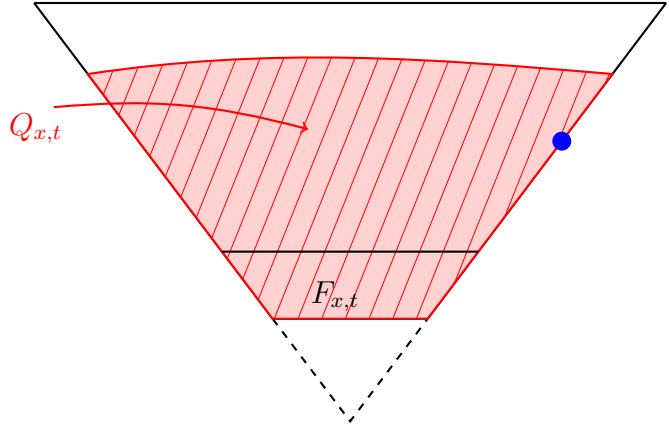

\begin{defn}[Skeleton]\label{def:skeleton}
For any set $S\subset\R^d\times\R_+$ and $(y,s)\in \R^d\times\R^+$, we say that a sequence  $(0,0)=(y_0,t_0),\ldots,(y_k,t_k)=(y,s)$ with $0=t_0<t_1<\cdots<t_k$ is an $S$-\emph{skeleton}  from $(0,0)$ to $(y,s)$ (with $k+1$ vertices) if 
\[
  (y_i-y_{i-1},t_i-t_{i-1})\in S \quad\text{ for all }\quad i=1,\ldots,k.
\] 
\end{defn}

\begin{prop}\label{prop:CHAP-GAP}
Let $M,a>1$.
Assume that the following statement is true:
\begin{equation}
\label{eq:assum-skeleton}
\left.
\begin{aligned}
  &\text{For each $(x,t)\in\R^d\times[M,\infty)$, there exist $n\ge 1$ and a $Q_{x,t}$-skeleton}\\
  &\text{with $k+1$ vertices from $(0,0)$ to $(nx,nt)$ such that $k\le an$.}
  \end{aligned}\ \ 
  \right\}
  \end{equation}
Then, inequality \eqref{eq:Em-mbar} holds for all  $t\geq M$, with $C=C(M,a)>0$.
\end{prop}

Observe that, for any $Q_{x,t}$-skeleton $\{(y_i,t_i)\}_{i=0}^k$ from $(0,0)$ to $(nx,nt)$, with $n\ge 1$, we must have
\begin{equation}\label{eq:k-ge-n}
  k\ge n.
\end{equation}
Indeed, by the $1$-homogeneity and subadditivity of $\ol m$, we get
    \[
    n\ol m(x,t)= \ol m(nx,nt) =\ell_{x,t}(nx,nt) = \sum_{i=1}^{k}\ell_{x,t}(v_i,t_{i}-t_{i-1})\leq k \ol m(x,t),
    \]
which implies \eqref{eq:k-ge-n}.

\begin{lem}\label{lem:iteration-GAP}
Let $M>1$ be large enough.	Assume that statement \eqref{eq:assum-skeleton} is true for $a>1$.  There exists $\tilde C=\tilde C(M,a)>0$ such that,
if for some $\beta\in(\tfrac{1}{2},1]$ and $r\in\Z_{\ge 0}$,
\begin{equation}\label{eq:m-itera-1}
  \E m(x,t)\le \ol m(x,t)+\tilde Ct^\beta\varphi(t)^r\Lambda_{x,t} \quad \text{ for all }(x,t)\in\R^d\times[M,\infty),
\end{equation}
then, with $\hat \beta:=\tfrac{\beta}{\tfrac{1}{2}+\beta}\in(\tfrac{1}{2},\beta)$, we have
	\begin{equation}\label{eq:m-itera-2}
  \E m(x,t)\le \ol m(x,t)+\tilde Ct^{\hat \beta}\varphi(t)^{r+3}\Lambda_{x,t} \quad \text{ for all }(x,t)\in\R^d\times[M,\infty).
\end{equation}
\end{lem}

\begin{proof}
Recall that $C>0$ denotes a generic constant that depends only on $(d,q,N_1,\nu)$  whose value may differ from line to line.

Let $C_3\ge 4d+4$ be a constant to be determined, and set $\tilde C=C_3^2$.

	We divide the proof into two cases.

\medskip

\noindent {\bf  Case 1: $t>C_3\tilde C M^{\tfrac{1}{2}+\beta}$}.
 Let $\rho=C_3t^\theta$, where 
\[
  \theta:=\frac{\beta-\tfrac{1}{2}}{\beta+\tfrac{1}{2}}\in(0,\tfrac{1}{3}].
\]

 Note that $t/\rho=t^{1-\theta}/C_3>C_3^{-\theta}\tilde C^{1-\theta}M^{(\tfrac{1}{2}+\beta)(1-\theta)}\ge M$. 
Applying \eqref{eq:assum-skeleton} to $(x/\rho,t/\rho)$ and by \eqref{eq:k-ge-n}, there exist $n\ge 1$ and a $Q_{x/\rho,t/\rho}$-skeleton $\{(v_i,t_i)\}_{i=0}^k$ from $(0,0)$ to $(x/\rho,t/\rho)$ such that $n\le k\le an$. We have
\[
\frac{n}{k\rho}(x,t)=\sum_{i=1}^k \frac{1}{k}(v_i,t_{i}-t_{i-1}) 
\]
which lies in the convex hull of $Q_{x/\rho,t/\rho}\subset\R^{d+1}$. Hence, by Carath\'eodory's theorem, there exist $d+2$ points  $\{(y_j,s_j)\}_{j=1}^{d+2}$ within $Q_{x/\rho,t/\rho}$ such that their convex hull contains $\tfrac{n}{k\rho}(x,t)$, that is,
\[
  \frac{n}{k\rho}(x,t)=\sum_{j=1}^{d+2}\alpha_j (y_j,s_j)
\]
for some $\alpha_j\in[0,1]$ with $\sum_{j=1}^{d+2}\alpha_j=1$. Thus $(x,t) = \sum_{j=1}^{d+2} \tfrac{k\rho\alpha_j}{n} (y_j,s_j)$. Note that $\sum_{j=1}^{d+2} \tfrac{k\rho\alpha_j}{n}\ge \rho\ge C_3>2(d+2).$ Hence $\tfrac{k\rho\alpha_j}{n}\ge 2$ for some $j$. Say, $\tfrac{k\rho\alpha_1}{n}\ge 2$,
 \begin{align*}
(x,t) &=\sum_{j=1}^{d+2} \floor{\tfrac{k\rho\alpha_j}{n}-\mathbbm 1_{j=1}} (y_j,s_j)+\sum_{j=1}^{d+2} (\tfrac{k\rho\alpha_j}{n} -\floor{\tfrac{k\rho\alpha_j}{n}}+\mathbbm 1_{j=1}) (y_j,s_j)\\
&=:(x^*,t^*) + \sum_{j=1}^{d+2} \gamma_j (y_j,s_j).
\end{align*}
Clearly, $1\le \sum_{j=1}^{d+2} \gamma_j\le d+3$, and $t-t^*\ge \gamma_1 s_1\ge 1$. 

\medskip

We claim that
\begin{equation}\label{eq:large-t-bound-1}
\E m(x^*,t^*) \leq \ell_{x,t}(x^*,t^*)  + aC_3^{1/2}t^{\hat\beta}\varphi(t)^3\Lambda_{x,t}.
\end{equation}
Indeed, by the subadditivity of $m$ and the definition of $Q_{x/\rho,t/\rho}$, we get
\begin{align*}
\E m(x^*,t^*)&\le \sum_{j=1}^{d+2}\floor{\tfrac{k\rho\alpha_j}{n}-\mathbbm 1_{j=1}} \E m(y_j,s_j)\\
&\le \sum_{j=1}^{d+2}\floor{\tfrac{k\rho\alpha_j}{n}-\mathbbm 1_{j=1}} \left[\ell_{x/\rho,t/\rho}(y_j,s_j)+(\tfrac{t}{\rho})^{1/2} \varphi(\tfrac{t}{\rho})^3 \Lambda_{x,t}\right]\\
&\le \ell_{x,t}(x^*,t^*)+a(\rho t)^{1/2}\varphi(t)^3 \Lambda_{x,t},
\end{align*}
where in the last inequality we used $k\le an$ and the linearity of $\ell_{x,t}=\ell_{x/\rho,t/\rho}$.
Recalling that $\rho=C_3t^\theta$ and $t^{(1+\theta)/2}=t^{\hat\beta}$, we obtain
\[
\E m(x^*,t^*) \leq \ell_{x,t}(x^*,t^*) + aC_3^{1/2}t^{\hat\beta}\varphi(t)^3\Lambda_{x,t},
\]
which is exactly \eqref{eq:large-t-bound-1}.

Next, we claim that 
\begin{equation}\label{eq:m-remainder-control}
  \E m(x-x^*,t-t^*)\le \ol m(x-x^*,t-t^*)+\tilde C C_3^{-1/2}t^{\hat\beta}\varphi(t)^{r+2}\Lambda_{x,t}.
\end{equation}
To this end, note that
\begin{equation}\label{eq:t-t*}
   1\le  t-t^*\ \leq \sum_{j=1}^{d+2} (1+\mathbbm 1_{j=1})s_j\leq \frac{(d+3)t}{\rho}.
\end{equation}
Further, since every $(y_j,s_j)$ belongs to $\cone_{x/\rho,t/\rho}$ which is convex, 
\[
  (x-x^*,t-t^*)=\sum_{j=1}^{d+2}\gamma_j(y_j,s_j)\in (\sum_{j=1}^{d+2}\gamma_j)\cone_{x/\rho,t/\rho}.
\]
Thus, by the definition \eqref{eq:def-cone} of $\cone_{x/\rho,t/\rho}$,
\begin{align}\label{eq:x-x*}
\frac{|x-x^*|^q}{(t-t^*)^q}\le (2\log t)^{2q}\varphi(t)\bigl(\frac{|x|^q}{t^q}+1\bigr).
\end{align}
Notice that \eqref{eq:x-x*} implies (note that $t\ge M$ and $M$ is sufficiently large) 
\begin{equation}\label{eq:big-Lambda}
\begin{aligned}
\Lambda_{x-x^*,t-t^*}&\le [(2\log t)^{2q}\varphi(t)(\tfrac{|x|^q}{t^q}+1)+1]\log[(\log 2t)^{2}\varphi(t)^{1/q}(\tfrac{|x|}{t}+1)+2]\\
&\leq  [(2\log t)^{2q}+1]\frac{\log[(\log 2t)^{2}\varphi(t)^{1/q}(\tfrac{|x|}{t}+1)+2]}{\log(|x|/t+2)}\varphi(t)\Lambda_{x,t}\\
&\leq  (\log t)^{2q+2}\varphi(t)\Lambda_{x,t}.
\end{aligned}
\end{equation}

We will prove claim \eqref{eq:m-remainder-control} by considering two sub-cases.
The first subcase is when $t-t^* \geq M$.
We apply the hypothesis \eqref{eq:m-itera-1} of the lemma to get
\begin{align*}
\E m(x-x^*,t-t^*)-\ol m(x-x^*,t-t^*)\le \tilde C(t-t^*)^\beta\varphi(t)^r\Lambda_{x-x^*,t-t^*}.
\end{align*}
This inequality, together with \eqref{eq:t-t*}, \eqref{eq:x-x*}, \eqref{eq:big-Lambda}, implies
\begin{align*}
 \E m(x-x^*,t-t^*)-\ol m(x-x^*,t-t^*)
  &\le C\tilde C \bigl(\frac{t}{\rho}\bigr)^\beta(\log t)^{2q+2}\varphi(t)^{r+1}\Lambda_{x,t}\\
  &\le \tilde C C_3^{-1/2}t^{\hat\beta}\varphi(t)^{r+2}\Lambda_{x,t}.
\end{align*}
Thus claim \eqref{eq:m-remainder-control} is proved for $t-t^*\ge M$.
It remains to consider the subcase where $t-t^*<M$. Since $t-t^*\ge 1$, by \eqref{eq:wlog} and \eqref{eq:x-x*}, 
\begin{align*}
  \E m(x-x^*,t-t^*)&\le C(t-t^*)(\log t)^{2q}\varphi(t)\bigl(\frac{|x|^q}{t^q}+1\bigr)\nonumber\\
  &\le \tilde CC_3^{-1/2}\varphi(t)^2\bigl(\frac{|x|^q}{t^q}+1\bigr)
\end{align*}
where the last inequality can be achieved by choosing $C_3\ge M^2$. Noting that \eqref{eq:wlog} guarantees $\ol m(x-x^*,t-t^*)\ge 0$, our proof of claim \eqref{eq:m-remainder-control} is complete. 

Furthermore, by the subadditivity and 1-homogeniety of $\ol m$, 
\begin{align*}
\ol m(x-x^*,t-t^*)&\le \sum_{j=1}^{d+2}\gamma_j\ol m(y_j,s_j)\\
&\le \sum_{j=1}^{d+2}\gamma_j \E m(y_j,s_j)\\
&\le \sum_{j=1}^{d+2}\gamma_j [\ell_{x/\rho,t/\rho}(y_j,s_j)+(\tfrac{t}{\rho})^{1/2}\varphi^3(\tfrac t\rho)\Lambda_{x,t}]\\
&\le \ell_{x/\rho,t/\rho}(x-x^*,t-t^*)+(d+3)C_3^{1/2}t^{\hat\beta}\varphi(t)^{3}\Lambda_{x,t},
\end{align*}
where in the last inequality we used the fact that $(y_j,s_j)\in Q_{x/\rho,t/\rho}$ and the definition of $Q_{x/\rho,t/\rho}$. This inequality, together with \eqref{eq:m-remainder-control},
yields
\begin{equation}\label{eq:m-remainder-replace-by-ell}
  \E m(x-x^*,t-t^*)\le \ell_{x,t}(x-x^*,t-t^*)+2\tilde C C_3^{-1/2}t^{\hat\beta}\varphi(t)^{r+3}\Lambda_{x,t}.
\end{equation}

Combining \eqref{eq:large-t-bound-1}, \eqref{eq:m-remainder-replace-by-ell}, and recalling that $\ell_{x,t}\le \ol m$, we have
\[
  \E m(x,t)\le \ol m(x,t)+3C_3^{-1/2}\tilde Ct^{\hat \beta}\varphi(t)^{r+3}\Lambda_{x,t}.
\]
We obtain \eqref{eq:m-itera-2} for $t>C_3\tilde C M^{\tfrac{1}{2}+\beta}$.

\noindent {\bf  Case 2: $M \leq t \leq C_3\tilde C M^{\tfrac{1}{2}+\beta}$.}
This is the case when $t$ is bounded.

For this case, we use the basic upper bound in \eqref{eq:wlog} to get
\begin{align*}
\E m(x,t)&\le Ct\bigl(\frac{|x|^q}{t^q}+1\bigr)\\
&\le C(C_3\tilde C M^{\tfrac{1}{2}+\beta})^{1-\hat\beta}t^{\hat\beta}\bigl(\frac{|x|^q}{t^q}+1\bigr)\\
&\le (CC_3\tilde C^{-2})^{1/(2\beta+1)}\tilde Ct^{\hat\beta}\bigl(\frac{|x|^q}{t^q}+1\bigr).
\end{align*}
Choosing $C_3$ so that $CC_3\le \tilde C^{2}$, inequality \eqref{eq:m-itera-2} follows.
\end{proof}

\begin{proof}[Proof of Proposition~\ref{prop:CHAP-GAP}]
Let $\tilde C$ be as in Lemma~\ref{lem:iteration-GAP}. By \eqref{eq:wlog}, 
\[
  \E m(x,t)\le\tilde Ct\bigl(\frac{|x|^q}{t^q}+1\bigr)\quad \text{ for all }(x,t)\in\R^d\times[M,\infty),
\]
and so \eqref{eq:m-itera-1} holds with $\beta=1$ and $r=0$. Without loss of generality, consider $M\ge 10$. 
Set $\beta_0=1$, and define inductively $ \beta_{j+1}=\tfrac{2\beta_j}{2\beta_j+1}\in(\tfrac{1}{2},\beta_j)$, we have
\[
\beta_j=\frac{1}{2-2^{-j}}\quad\text{ for all }j\in\Z_{\ge 0}.
\]
By Lemma~\ref{lem:iteration-GAP} and induction, for all $j\in\Z_{\ge 0}$ and all $(x,t)\in\R^d\times[M,\infty)$,
\[
    \E m(x,t)\le \ol m(x,t)+\tilde Ct^{\beta_j}\varphi(t)^{3j}\Lambda_{x,t}.
\]
Note that $0<\beta_j-\tfrac{1}{2}\le 2^{-j-1}$ for all $j\in\Z_{\ge 0}$.

Letting $J=J(t)\in\N$	 be such that $2^J\le \log t<2^{J+1}$, we have
\begin{align*}
t^{\beta_J-\tfrac{1}{2}}\varphi(t)^{3J}&\le\exp\left(2^{-J-1}\log t+3c_1J(\log t)^{1/(2\wedge q)}\right)\\
&\le \exp\left(4c_1(\log\log t)(\log t)^{1/(2\wedge q)}\right).
\end{align*}
Therefore, \eqref{eq:Em-mbar} holds.
\end{proof}

\subsection{Good skeletons}
Throughout this subsection, we fix $n=n(t)=t$. Set $t_0=0$ and 
\begin{equation}
  \label{eq:def-ti}
  t_i:=nt-\floor{nt}+i \quad\text{ for }i\ge 1.
\end{equation}
Clearly, $t_i-t_{i-1}\equiv 1$ for all $2\le i\le\floor{nt}$, and $t_1-t_0=nt-\floor{nt}+1\in[1,2)$.

By \eqref{eq:sub-frustum}, it is not hard to construct $Q_{x,t}$-skeletons from $(0,0)$ to $(nx,nt)$ with increments in $F_{x,t}$. However, to prove Theorem~\ref{thm:error-deterministic},  Proposition~\ref{prop:CHAP-GAP} tells us that it is crucial to have a skeleton with at most $Cn$ segments. Motivated by this, we define \emph{good} skeletons below, which have as few segments as possible.

\begin{defn}[good skeleton]\label{def:good-skeleton} Recall the sequence $(t_i)_{i=0}^{\floor{nt}}$ in \eqref{eq:def-ti}. A $Q_{x,t}$-skeleton $\{(v_j,s_j)\}_{j=0}^k$ from $(0,0)$ to $(nx,nt)$ is called a \emph{good $Q_{x,t}$-skeleton} if 
\begin{enumerate}[(i)]
  \item $(s_j)_{j=0}^k$ is a subsequence of $(t_i)_{i=0}^{\floor{nt}}$;
  \item for each $1\le j\le k-1$, there exists $(y,1)\in F_{x,t}$ such that:
\begin{itemize}
  \item  $|v_j-v_{j-1}+y|\le (s_j-s_{j-1}+1)(\log t)^2\varphi(t)^{1/q}(\tfrac{|x|}{t}+1)$;
  \item $(v_j-v_{j-1}+y,s_j-s_{j-1}+1)\notin Q_{x,t}.$
\end{itemize}
\end{enumerate}
	\end{defn}

To find a good skeleton with fewer than $Cn$ segments, we need each segment to be highly ``stretched", in the sense that $\E m-\ol m$ is small. 
The idea in \cite{Alexander-97} is that the skeletons extracted from optimal curves are the most stretched, thanks to the exponential concentration in Theorem~\ref{thm:fluctuation-mxt} together with the ergodic property \eqref{eq:subadd-ergodic}.

\medskip

The goal of this subsection is to present the following corollary.

\begin{cor}\label{cor:good-skeleton} Let $M\ge 1$.
We let $\Gamma_1=\Gamma_1(x,t,n)$ denote the event that every good $Q_{x,t}$-skeleton $\{(v_j,s_j)\}_{j=0}^k$ from $(0,0)$ to $(nx,nt)$, if exists, satisfies
	\[
  \sum_{j=1}^k|\E m(v_{j-1},s_{j-1};v_j,s_j)-m(v_{j-1},s_{j-1};v_j,s_j)|\le {3(\log M)^{-1}} kt^{1/2}\varphi(t)^3\Lambda_{x,t}.
\]
There exists $M\ge 1$ large enough such that for any $(x,t)\in\R^d\times[M,\infty)$, 
\[
  \bbP(\Gamma_1(x,t,n))\ge 0.9.
\]
\end{cor}

In contrast to the classical FPP model, where one can always extract a skeleton from an arbitrary path, the lack of a deterministic Lipschitz bound means that an optimal path need not remain inside any fixed space-time cone. Consequently, one cannot in general extract a skeleton from an optimal path using a deterministic family of bounded-slope increments.
While no Lipschitz control is available at the microscopic scale, Theorem~\ref{thm:path-reg} provides a (random) Lipschitz regularity at the $O(1)$-scale, which allows us to find, with high probability, a good $Q_{x,t}$-skeleton along an optimal path.

\begin{lem}\label{lem:nice-segments} 
Recall the sequence $(t_i)_{i=0}^{\floor{nt}}$ in \eqref{eq:def-ti}.
Let $\Gamma_2=\Gamma_2(x,t,n)$ denote the event that there exists an $\omega$-optimal path $\gamma$ from $(0,0)$ to $(nx,nt)$ such that, for all $i,j\in\{1,\ldots,\floor{nt}\}$,
\begin{equation}\label{eq:extremely-good-skeleton}
  |\gamma(t_i)-\gamma(t_{j})|\le |t_i-t_j|(\log t)^2\varphi(t)^{1/q}\bigl(\frac{|x|}{t}+1\bigr).
\end{equation}

	There exists a constant $M\ge 1$ such that for any $t\ge M$, 
	\begin{align*}
\bbP(\Gamma_2(x,t,n))\ge 0.9.
\end{align*}
\end{lem}

Here, if $\gamma$ satisfies \eqref{eq:extremely-good-skeleton}, then $(\gamma(t_j)-\gamma(t_i),t_j-t_i)\in \cone_{x,t}$ for $i<j<i+t$.
In other words, $(\gamma(t_j),t_j)$ stays in the conic frustum $(\gamma(t_i),t_i)+\cone_{x,t}$ of vertex $(\gamma(t_i),t_i)$ for all $i<j<i+t$.

\begin{proof}
	If $\Gamma_2$ does not happen, then, for every optimal path $\gamma$ from $(0,0)$ to $(nx,nt)$, there exists a pair of indices $i<j$ within $\{1,\ldots,\floor{nt}\}$ such that 
\[
 |\gamma(t_i)-\gamma(t_{j})|> (t_j-t_i)(\log t)^2  \varphi(t)^{1/q} \bigl( \frac{|x|}{t} + 1 \bigr).
\]
Notice that $\varphi(nt)=\varphi(t^2)\le\varphi(t)^2$. 
By Proposition~\ref{prop:path-metric-reg}, for any $i<j$ in $\{1,\ldots,\floor{nt}\}$,
\[
  |\gamma(t_i)-\gamma(t_{j})|\lesssim (t_j-t_i)\varphi(t)^{1/q}D_{nx,nt,t_i,t_j}^{1/q}
\]
where $D_{nx,nt,t_i,t_j}$ is from the proposition.
Hence $\omega\notin\Gamma_2(x,t,n)$ implies the existence of a pair $i<j$ such that
\begin{align*}\label{eq:no-gamma}
(\log t)^{2q}\bigl(\frac{|x|^q}{t^q}+1\bigr)\lesssim \frac{|x|^q}{t^q}+\sup_{w\ge 1}\Bigl[\avint_{(t_j-w)\vee 0}^{t_j}\nu_r\,\dd r+\avint_{t_{i}}^{(t_{i}+w)\wedge(nt)}\nu_r\,\dd r\Bigr].
\end{align*}
For $t\ge M$, this implies
\[
c(\log t)^{2q}\le \sup_{w\ge 1}\Bigl[\avint_{(t_j-w)\vee 0}^{t_j}\nu_r\,\dd r+\avint_{t_{i}}^{(t_{i}+w)\wedge(nt)}\nu_r\,\dd r\Bigr],
\]
which we denote by event $B_{i,j}$. Thus,
\begin{align}\label{eq:bad-i}
\bbP(\Gamma_2^c(x,t,n))\le\sum_{i,j=1}^{\floor{nt}}\bbP(B_{i,j}).
\end{align}

For any $i<j$, by Chebyshev's inequality and inequality \eqref{eq:expec-ave} (applied to $h(x)=\exp(cx)$), 
\begin{align*}
\bbP(B_{i,j})&\le 
\bbP\left(\sup_{w\ge 1}\Bigl[\avint_{(t_j-w)\vee 0}^{t_j}\nu_r\,\dd r+\avint_{t_{i}}^{(t_{i}+w)\wedge(nt)}\nu_r\,\dd r\Bigr]\ge \frac{c}2(\log t)^{2q}\right)\\
&\le \E\exp\Bigl[c(\int_0^1\nu_r\,\dd r)\Bigr]\exp[-C(\log t)^{2q-1}],
\end{align*}
where we used the fact $\log(nt)\lesssim\log t$.
Hence, by \eqref{eq:gaussian-integ},
\[
  \bbP(B_{i,j})\lesssim
  \exp(-c(\log t)^{2q-1}). 
\]
This inequality, together with \eqref{eq:bad-i}, yields 
\begin{align*}
\bbP(\Gamma_2^c)\lesssim (nt)^2\exp(-c(\log t)^{2q-1})\lesssim t^4\exp(-c(\log t)^{2q-1}).
\end{align*}
Since $2q-1>1$, choosing $M>1$ sufficiently large, we have $\bbP(\Gamma_2^c)<0.1$ for all $t\ge M$. Lemma~\ref{lem:nice-segments} is proved.   
\end{proof}

\begin{defn}[lattice $\cone(2)$-skeleton]
Recall the sequence $(t_i)_{i=0}^{\floor{nt}}$ in \eqref{eq:def-ti} and the notation $\cone_{x,t}(2)$ in \eqref{eq:def-cone}. A $\cone_{x,t}(2)$-skeleton $\{(z_j,s_j)\}_{j=0}^k$ from $(0,0)$ to $(nx,nt)$ is said to be a \emph{lattice} $\cone_{x,t}(2)$-skeleton if
\begin{itemize}
  \item $z_j\in\Z^d$ for all $0\le j\le k-1$;
 \item $(s_j)_{j=0}^k$ is an increasing subsequence of $(t_i)_{i=0}^{\floor{nt}}$ with $s_0=0, s_k=nt$.
\end{itemize}
\end{defn}

\begin{prop}\label{prop:cone-walk}
	For every lattice $\cone_{x,t}(2)$-skeleton $\mathfrak 
	S=\{(z_j,s_j)\}_{j=0}^k$ from $(0,0)$ to $(nx,nt)$, we define, for $1\le j\le k$, the random variables
	\begin{align*}
	  \MoveEqLeft Y_j=Y_j(\mathfrak S)=\sup_{u,v}\Bigl\{|\E m(u,s_{j-1};v,s_j)-m(u,s_{j-1};v,s_j)|:\\ 
	 &\qquad \qquad 1\le\norm{u-z_{j-1}}_\infty\le 2,1\le \norm{v-z_{j}}_\infty\le 2\Bigr\}.
\end{align*}
There exists $M\ge 2$ such that for any $(x,t)\in\R^d\times[M,\infty)$, with probability at least $0.9$, every lattice $\cone_{x,t}(2)$-skeleton $\mathfrak S$ from $(0,0)$ to $(nx,nt)$ satisfies
\begin{equation}\label{eq:every-walk-nice}
    \sum_{j=1}^{k}Y_j\le 3\iota kt^{1/2}\varphi(t)^3\Lambda_{x,t},
\end{equation}
where $\iota:=(\log M)^{-1}$ and $k=k(\mathfrak S)$ is the number of segments in $\mathfrak S$.
\end{prop}

\begin{proof}
We define an event $A$ where $(\nu_r)_{r\ge 0}$ has good averaging property as
\[
    A=\left\{\int_{(nt-i-1)\vee 0}^{nt-i} \nu_r\,\dd r \leq (\log t)^{2},\, \forall\, i=0,\ldots,\floor{nt}\right\}.
\]
Note that by the union bound and Chebyshev's inequality,
\begin{align}\label{eq:prob-ac}
\MoveEqLeft\bbP(A^c)\le\sum_{i=0}^{\floor{nt}}\bbP\Bigl(\int_{(nt-i-1)\vee 0}^{nt-i} \nu_r\,\dd r > (\log t)^{2}\Bigr)\\
&\lesssim nt\exp[-c(\log t)^2]\ \E\exp\Bigl[c(\int_0^1\nu_r\,\dd r)\Bigr]\stackrel{\eqref{eq:gaussian-integ}}\lesssim \exp[-C(\log t)^2].\nonumber
\end{align}
\noindent {\bf Step 1.} 	For every lattice $\cone_{x,t}(2)$-skeleton $\mathfrak S=(z_j,s_j)_{j=0}^k$,
 let 
\[
  \zeta_j=|\E m(z_{j-1},s_{j-1};z_j,s_j)-m(z_{j-1},s_{j-1};z_j,s_j)|
\]
 and write $\Delta z_j=z_j-z_{j-1}$, $\Delta s_j=s_j-s_{j-1}$. Note that
 \[
m(z_{j-1},s_{j-1};z_j,s_j,\omega) =m(z_j-z_{j-1},s_j-s_{j-1};T_{z_{j-1},s_{j-1}}\omega),
 \] 
 and if $1\le\norm{u-z_{j-1}}_\infty\le 2$ and $M$ is large enough,
 \[
2\varphi(\Delta s_j) \ge |u-z_{j-1}|\varphi\bigl(\frac{\Delta s_j}{|u-z_{j-1}|}\bigr)\geq 1.
 \]
The same holds with $u$ and $z_{j-1}$ replaced by $v$ and $z_j$, respectively.
 Hence, by applying Theorem~\ref{thm:sp-reg-t} twice, we have
 \begin{equation*}
 \begin{aligned}
\MoveEqLeft |m(z_{j-1},s_{j-1};z_j,s_j) - m(u,s_{j-1};v,s_j)| \\
    &\lesssim\varphi(\Delta s_j)  
\left(\frac{|\Delta z_j|^q}{|\Delta s_j|^q}+\sup_{\ell\ge 1}\avint_{(\Delta s_j-\ell)\vee 0}^{\Delta s_j}\nu_{r+s_{j-1}}\,\dd r+\avint_{0}^{\ell\wedge\Delta s_j}\nu_{r+s_{j-1}}\,\dd r\right)\\
&\lesssim\varphi(t)\bigl(\frac{|\Delta z_j|^q}{|\Delta s_j|^q}+\eta_j\bigr), 
 \end{aligned}
\end{equation*}
where 
\beq\lb{4.23}
\begin{aligned}
\eta_j:=\sup_{\ell\ge 1}\avint_{(\Delta s_j-\ell)\vee 0}^{\Delta s_j}\nu_{r+s_{j-1}}\,\dd r+\avint_{0}^{\ell\wedge\Delta s_j}\nu_{r+s_{j-1}}\,\dd r\lesssim \sup_{i\in\bbN} \int_{(s_j-i-1)\vee s_{j-1}}^{ s_j-i}\nu_{r}\,\dd r.
\end{aligned}
\eeq

Since $(\Delta z_j,\Delta s_j)\in\cone_{x,t}(2)$, we have
\begin{equation}\label{eq:ratio-zs-bd}
  \frac{|\Delta z_j|^q}{|\Delta s_j|^q}\lesssim (\log t)^{2q}\varphi(t)\bigl(\frac{|x|^q}{t^q}+1\bigr).
\end{equation}
Note that, by \eqref{eq:expec-ave}, we also have $\E[\eta_j]\lesssim\log(\Delta s_j)\le\log t$. 
Hence
\begin{align*}
Y_j
&\le \zeta_j+C\varphi(t)^2(\log t)^{2q}\bigl(\frac{|x|^q}{t^q}+1\bigr)+C\varphi(t)\eta_j.
\end{align*}
In particular, since $\eta_j\lesssim (\log t)^{2}$ on the event $A$ by \eqref{4.23}, we have
\[
  Y_j\mathbbm 1_A\le \zeta_j+C\varphi(t)^3\bigl(\frac{|x|^q}{t^q}+1\bigr).
\]
Thus, recalling \eqref{eq:prob-ac} that $\bbP(A)>0.5$, we get
\begin{align}\label{eq:two-terms}
 \MoveEqLeft \mb P\left(\text{\eqref{eq:every-walk-nice} fails for $\mathfrak S$}\big|A\right)\nonumber\\
  &\le \mb P\left(\sum_{j=1}^k\bigl[\zeta_j+C\varphi(t)^3\bigl(\frac{|x|^q}{t^q}+1\bigr)\bigr]\ge 2 \iota kt^{1/2}\varphi(t)^3\Lambda_{x,t}\bigg|A\right)\nonumber\\
  &\le \mb P\left(\sum_{j=1}^k\zeta_j\ge  \iota kt^{1/2}\varphi(t)^3\Lambda_{x,t}\bigg|A\right)
\end{align}
 where we used the fact that $t\ge M$ and $M$ is large enough.

\medskip

\noindent {\bf Step 2.} 
To bound the probability in \eqref{eq:two-terms}, we will show a moment bound \eqref{eq:mmt-bd-zetaprime} in a modified environment.
Recall that $(s_j)\subset(t_i)_{i=0}^{\floor{nt}}$ and consider events
\begin{align*}
  A_j&=\left\{ \int_{(nt-i-1)\vee 0}^{nt-i} \nu_r\,\dd r \leq (\log t)^{2},\,\forall\, nt-i\in (s_{j-1},s_j]\right\}.
\end{align*}
Then $A\subset\bigcap_{j=1}^k A_j$. 
Let us define a modified Lagrangian $\tilde L$, for $(y,r,v) \in \R^d\times \R \times \R^d$, as
\[
\tilde L(y,r,v,\omega):=
\begin{cases}
 \E[L(0,0,v,\cdot)] & \text{ if }r\in (s_{j-1},s_j]\text{ and $A_j^c$ happens},1\le j\le k,\\
 L(y,r,v,\omega) & \text{ otherwise.}
\end{cases}
\]
Then $\tilde L$ is convex in $v$. Moreover, letting 
\[
\tilde\nu_r(\omega):=
\begin{cases}
  \E[\nu_r(\cdot)]& \text{ if }r\in (s_{j-1},s_j]\text{ and $A_j^c$ happens}, 1\le j\le k,\\
\nu_r(\omega) & \text{ otherwise},
\end{cases}
\]
then the ensemble $\{\tilde L(y,r,v), \tilde\nu_r:(y,r,v)\in\R^{d}\times \R\times \R^d\}$ satisfies (A2), and its law has a unit-range of dependence in time. Clearly, 
\[
\int_s^{s+1}\tilde\nu_r\,\dd r\le 2(\log t)^{2}
\]
$\bbP$-a.s. for all $s>0$, since $t\ge M$ and $M$ is large enough.

We also define $\tilde m$ correspondingly.

It follows from Theorem \ref{thm:fluctuation-mxt} (c) that there exists $c_1>0$ independent of  $j$ such that for all $1\le j\le k$ and any $\lambda>0$,
\[
\bP\left(\tilde \zeta_j\ge \lambda (\Delta s_j)^{1/2}\varphi(\Delta s_j)\bigl(\frac{|\Delta z_j|^q}{\Delta s_j^q}+2(\log t)^{2}\log \Delta s_j\bigr)\right)\le {2}\exp(-c_1\lambda^2),
\]
where $\tilde \zeta_j:=|\tilde m(z_{j-1},s_{j-1};z_j,s_j)-\E \tilde m(z_{j-1},s_{j-1};z_j,s_j)|$.
Since $\Delta s_j\le t$, setting
\[
\tilde \zeta_j':=\frac{\left|\E \tilde m(z_{j-1},s_{j-1};z_j,s_j)-\tilde m(z_{j-1},s_{j-1};z_j,s_j)\right|}{t^{1/2}\varphi(t)(\tfrac{|\Delta z_j|^q}{\Delta s_j^q}+2(\log t)^{2}\log t)},
\]
it follows that $\E\exp (4\tilde \zeta_j' )\leq C$.
Hence, by the Cauchy-Schwarz inequality,
\begin{equation}\label{eq:mmt-bd-zetaprime}
    \E\exp (2\sum_{j=1}^k\tilde \zeta_j' )\le \E\Bigl[\exp (4\sum_{j\text{ even }}\tilde \zeta_j' )\Bigr]+\E\Bigl[\exp (4\sum_{j\text{ odd }}\tilde \zeta_j' )\Bigr]\le C^k,
\end{equation}
where in the last inequality we used the fact that $(\tilde\zeta_j'   )_{j\text{ is even}}$ and $(\tilde\zeta_j'   )_{j\text{ is odd}}$ are both independent sequences of random variables.

\medskip

\noindent {\bf Step 3.} Next, we will derive from \eqref{eq:mmt-bd-zetaprime} an exponential moment bound \eqref{eq:mmt-bd-zeta-at-a} for $\sum_{j=1}^k\zeta_j$. 
Observe that $\tilde L=L$, $\tilde \nu=\nu$, and $\tilde m=m$ on the event $A$. Hence
\begin{align}\label{eq:em-etildem}
&\bigl|\E\tilde m(z_{j-1},s_{j-1};z_j,s_j)-\E m(z_{j-1},s_{j-1};z_j,s_j)\bigr|\nonumber\\
\le \ &  \E\Bigl[\bigl(|\tilde m(z_{j-1},s_{j-1};z_j,s_j)|+|m(z_{j-1},s_{j-1};z_j,s_j)|\bigr)\mathbbm 1_{A^c}\Bigr]\nonumber\\
\stackrel{\eqref{eq:wlog}}\lesssim \ & 
t\bigl(\frac{|\Delta z_j|^q}{\Delta s_j^q}+1\bigr)\bbP(A^c)\stackrel{\eqref{eq:prob-ac}}\lesssim \bigl(\frac{|\Delta z_j|^q}{\Delta s_j^q}+1\bigr)\exp[-c(\log t)^2].
\end{align}

Thus, setting 
\begin{equation}\label{eq:def-zeta-prime}
  \zeta'_j:=\frac{\zeta_j}{t^{1/2}\varphi(t)(\tfrac{|\Delta z_j|^q}{\Delta s_j^q}+2(\log t)^{2}\log t )},
\end{equation}
and noticing that $\tilde m=m$ on the event $A$, we obtain
\begin{equation}\label{eq:mmt-bd-zeta-at-a}
  \E\bigl[\exp (\sum_{j=1}^k\zeta_j' )\mathbbm 1_{A}\bigr]
  \stackrel{\eqref{eq:em-etildem}}\le \E\bigl[\exp (2\sum_{j=1}^k\tilde \zeta_j')\mathbbm 1_{A}\bigr]
  \stackrel{\eqref{eq:mmt-bd-zetaprime}}\le C^k.
\end{equation}

\noindent {\bf Step 4.} We can use \eqref{eq:mmt-bd-zeta-at-a} to control the first probability in \eqref{eq:two-terms}.
Indeed, notice that, for $t\ge M$ and $M$ is large enough, the denominator in \eqref{eq:def-zeta-prime} is bounded by 
\[
  t^{1/2}\varphi(t)\bigl(\frac{|\Delta z_j|^q}{\Delta s_j^q}+2(\log t)^{2}\log t \bigr)\stackrel{\eqref{eq:ratio-zs-bd}}\lesssim t^{1/2}\varphi(t)^3\bigl(\frac{|x|^q}{t^q}+1\bigr)(\log t)^{-3}.
\]
Display \eqref{eq:mmt-bd-zeta-at-a}, $\iota=(\log M)^{-1}\geq (\log t)^{-1}$, and Chebyshev's inequality imply
\begin{align*}
\bbP\bigl(\sum_{j=1}^k\zeta_j\ge \iota kt^{1/2}\varphi(t)^3\Lambda_{x,t}\big|A\bigr)
&\lesssim C^k\exp\bigl(-c\iota k(\log t)^3\log(\tfrac{|x|}{t}+2)\bigr)\\
&\lesssim \exp\bigl(-ck(\log t)^2\log(\tfrac{|x|}{t}+2)\bigr).
\end{align*}
This inequality, together with \eqref{eq:two-terms}, yields
\begin{align*}
\mb P\left(\text{\eqref{eq:every-walk-nice} fails for $\mathfrak S$}\big|A\right)
\lesssim \exp\bigl(-ck(\log t)^2\log(\tfrac{|x|}{t}+2)\bigr).
\end{align*}

Since for each $k\ge 1$, there are at most $\bigl(Ct^{d+3}(\tfrac{|x|^d}{t^d}+1)\bigr)^k$  different lattice
$\cone_{x,t}(2)$-skeletons with $k+1$ vertices, a union bound further yields
\begin{align*}
 \MoveEqLeft \bbP(\text{\eqref{eq:every-walk-nice} fails for some lattice $\cone_{x,t}(2)$-skeleton with $k+1$ vertices}\big|A)\\
 &\lesssim 
\exp\Bigl(Ck\log t \log(\tfrac{|x|}{t}+2)-ck(\log t)^2\log(\tfrac{|x|}{t}+2)\Bigr)\\
&\lesssim 
\exp(-k\log t\log(\tfrac{|x|}{t}+2)).
\end{align*}
Finally, summing over all $k\ge 1$, 
\begin{align*}
\bbP(\text{\eqref{eq:every-walk-nice} fails for some lattice $\cone_{x,t}(2)$-skeleton}\big|A)\lesssim\exp(-\log t\log(\tfrac{|x|}{t}+2)).
\end{align*}
    Therefore, recalling that $\bbP(A^c)\stackrel{\eqref{eq:prob-ac}}\lesssim\exp[-C(\log t)^2]$, we conclude that
\[
  \bbP(\text{\eqref{eq:every-walk-nice} fails for some lattice $\cone_{x,t}(2)$-skeleton})\le C\exp(-\log t\log2)
\]
which is less than $0.1$ for all $t\ge M$.
\end{proof}

\begin{proof}[Proof of Corollary~\ref{cor:good-skeleton}]
	For every good $Q_{x,t}$-skeleton $\{(v_j,s_j)\}_{j=0}^k$ from $(0,0)$ to $(nx,nt)$, we can find $z_j\in\Z^d$ such that $1\le \norm{v_j-z_{j}}_\infty\le 2$ for all $1\le j\le k-1$. Then, with $z_0=0$, $z_k=nt$, and $M>1$  sufficiently large,  the sequence $(z_j,s_j)_{j=0}^k$ is a lattice $\cone_{x,t}(2)$-skeleton. Indeed, it suffices to verify $(\Delta z_j,\Delta s_j) \in \cone_{x,t}(2)$, i.e., 
    \[
    |\Delta z_j| \le  2(\log t)^2 \, \varphi(t)^{1/q} \bigl( \frac{|x|}{t} + 1 \bigr)\Delta s_j,
    \]
    which is clearly justified by the definition of a good skeleton.
    The Corollary follows immediately from Proposition~\ref{prop:cone-walk}.
	\end{proof}

\subsection{Proof of Theorem~\ref{thm:error-deterministic}}

Finally, we are ready to prove Theorem~\ref{thm:error-deterministic}. The key ingredient is the following Proposition, which guarantees the existence of a skeleton from $(0,0)$ to $(nx,nt)$ with fewer than $Cn$ segments. 

\begin{prop}\label{prop:skeleton-10n-points}
    There exists $M>1$ large enough such that, for $(x,t)\in\R^d\times[M,\infty)$, we can find $n\in[1,t]$ and a $Q_{x,t}$-skeleton from $(0,0)$ to $(nx,nt)$ with at most $
    10n$ segments.
\end{prop}

\begin{proof}[Proof of Theorem~\ref{thm:error-deterministic}:]
	The theorem follows from Propositions~\ref{prop:CHAP-GAP} and \ref{prop:skeleton-10n-points}.
\end{proof}

To prove Proposition~\ref{prop:skeleton-10n-points}, we need two lemmas.

\begin{lem}\label{lem:em-nbarm-skeleton}
	There exists $M\ge 1$ large enough such that for any $(x,t)\in\R^d\times[M,\infty)$,  $n=t$,
	there exists a good $Q_{x,t}$-skeleton $\{(y_i,s_i)\}_{i=0}^k$ from $(0,0)$ to $(nx,nt)$ with the property (recall that $\iota =(\log M)^{-1}$)
\begin{equation}\label{eq:ideal-skeleton-prop}
    \sum_{i=1}^k \E m(v_{j-1},s_{j-1};v_j,s_j)\le n\ol m(x,t)+6\iota kt^{1/2}\varphi(t)^3\Lambda_{x,t}.
\end{equation}
\end{lem}

\begin{proof}
We let $\mathbb S(x,t,n)$ denote the collection of all
good $Q_{x,t}$-skeletons $\{(y_i,s_i)\}_{i=0}^k$ from $(0,0)$ to $(nx,nt)$ with  property \eqref{eq:ideal-skeleton-prop}. Set
\[
  Z(\{(y_i,s_i)\}_{i=0}^k):=\sum_{i=1}^k \E m(v_{i-1},s_{i-1};v_i,s_i).
  \]

We need to show that $\mathbb S(x,t,n)\neq\emptyset$ for any $(x,t)\in\R^d\times[M,\infty)$.

To this end, we let $\Gamma_3(x,t,r)$ be the event $\{\omega:m(0,0;rx,rt)\le r\ol m(x,t)+r\}$.
By \eqref{eq:subadd-ergodic}, $\lim_{r\to\infty} \bbP(\Gamma_3(x,t,r))=1$. We take $M>1$ large enough so that \[
   \bbP(\Gamma_3(x,t,r))\ge 0.9 \quad\text{for all } r\ge M.
\]

We let $\omega$ be any environment within $\Gamma_1(x,t,n)\cap \Gamma_2(x,t,n)\cap \Gamma_3(x,t,n)$. Recall that $\Gamma_1 $ and $ \Gamma_2$ are given in Corollary~\ref{cor:good-skeleton} and Lemma \ref{lem:nice-segments}, respectively.

We claim that there exists an optimal path $\gamma$ from $(0,0)$ to $(nx,nt)$ such that $\gamma$ contains a good $Q_{x,t}$-skeleton.

 Indeed, since $\omega\in\Gamma_2(x,t,n)$, we can find
 an optimal path $\gamma$ from $(0,0)$ to $(nx,nt)$ such that $(\gamma(t_i)-\gamma(t_{i-1}),t_i-t_{i-1})\in F_{x,t}$ for all $i=1,\ldots,\floor{nt}$. We define a sequence $\{(v_j,s_j)\}_{j=0}^k$ as follows. Let $(v_0,s_0)=(0,0)$. Given $(v_j,s_j)$, if $s_j<nt$, we let 
	\[
  s_{j+1}=\min\{t_i-1:t_i>s_j\,\text{ and }\,(\gamma(t_i)-v_j,t_i-s_j)\notin Q_{x,t}\}\wedge (nt)
\]
and $v_{j+1}=\gamma(s_{j+1})$.
Here we used the convention $\min\emptyset=\infty$. Thanks to $F_{x,t}\subset Q_{x,t}$ and property \eqref{eq:extremely-good-skeleton}, we have that $\{(v_j,s_j)\}_{j=0}^k$ is a well-defined subsequence of $(\gamma(t_i),t_i)_{i=0}^{\floor{nt}}$ and $\{(v_j,s_j)\}_{j=0}^k$ is a good $Q_{x,t}$-skeleton from $(0,0)$ to $(nx,nt)$. 

Moreover, by the definition of $\Gamma_1$ in Corollary~\ref{cor:good-skeleton}, $\{(v_j,s_j)\}_{j=0}^k$ satisfies
\begin{align*}
  Z(\{(y_i,s_i)\}_{i=0}^k)&\le \sum_{j=1}^km(v_{j-1},s_{j-1};v_j,s_j)+5\iota kt^{1/2}\varphi(t)^3\Lambda_{x,t}\\
  &=m(0,0;nx,nt)+5\iota kt^{1/2}\varphi(t)^3\Lambda_{x,t}\\
  &\le n+n\ol m(x,t)+5\iota kt^{1/2}\varphi(t)^3\Lambda_{x,t}
\end{align*}
where we also used the definition of $\Gamma_3$ in the last inequality. 
Noticing that the sequence $(v_j,s_j)_{j=0}^k\subset\gamma$ is a random sequence (with $k\ge n$ increments), we get
\begin{align*}
\MoveEqLeft  \bbP(\exists\text{ an $\omega$-optimal path who has a subsequence $\{(v_j,s_j)\}_{j=0}^k$ within $\mathbb S(x,t,n)$})\\
&\ge \bbP(\Gamma_1\cap\Gamma_2\cap\Gamma_3)\ge 0.7>0.
\end{align*}
Since $\mathbb S(x,t,n)$ is a deterministic set, we conclude that $\mathbb S(x,t,n)\neq\emptyset$.	
\end{proof}

Define $\varepsilon_{x,t}(y,s){:=}\E m(y,s)-\ell_{x,t}(y,s)$ as the \emph{inefficiency} of $(y,s).$

Recall Definition~\ref{def:good-skeleton}. For every good $Q_{x,t}$-skeleton $\{(v_j,s_j)\}_{j=0}^k$ from $(0,0)$ to $(nx,nt)$, its increments $(\Delta v_j,\Delta s_j)_{j=1}^{k-1}$ can be categorized by three sets:
\begin{align*}
\cB^E_{x,t}&=\{(y,s)\in Q_{x,t}:s\ge t-1\},\\
\cB^L_{x,t}&=\{(y,s)\in Q_{x,t}:\, \exists (z, 1)\in F_{x,t},\,\ell_{x,t}(y+z,s+1) >\ol m(x,t)\},\\
\cB^S_{x,t}&=\{(y,s)\in Q_{x,t}:\exists (z, 1)\in F_{x,t},\varepsilon_{x,t}(y+z,s+1)\ge  t^{1/2} \varphi(t) \Lambda_{x,t}\,\}.
\end{align*}
Note that the last increment $(\Delta v_k,\Delta s_k)$ may not fall into any of these categories.

\begin{lem}\label{lem:short-long-edges}
    There exists $M>1$ large enough such that, for $x\in\R^d$ and $t\geq M$, we have the following properties.
    \begin{enumerate}[(a)]
        \item\label{item:unit-cost} If $(y,s)\in F_{x,t}$, then $\E m(y,s) \leq C(\log t)^{2q}\varphi(t)(\tfrac{|x|^q}{t^q}+1)$.
        \item\label{item:long-edges} If $(y,s)\in \cB^L_{x,t}$, then $\ell_{x,t}(y,s) \geq \tfrac{1}{2} \ol m(x,t)$.
        \item\label{item:short-edges} If $(y,s)\in \cB^S_{x,t}$, then $\varepsilon_{x,t}(y,s) \geq \tfrac{1}{2}t^{1/2}\varphi(t)\Lambda_{x,t}$.
    \end{enumerate}
\end{lem}
\begin{proof}\eqref{item:unit-cost} This is a direct consequence of \eqref{eq:wlog} and the definition \eqref{eq:sub-frustum} of $F_{x,t}$.

\noindent\eqref{item:long-edges} If $(y,s)\in \cB^L_{x,t}$, then there exists $(z,1)\in F_{x,t}$ such that 
\begin{align*}
\ell_{x,t}(y,s)&=\ell_{x,t}(y+z,s+1)-\ell_{x,t}(z,1)\\
&\ge \ol m(x,t)-C(\log t)^q\varphi(t)\bigl(\frac{|x|^q}{t^q}+1\bigr)\ge \frac{1}{2}\ol m(x,t)
\end{align*}
where in the first inequality we used $\ell_{x,t}(z,1)\le \ol m(z,1)\le \E m(z,1)$ and \eqref{item:unit-cost}, and the second inequality is due to \eqref{eq:wlog}.
\medskip

\noindent\eqref{item:short-edges} If $(y,s)\in \cB^S_{x,t}$, then there exists $(z,1)\in F_{x,t}$ such that
\begin{align*}
\varepsilon_{x,t}(y,s)&\ge \varepsilon_{x,t}(y+z,s+1)-\varepsilon_{x,t}(z,1)\\
&\ge t^{1/2}\varphi(t)\Lambda_{x,t}-C(\log t)^{2q}\varphi(t)\bigl(\frac{|x|^q}{t^q}+1\bigr)\\
&\ge \frac{1}{2}t^{1/2}\varphi(t)\Lambda_{x,t}.
\end{align*}	
where in the first inequality we applied the subaddivitity of $\bbE m$, and in the second inequality we used \eqref{item:unit-cost} and the fact $\varepsilon_{x,t}(z,1)\le \E m(z,1)$.
\end{proof}

\begin{proof}[Proof of Proposition~\ref{prop:skeleton-10n-points}]
First, consider the case $\varphi(t)^3\log(\tfrac{|x|}{t}+2)\ge C_5t^{1/2}$. In this case, by \eqref{eq:def-Qxt}, we have
\[
 \Bigl\{(y,s)\in\cone_{x,t}:\ell_{x,t}(y,s)\le\ol m(x,t), \E m(y,s)\le\ell_{x,t}(y,s)+C_5t\bigl(\frac{|x|^q}{t^q}+1\bigr)\Bigr\}\subset Q_{x,t}.
\]
By \eqref{eq:wlog}, we can choose $C_5$ appropriately so that $(x,t)\in Q_{x,t}$. Thus, with $n=1$, $\{(0,0), (x,t)\}$ is the desired skeleton of two vertices. The Proposition follows.

\medskip

It remains to consider the case
\begin{equation}\label{eq:case-x-not-huge}
  \varphi(t)^3\log\bigl(\frac{|x|}{t}+2\bigr)<C_5 t^{1/2}.
\end{equation}
To this end, by Lemma~\ref{lem:em-nbarm-skeleton}, we can take $n=t$ and let $\{(y_i,s_i)\}_{i=0}^k$ be a $Q_{x,t}$-skeleton from $(0,0)$ to $(nx,nt)$ with the property \eqref{eq:ideal-skeleton-prop}. 

Note that, by stationarity, 
	\[
  \E m(v_{j-1},s_{j-1};v_j,s_j)=\E m(\Delta v_j,\Delta s_j).
\]
By \eqref{eq:ideal-skeleton-prop}, \eqref{eq:case-x-not-huge}, and \eqref{eq:wlog},
\begin{equation}\label{eq:260207}
  \sum_{j=1}^k\E m(\Delta v_j,\Delta s_j){\le n\ol m(x,t)+C\iota kt\bigl(\frac{|x|^q}{t^q}+1\bigr)}\le (n+C\iota {k})\ol m(x,t).
\end{equation}

Define the total number of increments in $\cB^L_{x,t}$ by
    \[
    I^L=\{1\leq i\le k\,:\, (\Delta v_i,\Delta s_{i})\in \cB^L_{x,t}\}.
    \]
Similarly we define $I^E, I^S$ for $\cB_{x,t}^E$ and $\cB_{x,t}^S$, respectively. Note that 
\begin{equation}\label{eq:e-number}
  |I^E|\le \frac{nt}{t-1}\le 2n.
\end{equation}

By Lemma~\ref{lem:short-long-edges}\eqref{item:long-edges}, 
\begin{align*}
\sum_{j=1}^k\E m(\Delta v_j,\Delta s_j)\ge \sum_{j\in I^L}\ell_{x,t}(\Delta v_j,\Delta s_j)\ge \frac12|I^L|\ol m(x,t).
\end{align*}
This inequality, together with \eqref{eq:260207}, yields
\begin{equation}\label{eq:l-number}
    \frac12|I^L|\le n+C\iota k.
\end{equation}

Moreover, by Lemma~\ref{lem:short-long-edges}\eqref{item:short-edges},
    \begin{align*}
        \sum_{j=1}^k \E m(\Delta v_j,\Delta s_j) &= \sum_{j=1}^k \left(\ell_{x,t}(\Delta v_j,\Delta s_j)+\varepsilon_{x,t}(\Delta v_j,\Delta s_j)\right)\\
        &\geq \ell_{x,t}(nx,nt) +\frac12|I^S|t^{1/2}\varphi(t)\Lambda_{x,t}\\
        &=n\ol m(x,t)+\frac12|I^S|t^{1/2}\varphi(t)\Lambda_{x,t}.
    \end{align*}
This inequality, together with \eqref{eq:ideal-skeleton-prop}, yields
\begin{equation}\label{eq:s-number}
    \frac{1}{2}|I^S|\le 6\iota k.
\end{equation}
Combining \eqref{eq:e-number}, \eqref{eq:l-number}, \eqref{eq:s-number}, we obtain
\[
  k\le 1+|I^E|+|I^L|+|I^S|\le {4}n+C\iota k.
\]
Recall that $\iota =(\log M)^{-1}$. For $M$ big enough, $C\iota\le\tfrac{1}{2}$, and 
this completes our proof of Proposition~\ref{prop:skeleton-10n-points}.    
\end{proof}

As a direct consequence of Theorems \ref{thm:fluctuation-mxt} and \ref{thm:error-deterministic}, we have the convergence of $m(x,t)\to\overline{m}(x,t)$ as  $t\to\infty$, where $\ol m$ is a deterministic constant given in \eqref{eq:subadd-ergodic}.

\begin{cor}\lb{C.4.10}
Assume {\rm(A1)--(A2)} and \eqref{eq:gaussian-integ}. 
Then, for any $(x,t)\in\R^d\times[e,\infty)$ and $\lambda>0$, we have
\[
\bP\left(\left|\frac{m(x,t)-\ol m(x,t)}{t}\right|\geq  \lambda t^{-1/2}\psi(t)(\tfrac{|x|^q}{t^q}+1)\log(\tfrac{|x|}t+2)\right)\le C\exp\!\left(-c \lambda^{2/3}\right).
\]
\end{cor}

\section{Proof of the main results}\label{sec:proof of main results}
Without loss of generality, we assume \eqref{eq:wlog} in this section.
Recall that $\ol m$ is subadditive and positively $1$-homogeneous, and \eqref{eq:olLdef} says
\[
\ol L(v) = \ol m(v,1) \quad \text{ for } v\in \R^d.
\]
We first note that the effective Lagrangian satisfies the following growth property.

\begin{prop} 
Assume \eqref{eq:wlog}.
Then, there exists $C>0$ such that
\begin{equation}\label{eq:bound-L-bar}
  \frac{1}{N_1}|v|^q+1 \le \ol{L}(v) \le N_1 |v|^{q} +C.
\end{equation}
As a direct result, for each $p\in \R^d$, there exists a positive number $\tilde C_p$, so that
\begin{equation}
\label{eq:olHformula}
  \ol{H}(p) = \sup_{|v|\le  \tilde C_p} (p\cdot v - \ol{L}(v)) = \sup_{|v|\le  \tilde C_p} (p\cdot v - \ol{m}(v,1)).
\end{equation}
\end{prop}
\begin{proof}
We have that \eqref{eq:bound-L-bar} follows directly from \eqref{eq:wlog}.
Equation \eqref{eq:olHformula} follows directly from the bounds in \eqref{eq:bound-L-bar}.
\end{proof}

Besides, in light of the Legendre transform \eqref{eq:olHformula} above, there exists $C>0$ such that
\begin{equation}\label{eq:H bar superlinear}
\frac{1}{C}|p|^{q'} -C \leq \ol H(p) \leq C|p|^{q'} + C \quad \text{ for all } p \in \R^d.
\end{equation}
In particular, $\ol H$ is coercive, which implies that $\ol u$, the solution to \eqref{eq:effHJ}, is globally Lipschitz.

\subsection{Proof of Theorem \ref{thm:main-1}}
We assume the settings of Theorem \ref{thm:main-1} in this subsection.
Now we quantify the rate of convergence in 
\begin{equation*}
  \lim_{t\to \infty} \frac{u(x,t,\omega)}{t} = -\ol{H}(0).
\end{equation*}
The representation formulas for the original equation \eqref{eq:HJ} and its effective equation read:
\begin{equation*}
    u(x,t,\omega) = \inf_{y\in \R^d} \left( g(y) + m(y,0;x,t,\omega) \right)
\end{equation*}
and 
\begin{equation*}
    \ol u(x,t) = \inf_{y\in \R^d} \left( g(y) + t\ol L\bigl(\frac{x-y}{t}\bigr) \right) = \inf_{y\in \R^d} \left( g(y) + \ol m(x-y,t) \right).
\end{equation*}
See \cite{tranbook}.
For $g\equiv 0$, the expression of $\ol u$ becomes
\begin{equation*}
    \ol u(x,t) = -t\sup_{y\in \R^d} \left(0\cdot \frac{x-y}{t} - \ol L\bigl(\frac{x-y}{t}\bigr)\right) = -t\ol H(0).
\end{equation*}
As a result, quantifying the convergence rate of $t^{-1}u(x,t)$ to $-\ol{H}(0)$ is precisely quantifying the averaged discrepancy between $u$ and $\ol{u}$:
\begin{equation*}
    \frac{u(x,t,\omega)}{t}+\ol H(0) = \frac{u(x,t,\omega)-\ol u(x,t)}{t}.
\end{equation*}
We have seen from \eqref{eq:olHformula} that, to represent $\ol{H}(0)$, the range of $v$ can be restricted to a compact set. This is also the case for the control representation of $u(x,t)$. 

\begin{lem} \label{lem:bound-C-t}
Let $\mathfrak{C}_t$ be the random variable defined by
\begin{equation}
\label{eq:frakC_t}
  \mathfrak{C}_t := \left(4N_1 \fint_0^t \nu_r\,\dd r\right)^{1/q}.
\end{equation}
Then the control representation for the solution to \eqref{eq:HJ} can be rewritten as
\begin{equation}
  \label{eq:ucontrol}
    u(x,t,\omega) = \inf_{|v| \le \mathfrak{C}_t} \left(m(x-tv,0;x,t,\omega)\right).
\end{equation}
\end{lem}
\begin{proof} Our goal is to show
\begin{equation*}
    u(x,t,\omega) =\inf_{y\in \R^d} m(y,0;x,t,\omega)= \inf_{|x-y|/t\le \mathfrak{C}_t} m(y,0;x,t,\omega).
\end{equation*}
This follows from the aforementioned control formula for $u$ and the basic estimates \eqref{eq:m-upper-bd} and \eqref{eq:m-lower-bd} of $m$.
\end{proof}

Similarly, by \eqref{eq:bound-L-bar},  the expression of $\ol u(x,t) = -t\ol{H}(0)$ can be further replaced by
\begin{equation}
\label{eq:olHcontrol}
    \ol{H}(0) = \sup_{v\in \R^d} \{-\ol m(v,1)\} = -\frac{1}{t} \inf_{|v| \le \tilde C_0} \ol{m}(tv,t).
\end{equation}
Note that we can enlarge $\frakC_t$ in \eqref{eq:frakC_t} and enlarge $\tilde C_0$ in \eqref{eq:olHcontrol}, and the two representation formulas still hold. 

\medskip

The following estimate for the random variable defined in \eqref{eq:frakC_t} will be useful.
\begin{lem} \label{lem:A1-A2}
Assume $\E[\exp(c\int_0^1 \nu_r\,\dd r)] < \infty$ for some $c>0$. There exist some constants $C>0$ such that 
for any $K>C$ and $t\ge a>0$, 
\begin{equation}
\label{eq:calA_tail}
\bP\left(\sup_{w\in [0,t-a]}\fint_w^t \nu_r \,\dd r > K\right) \le C\exp\bigl(- \frac{cKa}{6} \bigr).
\end{equation}
\end{lem}
This is essentially Chernoff's bound for variables with a finite-range of dependence. The proof, which is standard, can be found in Appendix~\ref{asec:proof-chernoff}.

As a consequence of Theorems \ref{thm:sp-reg-t}--\ref{thm:time-reg} and Lemma \ref{lem:bound-C-t}, we immediately have the following regularity estimates, which can be of independent interest.

\begin{thm}\label{thm:reg-u-nearly-Lip}
    Let $u$ be the solution to \eqref{eq:HJ}.
    Then, for any $x,y\in\R^d$ and $0<s<t$, we have
  \begin{equation*}
  \begin{aligned}
  |u(x,t)-u(y,t)|\lesssim
|x-y|\varphi\bigl(\frac{t}{|x-y|}\bigr) &\left( \frac{|x-y|^q}{t^q}+\sup_{w\leq t-|x-y|} \fint_{w\vee 0}^t \nu_r\,\dd r \right.\\
&\qquad \qquad \left.+ \sup_{w\geq |x-y|} \fint_0^{w\wedge t} \nu_r\,\dd r
\right),
\end{aligned}
\end{equation*}
and
\[
-\int_s^t\nu_r\,\dd r\le u(x,s)-u(x,t)\lesssim (t-s)\varphi\bigl(\frac{t}{t-s}\bigr) \bigl(\frac{t}{s}\bigr)^{q-1} \sup_{w\leq s} \fint_{w \vee 0}^t \nu_r\,\dd r.
\]
\end{thm}

Now we are ready to prove Theorem \ref{thm:main-1}. 

\begin{proof}[Proof of Theorem \ref{thm:main-1}]
  Denote by $B^\dsc_R=\{x\in\Z^d:|x|<R\}$ the discrete ball with $\ell^2$-radius $R$ and let $B^\dsc_R(x)=x+B_R^\dsc$. We write
  \[
 \widetilde{\frakC}_t:=\frakC_t + \tilde C_0. 
  \]
  
 \noindent {\bf Step 1.} (Pointwise upper bound) First, we show that \begin{equation}\label{eq:lowb-error1}
 \mb  P(u(x,t) + t\ol{H}(0)\ge \lambda\psi(t)t^{1/2})
  \le C\exp (-c\lambda^{2/3})
\end{equation}
  for any $\lambda>0$, $t\ge C$, and $x\in\R^d$.
 
 By \eqref{eq:ucontrol}, \eqref{eq:olHformula}, and the homogeneity of $\ol{m}$,
  \begin{align*}
u(x,t) + t\ol{H}(0) &= \inf_{|v|\le \widetilde\frakC_t} m(x-tv,0;x,t) - \inf_{|v|\le \widetilde\frakC_t} \ol{m}(tv,t). 
\end{align*}
We let $v_0\in \R^d$ denote a deterministic point with $|v_0|\le\tilde C_0$ such that $\ol m(v_0,1)=\inf_{v\in\R^d}\ol m(v,1)$.
Hence, 
\begin{align*}
u(0,t) + t\ol{H}(0)
&\le m(-tv_0,0;0,t)-\ol m(tv_0,t)\\
&\le m(-tv_0,0;0,t)-\E m(tv_0,t)+Ct^{1/2}\psi(t).
\end{align*}
where we used \eqref{eq:Em-mbar} in the last inequality. 
Thus, using Theorem~\ref{thm:fluctuation-mxt}\eqref{item:exp-integ-nu}, for any $\lambda\ge 2C$, $t\ge C$,
\begin{align*}
&\mb P(u(0,t) + t\ol{H}(0)\ge \lambda\psi(t)t^{1/2})\\
&\le\mb P(m(-tv_0,0;0,t)-\E m(tv_0,t)\ge (\lambda-C)\psi(t)t^{1/2})\\
&\le \mb P(m(tv_0,t)-\E m(tv_0,t)\ge \tfrac{\lambda}{2}\psi(t)t^{1/2})\\
&\le C\exp (-c\lambda^{2/3}).
\end{align*} 
Estimate \eqref{eq:lowb-error1} follows by the stationarity of the environment.

\medskip 
 \noindent  {\bf Step 2.} (Pointwise lower bound) Recall that $
B_R^\dsc=B_R\cap\Z^d$. We will show that 
\begin{align}\label{eq:discretize-cfrak}
 \MoveEqLeft u(x,t) + t\ol{H}(0)\\
 &\ge 
  -\sup_{v\in B_{t\tilde\frakC_t}^\dsc}\Abs{m(x-v,0;x,t)-\E m(v,t)}-C\varphi(t)\sup_{w\ge 1}\avint_0^{w\wedge t}\nu_r\dd r
  \nonumber
  \end{align}
for any $x\in \R^d$. Indeed,
\begin{align*}
u(x,t) + t\ol{H}(0)&=\inf_{|v|\le \widetilde\frakC_t} m(x-tv,0;x,t) - \inf_{|v|\le \widetilde\frakC_t}\ol m(tv,t)\\
&\ge \inf_{|v|\le \widetilde\frakC_t} m(x-tv,0;x,t) - \inf_{|v|\le \widetilde\frakC_t}\E m(tv,t)\\
&\ge -\sup_{|v|\le t\tilde\frakC_t}\Abs{m(x-v,0;x,t)-\E m(v,t)}
\end{align*}
Let us discretize $B_{t\widetilde\frakC_t}$. Note that for any $v\in B_{t\tilde\frakC_t}$, we can find $z=z(v)\in B_{t\tilde\frakC_t}^\dsc$ with $1\le \|v-z\|_\infty\le 2$, where $\|\cdot\|_\infty$ denotes the maximum distance in $\bbR^d$. Then, by Theorem~\ref{thm:sp-reg-t} and by reversing the time as done in \eqref{revrese}, 
\begin{align*}
\MoveEqLeft\abs{m(x-v,0;x,t)-m(x-z,0;x,t)}\\
&\lesssim \varphi(t)\left(\frakC_t^q+\sup_{w\ge 1}\avint_0^{w\wedge t}\nu_r\dd r\right)\lesssim \varphi(t)\sup_{w\ge 1}\avint_0^{w\wedge t}\nu_r\dd r.
\end{align*}
This bound holds uniformly for all $v\in t\tilde\frakC_t$ and $x\in\R^d$. Taking expectations yields
\[
  \abs{\E m(x-v,0;x,t)-\E m(x-z,0;x,t)}\lesssim \varphi(t).
\]
Display \eqref{eq:discretize-cfrak} follows.

\medskip 
 \noindent{\bf Step 3.} 
We will prove that, for $x\in\R^d$, $\lambda\ge 1, t\ge C$,
\begin{equation}\label{eq:upper-tail}
  \mb P(u(x,t) + t\ol{H}(0)\le -\lambda t^{1/2}\psi(t))\le C\exp(-c\lambda^{2/5}).
\end{equation}
Set $\theta=\frac{2}{5q}$. For any $x\in\R^d$, $\lambda>C, t>C$, 
\begin{align}\label{eq:dsc-sup-m-em}
&\mb P\Bigl(\sup_{v\in B^\dsc_{t\tilde\frakC_t}}\Abs{m(x-v,0;x,t)-\E m(v,t)}	\ge \lambda t^{1/2}\psi(t)
\Bigr)\nonumber\\
&=\mb P\Bigl(\sup_{v\in B^\dsc_{t\tilde\frakC_t}}\Abs{m(v,t)-\E m(v,t)}	\ge \lambda t^{1/2}\psi(t)
\Bigr)\nonumber\\
&\le \mb P\Bigl(\sup_{v\in B^\dsc_{t\lambda^\theta \varphi(t)}}\Abs{m(v,t)-\E m(v,t)}	\ge \lambda t^{1/2}\psi(t)
\Bigr)+\mb P(\tilde\frakC_t\ge\lambda^\theta \varphi(t))\nonumber\\
&\le \sum_{v\in B^\dsc_{t\lambda^\theta \varphi(t)}}\mb P\Bigl(\abs{m(v,t)-\E m(v,t)}	\ge \lambda t^{1/2}\psi(t)
\Bigr)+\exp(-c\lambda^{\theta q}\varphi^q(t))
\end{align}
where in the first equality we used the (spatial-)translation invariance of the environment measure. 
Noting that $(1-\theta q)2/3=\theta q=2/5$, we get
\begin{align*}
\MoveEqLeft \sum_{v\in B^\dsc_{t\lambda^\theta \varphi(t)}}\mb P\Bigl(\abs{m(v,t)-\E m(v,t)}	\ge \lambda t^{1/2}\psi(t)
\Bigr)\\
&\le \sum_{v\in B^\dsc_{t\lambda^\theta \varphi(t)}}\mb P\Bigl(\abs{m(v,t)-\E m(v,t)}	\ge \lambda^{1-\theta q} t^{1/2}\varphi^4(t)(\tfrac{|v|^q}{t^q}+\log t)
\Bigr)\\
&\stackrel{Theorem~\ref{thm:fluctuation-mxt}}\le (t\lambda^\theta \varphi(t))^d\exp (-c\lambda^{(1-\theta q)2/3}\varphi^2(t))\lesssim \exp(-c\lambda^{2/5}\varphi(t)).
\end{align*}
The above inequality, together with \eqref{eq:dsc-sup-m-em}, yields
\begin{align*}
&\mb P\Bigl(\sup_{v\in B^\dsc_{t\tilde\frakC_t}}\abs{m(x-v,0;x,t)-\E m(v,t)}	\ge \lambda t^{1/2}\psi(t)
\Bigr)\\
&\lesssim \exp(-c\lambda^{2/5}\varphi(t))+\exp(-c\lambda^{\theta q}\varphi^q(t))
\asymp\exp(-c\lambda^{2/5}\varphi(t)).
\end{align*}
This inequality and \eqref{eq:discretize-cfrak} imply, for any $x\in\R^d$, $\lambda>C, t>C$,
\begin{align*}
\MoveEqLeft\mb P(u(x,t) + t\ol{H}(0)\le -\lambda t^{1/2}\psi(t))\\
&\le\mb P\Bigl(\sup_{v\in B^\dsc_{t\tilde\frakC_t}}\abs{m(x-v,0;x,t)-\E m(v,t)}	\ge \frac{1}{2}\lambda t^{1/2}\psi(t)
\Bigr)\\
&\qquad+\mb P\left(C\varphi(t)\sup_{w\ge 1}\avint_0^{w\wedge t}\nu_r\dd r\ge \frac{1}{2}\lambda t^{1/2}\psi(t)\right)\\
&\lesssim 
\exp(-c\lambda^{2/5}\varphi(t)).
\end{align*}
We have proved \eqref{eq:upper-tail}.

\medskip 
 \noindent{\bf Step 4.} (Discretization of $B_R$)  
Since 
for any $x\in B_R$, we can find $y=y(x)\in B_R^\dsc$ with $1\le |x-y|\le 2$,
by Theorem~\ref{thm:reg-u-nearly-Lip}, we get, for all such $x,y\in\R^d$,
\begin{align*}
 |u(x,t)-u(y,t)|\lesssim \frac{\varphi(t)}{t^q}+\varphi(t)\left(\sup_{w\leq t-1} \fint_{w\vee 0}^t \nu_r\,\dd r + \sup_{w\geq 1} \fint_0^{w\wedge t} \nu_r\,\dd r\right).
\end{align*}
Hence we have
\begin{align}\label{eq:discretize-BR}
&\Abs{\sup_{x\in B_R}\Bigl(u(x,t) + t\ol{H}(0)\Bigr)-\sup_{x\in B_R^\dsc}\Bigl(u(x,t)+ t\ol{H}(0)\Bigr)
}\nonumber\\
&\le \frac{C\varphi(t)}{t^q}+C\varphi(t)\left(\sup_{w\leq t-1} \fint_{w\vee 0}^t \nu_r\,\dd r + \sup_{w\geq 1} \fint_0^{w\wedge t} \nu_r\,\dd r\right)
\end{align}
and \eqref{eq:discretize-BR} still holds when both $\sup$'s on the left side are replaced by $\inf$'s.

\medskip 
 \noindent{\bf Step 5.}  By \eqref{eq:discretize-BR} and \eqref{eq:lowb-error1}, for any $\lambda\ge 1, t\ge C, R\ge 2$,
\begin{align*}
\MoveEqLeft\mb P\left(\sup_{x\in B_R}\bigl(u(x,t)+ t\ol{H}(0)\bigr)\ge  (\log R)^{3/2}\lambda t^{1/2}\psi(t)\right)\\
&\le \mb P\left(\sup_{x\in B_R^\dsc}\bigl(u(x,t)+ t\ol{H}(0)\bigr)\ge  \tfrac13(\log R)^{3/2}\lambda t^{1/2}\psi(t)\right)\\
&\qquad +\mb P\left(C\varphi(t)\Bigl(\sup_{w\leq t-1} \fint_{w\vee 0}^t \nu_r\,\dd r + \sup_{w\geq 1} \fint_0^{w\wedge t} \nu_r\,\dd r\Bigr) \ge \tfrac13(\log R)^{3/2}\lambda t^{1/2}\psi(t)
\right)\\
&\stackrel{\eqref{eq:calA_tail}}\lesssim \sum_{x\in B_R^\dsc}\mb P\left(u(x,t)+ t\ol{H}(0)\ge  \tfrac13(\log R)^{3/2}\lambda t^{1/2}\psi(t)\right)+\exp(-c\lambda t^{1/2})\\
&\lesssim R^d\exp (-c(\log R)\lambda^{2/3})\lesssim \exp(-c\lambda^{2/3}).
\end{align*}
This tells us that, for any $t\ge C, R\ge 2$, almost surely, 
\[
  \sup_{x\in B_R}\bigl(u(x,t)+ t\ol{H}(0)\bigr)\le (\log R)^{3/2}t^{1/2}\psi(t)\mathcal X_1
\]
 for some random variable $\mathcal X_1$ with $\E[\exp(c\mathcal X_1^{2/3})]\le C$.  
 
\medskip 
 \noindent {\bf Step 6.} Similarly, applying \eqref{eq:discretize-BR} to infimums,
 \begin{align*}
\MoveEqLeft\mb P\left(-\inf_{x\in B_R}\bigl(u(x,t)+ t\ol{H}(0)\bigr)\ge  (\log R)^{5/2}\lambda t^{1/2}\psi(t)\right)\\
&\le \mb P\left(-\inf_{x\in B_R^\dsc}\bigl(u(x,t)+ t\ol{H}(0)\bigr)\ge  \tfrac13(\log R)^{5/2}\lambda t^{1/2}\psi(t)\right)\\
&\qquad +\mb P\left(C\varphi(t)\Bigl(\sup_{w\leq t-1} \fint_{w\vee 0}^t \nu_r\,\dd r + \sup_{w\geq 1} \fint_0^{w\wedge t} \nu_r\,\dd r \Bigr)\ge \tfrac13(\log R)^{5/2}\lambda t^{1/2}\psi(t)
\right)\\
&\stackrel{\eqref{eq:calA_tail}}\lesssim \sum_{x\in B_R^\dsc}\mb P\left(\abs{u(x,t)+ t\ol{H}(0)}\ge  \tfrac13(\log R)^{5/2}\lambda t^{1/2}\psi(t)\right)+\exp(-c\lambda t^{1/2})\\
&\stackrel{\eqref{eq:upper-tail}}\lesssim R^d\exp (-c(\log R)\lambda^{2/5})\lesssim \exp(-c\lambda^{2/5}).
\end{align*}
This tells us that, for any $t\ge C, R\ge 2$, almost surely,
\[
  \inf_{x\in B_R}\bigl(u(x,t)+ t\ol{H}(0)\bigr)\ge -(\log R)^{5/2}t^{1/2}\psi(t)\mathcal X_2
\]
 for some random variable $\mathcal X_2$ with $\E[\exp(c\mathcal X_2^{2/5})]\le C$.
\end{proof}


\subsection{Proof of Theorem \ref{thm:main-2}}

Using the optimal control formula (see \cite{tranbook}) and the fact that $|g(y)-g(0)|\leq C|y|$, and by the same argument as in the proof of Lemma \ref{lem:bound-C-t}, we imply that
\begin{align*}
  u^\ep(x,t)&=\inf_{y\in\R^d}\left(g(y)+\varepsilon m(\frac{y}{\varepsilon},0;\frac{x}{\ep},\frac t\ep) \right)\\
  &=\inf_{|y|\leq (\frakC_{t/\ep} + C)t} \left( g(x-y) + \ep m\bigl(\frac{x-y}{\ep},0;\frac x\ep,\frac{t}{\ep}\bigr)\right).
\end{align*}
The optimal control formula for $u$ also gives
\begin{align*}
    \ol u(x,t)
    &=\inf_{y\in \R^d} \left( g(y) + \ol m(x-y,t)\right)\\
    &=\inf_{|y|\leq \tilde C_0 t} \left( g(x-y) + \ol m(y,t)\right).
\end{align*}
Note that in this proof the generic constants $C$ also depends on $\norm{g}_{\rm Lip}$. We write 
  \[
 \widetilde{\frakC}_t:=\frakC_t + C+ \tilde C_0. 
  \]

Using those representations, and as a consequence of Theorems \ref{thm:sp-reg-t}--\ref{thm:time-reg}, we have the following regularity estimates, analogous to Theorem \ref{thm:reg-u-nearly-Lip}, for the solutions to the oscillatory equations \eqref{eq:HJ-ep}. 

\begin{thm}\label{thm:reg-uep-nearly-Lip}
    For each $\eps \in (0,1)$, let $u^\eps$ be the solution to \eqref{eq:HJ-ep}.
    Then, for any $x,y\in\R^d$ and $0<s<t$, we have
  \begin{equation*}
  \begin{aligned}
  |u^\eps(x,t)-u^\eps(y,t)| &\lesssim [g]_{\rm Lip}|x-y| + 
|x-y|\varphi\bigl(\frac{t}{|x-y|}\bigr)\times \\
&\left(\frac{|x-y|^q}{t^q} +\sup_{w\leq t-|x-y|} \fint_{(w\vee 0)/\ep}^{t/\ep} \nu_r\,\dd r + +\sup_{w\geq |x-y|} \fint_0^{(w\wedge t)/\ep} \nu_r\,\dd r\right)\\
\end{aligned}
\end{equation*}
and
\[
-\int_s^t\nu_r\,\dd r\le u^\eps(x,s)-u^\eps(x,t)\lesssim (t-s)\varphi\bigl(\frac{t}{t-s}\bigr) \bigl(\frac{t}{s}\bigr)^{q-1} \sup_{w\leq s} \fint_{(w \vee 0)/\eps}^{t/\eps} \nu_r\,\dd r.
\]
\end{thm}

Now we are ready to prove Theorem \ref{thm:main-2}. 

\begin{proof}[Proof of Theorem \ref{thm:main-2}]
We divide the proof into several steps.

\medskip 
 \noindent {\bf Step 1.} (Basic bounds)
Thanks to \eqref{eq:H bar superlinear}, $g(x)-\bar Ct$ is a subsolution to \eqref{eq:effHJ}, and $g(x)+\bar Ct$ is a supersolution to \eqref{eq:effHJ}, respectively,  for $\bar C$ sufficiently large depending on $c,d,q,C_0, N_1, [g]_{\rm Lip}$.
 By the usual comparison principle, for $(x,t)\in \R^d\times [0,\infty)$,
 \[
 g(x)-\bar Ct \leq\ol u(x,t) \leq g(x)+\bar Ct.
 \]

 Similarly, in light of (A2), we have that $g(x)-\bar Ct - \ep\int_0^{t/\ep} \nu_0(s)\,\dd s$ is a subsolution to \eqref{eq:HJ-ep}, and $g(x)+\bar Ct+\ep\int_0^{t/\ep} \nu_0(s)\,\dd s$ is a supersolution to \eqref{eq:HJ-ep}, respectively.
 By the usual comparison principle, for $(x,t)\in \R^d\times [0,\infty)$,
 \[
 g(x)-\bar Ct  - \ep\int_0^{t/\ep} \nu_0(s)\,\dd s\leq u^\ep(x,t) \leq g(x)+\bar Ct + \ep\int_0^{t/\ep} \nu_0(s)\,\dd s.
 \]
Hence, for $(x,t)\in \R^d\times [0,C\ep]$,
 \[
 |u^\ep(x,t)-\ol u(x,t)|\leq 2\bar C t +\ep\int_0^{t/\ep} \nu_0(s)\,\dd s \leq 2\bar C C \ep+\ep\int_0^{C} \nu_0(s)\,\dd s.
 \]
 We are done with $t\in [0,C\ep]$, and we only need to consider the case that $t>C\ep$ below.
 
\medskip 
 \noindent {\bf Step 2.} (Pointwise upper bound)
First, we show that
\begin{equation}\label{eq:lowb-homog}
   \mb  P\Bigl(u^\ep(x,t)-\ol u(x,t)\ge \lambda(\varepsilon t)^{1/2}\psi(\frac{t}{\ep})\Bigr)
  \le C\exp (-c\lambda^{2/3})
\end{equation}
  for any $\ep\in(0,1),\lambda>0$, $t\ge C\ep$, and $x\in\R^d$. It suffices to consider $\lambda>C$.

Since there exists a deterministic $y_{x,t}\in \R^d$ with $|y_{x,t}|\le\tilde C_0 t$ such that 
\[
  \ol u(x,t)=g(x-y_{x,t})+\ol m(y_{x,t},t),
\]
we have
\begin{align*}
u^\ep(x,t)-\ol u(x,t)&\le g(x-y_{x,t}) + \ep m(\frac{x-y_{x,t}}\ep,0;\frac x\ep,\frac t\ep)-[g(x-y_{x,t})+\ol m(y_{x,t},t)]\\
&\stackrel{\eqref{eq:Em-mbar}}\le \ep\left[m(\frac{x-y_{x,t}}\ep,0;\frac x\ep,\frac t\ep)-\E m(\frac{x-y_{x,t}}\ep,0;\frac x\ep,\frac t\ep)\right]+\ep C(\frac{t}{\varepsilon})^{1/2}\psi(\frac{t}{\varepsilon}).
\end{align*}
Thus 
\begin{align*}
\MoveEqLeft\mb P(u^\ep(x,t)-\ol u(x,t)\ge \lambda(\varepsilon t)^{1/2}\psi(\frac{t}{\ep}))\\
&\le \mb P\left(m(\frac{y_{x,t}}\ep,\frac t\ep)-\E m(\frac{y_{x,t}}\ep,\frac t\ep)\ge c \lambda(\frac t\ep)^{1/2}\psi(\frac{t}{\ep})\right).
\end{align*}
Display \eqref{eq:lowb-homog} then follows from Theorem~\ref{thm:fluctuation-mxt}\eqref{item:exp-integ-nu}.

\medskip 
 \noindent{\bf Step 3.} (Pointwise lower bound) Recall that $
B_R^\dsc=B_R\cap\Z^d$. We will prove 
\begin{align}
\label{eq:discretize-upb-homog}
\MoveEqLeft u^\ep(x,t)-\ol u(x,t)\\
&\ge -\ep \sup_{y\in B^\dsc_{t\tilde\frakC_t}}\Abs{m\bigl(\frac{x-y}{\ep},0;\frac x\ep,\frac{t}{\ep}\bigr)-\E m(\frac y\ep,\frac t\ep)}-C\varphi(t)\sup_{w\ge 1}\avint_0^{(w\wedge t)/\ep}\nu_r\dd r\nonumber
\end{align}
for any $x\in \R^d, {t\ge  C\ep}, \ep\in(0,1)$. Indeed, by the representation formulas,
\begin{align*}
u^\ep(x,t)-\ol u(x,t)&\ge
\inf_{|y|\le t\tilde\frakC_t}\ep m\bigl(\frac{x-y}{\ep},0;\frac x\ep,\frac{t}{\ep}\bigr)-\ol m(y,t)\\
&\stackrel{\eqref{eq:Em-mbar}}\ge  \ep \inf_{|y|\le t\tilde\frakC_t}\left(m\bigl(\frac{x-y}{\ep},0;\frac x\ep,\frac{t}{\ep}\bigr)-\E m(\frac y\ep,\frac t\ep)\right).
\end{align*}
Let us discretize $B_{t\widetilde\frakC_t}$. Note that for any $y\in B_{t\tilde\frakC_t}$, we can find $z=z(y)\in B_{t\tilde\frakC_t}^\dsc$ with $1\le |y-z|\le 2$. Then, by Theorem~\ref{thm:sp-reg-t}, \begin{align*}
\MoveEqLeft\ep\Abs{m\bigl(\frac{x-y}{\ep},0;\frac x\ep,\frac{t}{\ep}\bigr)-m\bigl(\frac{x-z}{\ep},0;\frac x\ep,\frac{t}{\ep}\bigr)}\\
&\lesssim \varphi(t)\left(\frakC_t^q+\sup_{w\ge 1}\avint_0^{(w\wedge t)/\ep}\nu_r\dd r\right)\lesssim \varphi(t)\sup_{w\ge 1}\avint_0^{(w\wedge t)/\ep}\nu_r\dd r.
\end{align*}
Taking expectations yields
\[
\ep\Abs{\E m\bigl(\frac{x-y}{\ep},0;\frac x\ep,\frac{t}{\ep}\bigr)-\E m\bigl(\frac{x-z}{\ep},0;\frac x\ep,\frac{t}{\ep}\bigr)}
 \lesssim \varphi(t).
\]
Recalling that $\nu\ge 1$, display \eqref{eq:discretize-upb-homog} follows.

\medskip 
 \noindent{\bf Step 4.} We will prove that, for $x\in\R^d$, $\lambda\ge 1, t\ge C\ep, \ep\in(0,1)$,
\begin{equation}\label{eq:upper-tail-homog}
  \mb P(u^\ep(x,t)-\ol u(x,t)\le -\lambda(\varepsilon t)^{1/2}\psi(\tfrac{t}{\ep}))\le C\exp(-c\lambda^{2/5}).
\end{equation}
Set $\theta=\frac{2}{5q}$. Similar to \eqref{eq:dsc-sup-m-em}, we get
\begin{align*}
\MoveEqLeft  \mb P\left(\ep \sup_{y\in B^\dsc_{t\tilde\frakC_t}}\Abs{m\bigl(\frac{x-y}{\ep},0;\frac x\ep,\frac{t}{\ep}\bigr)-\E m(\frac y\ep,\frac t\ep)}\ge \lambda(\varepsilon t)^{1/2}\psi(\frac{t}{\ep})\right)\\
&\le \mb P\left(\sup_{y\in B^\dsc_{t\lambda^\theta \varphi(t/\ep)}}\Abs{m\bigl(\tfrac y\ep,\tfrac{t}{\ep}\bigr)-\E m(\tfrac y\ep,\tfrac t\ep)}\ge \lambda(\tfrac{t}{\ep})^{1/2}\psi(\tfrac{t}{\ep})\right)+\mb P(\tilde\frakC_t\ge\lambda^\theta \varphi(\tfrac t\ep))\\
&\le\sum_{y\in B^\dsc_{t\lambda^\theta \varphi(t/\ep)}}\mb P\left(\Abs{m\bigl(\tfrac y\ep,\tfrac{t}{\ep}\bigr)-\E m(\tfrac y\ep,\tfrac t\ep)}\ge \lambda(\tfrac{t}{\ep})^{1/2}\psi(\tfrac{t}{\ep})\right)+\exp(-c\lambda^{\theta q}\varphi^q(\tfrac t\ep))
\end{align*}
Furthermore, for each $y\in B_{t\lambda^\theta \varphi(t/\ep)}$, when $t\ge \ep C$,
\begin{align*}
\MoveEqLeft\mb P\left(\Abs{m\bigl(\tfrac y\ep,\tfrac{t}{\ep}\bigr)-\E m(\tfrac y\ep,\tfrac t\ep)}\ge \lambda(\tfrac{t}{\ep})^{1/2}\psi(\tfrac{t}{\ep})\right)\\
&\le \mb P\left(\Abs{m\bigl(\tfrac y\ep,\tfrac{t}{\ep}\bigr)-\E m(\tfrac y\ep,\tfrac t\ep)}
\ge c\lambda^{1-\theta q}(\tfrac{t}{\ep})^{1/2}\varphi^4(\tfrac{t}{\ep})(\tfrac{|y|^q}{t^q}+1)\right)\\
&\stackrel{Theorem~\ref{thm:fluctuation-mxt}}\le 
\exp (-c\lambda^{(1-\theta q)2/3}\varphi^2(\tfrac t\ep))\asymp \exp(-c\lambda^{2/5}\varphi^2(\tfrac t\ep)).
\end{align*}
Hence,  for $x\in\R^d$, $\lambda>0, t\ge C\ep, \ep\in(0,1)$, 
\begin{align*}
\MoveEqLeft  \mb P\left(\ep \sup_{y\in B^\dsc_{t\tilde\frakC_t}}\Abs{m\bigl(\frac{x-y}{\ep},0;\frac x\ep,\frac{t}{\ep}\bigr)-\E m(\frac y\ep,\frac t\ep)}\ge \tfrac12\lambda(\varepsilon t)^{1/2}\psi(\frac{t}{\ep})\right)\\
&\lesssim (t\lambda^\theta \varphi(\tfrac t\ep))^d \exp(-c\lambda^{2/5}\varphi^2(\tfrac t\ep))+\exp(-c\lambda^{\theta q}\varphi^q(\tfrac t\ep))
\lesssim \exp(-c\lambda^{2/5}\varphi(\tfrac t\ep)).
\end{align*}

This inequality, together with \eqref{eq:discretize-upb-homog}, yield
\begin{align*}
\MoveEqLeft\mb P\left(u^\ep(x,t)-\ol u(x,t)\le -\lambda(\varepsilon t)^{1/2}\psi(\tfrac{t}{\ep})\right)\\
&\lesssim \exp(-c\lambda^{2/5}\varphi(\tfrac t\ep))+\mb P\left(C\varphi(t)\sup_{w\ge 1}\avint_0^{(w\wedge t)/\ep}\nu_r\dd r\ge \tfrac12\lambda(\varepsilon t)^{1/2}\psi(\frac{t}{\ep})\right)\\
&\stackrel{\eqref{eq:calA_tail}}\lesssim 
\exp(-c\lambda^{2/5}\varphi(\tfrac t\ep))+\exp(-c\lambda(\tfrac t\ep)^{1/2}).
\end{align*}
Inequality \eqref{eq:upper-tail-homog} follows.

\medskip 
 \noindent{\bf Step 5.} (Discretization of $B_R$)
Since 
for any $x\in B_R$, we can find $y=y(x)\in B_R^\dsc$ with $1\le \|x-y\|_\infty\le 2$,
by Theorem~\ref{thm:reg-uep-nearly-Lip} and noting that $\varphi(t/|x-y|) \le \varphi(t)$ and the boundedness of $|x-y|$,
 we get
\begin{align*}
 |u^\ep(x,t)-u^\ep(y,t)|
 \lesssim \frac{\varphi(t)}{t^q}+\varphi(t)\left(\sup_{w\leq t-1} \fint_{(w\vee 0)/\ep}^{t/\ep} \nu_r\,\dd r + \sup_{w\geq 1} \fint_0^{(w\wedge t)/\ep} \nu_r\,\dd r\right).
\end{align*}
Hence we have
\begin{align}\label{eq:discretize-BR-homog}
&\Abs{\sup_{x\in B_R}\Bigl(u^\ep(x,t)-\ol u(x,t)\Bigr)-\sup_{x\in B_R^\dsc}\Bigl(u^\ep(x,t)-\ol u(x,t)\Bigr)
}\nonumber\\
&\le C\varphi(t)\left(\frac{1}{t^q}+ \sup_{w\leq t-1} \fint_{(w\vee 0)/\ep}^{t/\ep} \nu_r\,\dd r + \sup_{w\geq 1} \fint_0^{(w\wedge t)/\ep} \nu_r\,\dd r\right)
\end{align}
and \eqref{eq:discretize-BR-homog} still holds when both $\sup$'s on the left side are replaced by $\inf$'s.

Finally, with \eqref{eq:lowb-homog}, \eqref{eq:upper-tail-homog}, and \eqref{eq:discretize-BR-homog}, Theorem \ref{thm:main-2} follows verbatim from the arguments as in Steps 5 and 6 in the proof of Theorem \ref{thm:main-1}.
\end{proof}

\appendix
\section{Some basic estimates}\label{appendix:basic estimates}
\begin{lem}\label{lem:Gamma function}
Let $q>1$ and $q'=q/(q-1)$.
For $n\in\N$, we have
\[
l_n=\frac{(q-1)(1+q')\cdots(n+q')}{n!} \leq (q-1)e^{q'}n^{q'}.
\]
\end{lem}
\begin{proof}
    \begin{align*}
    l_n&=(q-1)\prod_{i=1}^n\bigl(1+\frac{q'}{i}\bigr)\\
    &\le(q-1)\exp\left(\sum_{i=1}^n\frac{q'}{i}\right)\\
    &\le (q-1)\exp(q'\log n+q')=(q-1)e^{q'}n^{q'}.
\end{align*}
\end{proof}

\begin{lem}\label{lem:b_n}
    Let $a_0=1/q'$, and for $n\in \N$,
\[
  a_n=f^{\circ n}(a_0).
\]
Then, for $n\in \N$,
\begin{align*}
  b_n=1-a_n=\frac{1}{n(q-1)+q}.
\end{align*}
\end{lem}
\begin{proof}
    Recall that

\[
f(a) = \frac{a}{(1-a)q(q-1)+q} + \frac{1}{q}.
\]
Hence
\[
1 - f(a)
= \frac{1-a}{(1-a)(q-1)+q} .
\]
Therefore,
\begin{align*}
b_{n+1}
&= \frac{b_n }{b_n (q-1)+1}.
\end{align*}
Taking reciprocals on both sides,
\[
\frac{1}{b_{n+1}}
=(q-1)+\frac{1}{b_n},
\]
which implies
\[
\frac{1}{b_n}=n(q-1)+\frac{1}{b_0}=n(q-1)+q.
\]
The lemma follows.
\end{proof}

\section{Concentration inequalities}\label{appendix:concentration inequalities}
We use the following martingale inequalities in the paper.
\begin{lem}[Burkholder's inequality] {\cite[Theorem 11.2.1]{Chow&Teicher-book}}
Let $(X_n)_{n\ge 0}$ be an $L^1$-martingale with $X_0=0$. For $p>1$, 
there exist universal constants $A_p, B_p>0$ depending only on $p$ such that
for every $n\ge 1$,
\[
A_p \,\|S_n(X)\|_{p} \;\le\; \|X_n\|_{p} \;\le\; B_p \,\|S_n(X)\|_{p},
\]
\[
A_p \,\|S(X)\|_{p} \;\le\; \|X\|_{p} \;\le\; B_p \,\|S(X)\|_{p},
\]
where
\[
S_n(X) := \Bigl(\sum_{j=1}^n (X_j - X_{j-1})^2\Bigr)^{1/2},
\qquad
S(X) := \sup_{n\ge 1} S_n(X),
\]
\[
\|X_n\|_p := \bigl(\E|X_n|^p\bigr)^{1/p},
\qquad
\|X\|_p := \sup_{n\ge 1}\|X_n\|_p.
\]
\end{lem}


\begin{lem}[Azuma's inequality]\label{L.3.5} 
{\cite{azuma1967weighted}} Let $\left\{X_{n}\right\}_{n\geq 0}$  be a martingale. If for each $n\geq 0$ there is a constant $c_n\ge 0$ such that $|X_{n+1}-X_{n}|\leq c_n$ almost surely, 
then for all $\lambda>0$ and $n\geq0$,
\[
 \bP(|X_{n}-X_{0}|\geq \lambda)\leq 2\exp \left(-{\frac {\lambda^{2}}{2\sum _{k=0}^{n-1}c_{k}^{2}}}\right).
 \]
\end{lem}

\begin{lem}\label{thm:LV32}
\cite[Theorem 3.2]{lesigne-volny-01} and \cite[Theorem 2.1]{fan2012large}
Let $\{X_n\}_{n\geq 0}$ be a martingale satisfying $\sup_n \mathbb{E} \exp \{ |X_{n+1}-X_n| \} \le C_1$ for some constant $C_1>0$. Then, for all $\lambda > 0$ and $n\geq 1$,
\begin{equation*}
\mathbb{P} \left( \max_{1 \le k \le n} |X_n -X_0| \ge \lambda n \right) \le 2C( \lambda) \exp \left\{ - \left( \frac{\lambda}{4} \right)^{2/3} n^{1/3} \right\}, 
\end{equation*}
where
\[
C(\lambda): = 2 + 35C_1 \left( \frac{1}{\lambda^{2/3} 16^{2/3}} + \frac{9}{\lambda^2}\right).
\]
\end{lem}


\section{Proof of Lemma~\ref{lem:A1-A2}}\label{asec:proof-chernoff}

\begin{proof}[Proof of Lemma~\ref{lem:A1-A2}] 
First, we will show that for any $K>0$ and $t>0$, 

\begin{equation}
\label{eq:frakC_t_tail}
\bP\left(\fint_0^t \nu_r \,\dd r > K\right) \lesssim \exp\bigl(\bigl(C-\frac12 cK\bigr)t\bigr) .
\end{equation}
Indeed, by Chebyshev's inequality, for any $t>0$ and $K>0$,
	\begin{align*}
\bP\left(\fint_0^t \nu_r \,\dd r > K\right)\le e^{-ctK/2}\E\exp\Bigl(\frac c2\int_0^t\nu_r\dd r\Bigr).
\end{align*}
Further, using the Cauchy-Schwarz inequality and the independence and stationarity of $(\int_j^{j+1}\nu_r\dd r)_{j\ge 0}$, we  have the exponential moment bound
\begin{align*}
2\E\left[\exp\Bigl(\frac c2\int_0^t\nu_r\dd r\Bigr)\right]
&\le \E\left[\exp\Bigl(\sum_{j<t\text{ odd}}c\int_{j}^{j+1} \nu_r\dd r\Bigr)+\exp\Bigl(c\sum_{j<t\text{ even}}\int_{j}^{j+1} \nu_r\dd r\Bigr)\right]\\
&\le 2 \left(\E \Bigl[\exp\bigl({c}\int_{0}^{1} \nu_r \,\dd r\bigr)\Bigr]\right)^{(t+1)/2}
\lesssim \exp(Ct).
\end{align*}
Thus, inequality \eqref{eq:frakC_t_tail} follows.

To prove \eqref{eq:calA_tail}, 
note that it suffices to consider the case when $Ka\ge 1$. For $K\ge C$ large enough, 
by a union bound,
\begin{align*}
\bP\left(\sup_{w\in [0,t-a]}\fint_w^t \nu_r \,\dd r > K\right) &\le 
\sum_{1\le j\le t/a}\bP\left(\fint_{t-aj}^t \nu_r \,\dd r >\frac K2\right)\\
&\stackrel{\eqref{eq:frakC_t_tail}}\lesssim \sum_{j=1}^\infty\exp\bigl( - \frac{cKaj}{6} \bigr)\lesssim \exp\bigl( - \frac{cKa}{6} \bigr).
\end{align*}
Inequality \eqref{eq:calA_tail} is proved.
\end{proof}

\bibliographystyle{abbrv} 
\bibliography{bib} 

\end{document}